\documentclass[10pt]{amsart}

\usepackage{xypic,amscd,amssymb,latexsym}
\usepackage{amsmath}	
\usepackage[breaklinks=true]{hyperref}

\usepackage{mathrsfs}

\usepackage{pifont}

\usepackage{url}

\usepackage{pstricks}
\usepackage{pst-poly}

\usepackage{graphicx}

\usepackage{stmaryrd}

\usepackage{cancel}

\usepackage{caption}
\usepackage{subcaption}

\begin{document}
\pagenumbering{arabic}

\newtheorem{theorem}{Theorem}[section]
\newtheorem{lemma}[theorem]{Lemma}
\newtheorem{proposition}[theorem]{Proposition}
\newtheorem{corollary}[theorem]{Corollary}
\newtheorem{definition}[theorem]{Definition}
\newtheorem{remark}[theorem]{Remark}
\newtheorem{notation}[theorem]{Notation}

\newcommand{\vs}[0]{\vspace{2mm}}

\newcommand{\mcal}[1]{\mathcal{#1}}
\newcommand{\ul}[1]{\underline{#1}}
\newcommand{\ol}[1]{\overline{#1}}
\newcommand{\til}[1]{\widetilde{#1}}
\newcommand{\wh}[1]{\widehat{#1}}

\address{School of Mathematics, Korea Institute for Advanced Study (KIAS), 85 Hoegiro Dongdaemun-gu, Seoul 130-722, Republic of Korea}

\email[H.~Kim]{hkim@kias.re.kr, ~ hyunkyu87@gmail.com}

\author{Hyun Kyu Kim}
\thanks{This work is a substantial re-organization of part of the author's Ph.D. thesis \cite{Ki}.}

\numberwithin{equation}{section}

\title[Central extension of Ptolemy-Thompson group via Kashaev quantization]{The dilogarithmic central extension of the Ptolemy-Thompson group via the Kashaev quantization}

\begin{abstract}
Quantization of universal Teichm\"uller space provides projective representations of the Ptolemy-Thompson group, which is isomorphic to the Thompson group $T$. This yields certain central extensions of $T$ by $\mathbb{Z}$, called dilogarithmic central extensions. We compute a presentation of the dilogarithmic central extension $\widehat{T}^{\rm Kash}$ of $T$ resulting from the Kashaev quantization, and show that it corresponds to $6$ times the Euler class in $H^2(T;\mathbb{Z})$.
Meanwhile, the braided Ptolemy-Thompson groups $T^*$, $T^\sharp$ of Funar-Kapoudjian are extensions of $T$ by the infinite braid group $B_\infty$, and by abelianizing the kernel $B_\infty$ one constructs central extensions $T^*_{\rm ab}$, $T^\sharp_{\rm ab}$ of $T$ by $\mathbb{Z}$, which are of topological nature.
We show $\widehat{T}^{\rm Kash}\cong T^\sharp_{\rm ab}$. Our result is analogous to that of Funar and Sergiescu, who computed a presentation of another dilogarithmic central extension $\widehat{T}^{\rm CF}$ of $T$ resulting from the Chekhov-Fock(-Goncharov) quantization and thus showed that it corresponds to $12$ times the Euler class and that $\wh{T}^{\rm CF} \cong T^*_{\rm ab}$. In addition, we suggest a natural relationship between the two quantizations in the level of projective representations.
\end{abstract}

\maketitle

\tableofcontents

\date{November 2012}

\section{Introduction and overview}
\label{sec:introduction}

Quantum Teichm\"uller theory has appealed to mathematicians and physicists in the recent couple of decades as an approach to quantization of $(2+1)$-gravity theory in physics. The main construction was established mathematically by Kashaev \cite{Kash98} and by Chekhov-Fock \cite{Fo} \cite{FC} independently, in slightly different ways based on some nice coordinate systems \cite{Th} \cite{Penner}, and they used a common main ingredient, namely, a special function called the {\em quantum dilogarithm} introduced by Faddeev and Kashaev  \cite{FK} \cite{F}. Later, the Chekhov-Fock construction was generalized to quantization of cluster varieties by Fock-Goncharov \cite{FG}. 

\vs

The two basic objects in the formulation of quantum Teichm\"uller theory are the {\em Teichm\"uller space} $\mcal{T}(\Sigma)$ and the {\em mapping class group} $M(\Sigma)$ of a Riemann surface $\Sigma$. They are defined as the space of all complete hyperbolic metrics on $\Sigma$ modulo isotopy, and the group of all orientation-preserving diffeomorphisms $\Sigma \to \Sigma$ modulo isotopy (i.e. homotopy), respectively. One of the main tasks and the main goals of the construction of quantum Teichm\"uller theory is to find certain family of projective unitary representations of  $M(\Sigma)$ on a Hilbert space $\mathscr{H}$.  

\vs

In general, a {\em projective representation} of a group $G$ on a vector space $V$ is given by a map
\begin{align}
\label{eq:projective_representation}
\rho : G \to {\rm GL}(V)
\end{align}
that is `almost' a group homomorphism\footnote{This is a little more than just having a group homomorphism $G\to {\rm PGL}(V)$, which is usually referred to as a `projective representation' of $G$. To distinguish, we will call $\rho$ \eqref{eq:projective_representation} an `almost-linear' representation in the later sections of the present paper; see Def.\ref{def:projective_and_almost-linear_representations}.}, i.e.
\begin{align}
\label{eq:general_rho_relation}
\rho( g_1 g_2) = c_{g_1,g_2} \, \rho(g_1) \rho(g_2), \qquad \forall g_1,g_2\in G,
\end{align}
for some constants $c_{g_1,g_2} \in \mathbb{C}^\times = \mathbb{C}\setminus\{0\}$. 
We use the well-known fact that one can `resolve' the projective representation $\rho$ of $G$ by a genuine representation (i.e. a group homomorphism)
$$
\wh{\rho} : \wh{G} \to {\rm GL}(V)
$$
of a central extension $\wh{G}$ of $G$, which means that there exists a set map $s: G \to \wh{G}$, such that ${\rm proj} \circ s = {\rm id}_G$ where ${\rm proj} : \wh{G} \to G$ is the projection, making the following diagram to commute:
\begin{align}
\nonumber
\begin{array}{l}
\xymatrix@R-4mm@C-5mm{
\wh{G} \ar[rd]^{\wh{\rho}} & & G \ar[ll]_{s} \ar[ld]_{\rho} \\
& {\rm GL}(V) &
}
\end{array}
\end{align}
It is easy to construct one such example $\wh{G}$. However, the most interesting is the smallest one.
\begin{definition}
\label{def:minimal_central_extension}
Among all central extensions $\wh{G}$ of $G$ for which the above is possible, we call the one that is contained in all the others the {\em minimal central extension resolving $\rho$}, if it exists.
\end{definition}
We concentrate on the minimal central extension of $G$ resolving a given projective representation $\rho$ of $G$. In a sense, we are taking only $c_{g_1,g_2}$ out of the data $\rho$, 
but we shall see that this already contains crucial information about $\rho$. See \S\ref{subsec:minimal_central_extensions}, \S\ref{subsec:algebraic_proof} for detailed development.

\vs

In the case of quantum Teichm\"uller theory, we have $G = M(\Sigma)$ for some Riemann surface $\Sigma$ and $V = \mathscr{H}$ for some Hilbert space $\mathscr{H}$, where the images of $\rho$ are unitary operators on $\mathscr{H}$ (so we can replace ${\rm GL}(V)$ in \eqref{eq:projective_representation} by ${\rm U}(\mathscr{H})$). 
It turns out that the corresponding minimal central extension of $M(\Sigma)$ is a central extension of $M(\Sigma)$ by $\mathbb{Z}$. 
Since the relevant projective representations $\rho$ involve the quantum dilogarithm function, this resulting central extension is called a {\em dilogarithmic} central extension of $M(\Sigma)$ by Funar, Sergiescu, and collaborators. Notice that there are two kinds of dilogarithmic central extensions of $M(\Sigma)$, one from the Kashaev quantization and the other from the Chekhov-Fock(-Goncharov) quantization. This paper grew out of the question of how to compare these two central extensions.

\vs

For quantization of $\mcal{T}(\Sigma)$, Chekhov and Fock used a coordinate system on $\mcal{T}(\Sigma)$ which requires a choice of some combinatorial-topological data on the surface $\Sigma$, namely, an `ideal triangulation' of $\Sigma$. This means a triangulation whose vertices are at punctures and boundary components, and whose edges are defined up to homotopy. Then elements of $M(\Sigma)$ are realized as transformations of ideal triangulations of $\Sigma$, thus as sequences of `flips' along edges of ideal triangulations (see Prop.\ref{prop:T_acts_transitively} for flips). So, it suffices to describe how these flips are represented as operators on $\mathscr{H}$. Meanwhile, Kashaev used an enhanced version of ideal triangulation, which we call a {\em dotted triangulation}; this is an ideal triangulation of $\Sigma$ together with the choice of a distinguished corner for each triangle, depicted as a dot $\bullet$ in pictures. Elements of $M(\Sigma)$ are then realized as transformations of dotted triangulations of $\Sigma$, and Kashaev represented `elementary' transformations of dotted triangulations as operators on some other Hilbert space. 

\vs

These results are often described in terms of the {\em Ptolemy groupoid} $Pt(\Sigma)$ for the Chekhov-Fock quantization and the {\em dotted Ptolemy groupoid} $Pt_{\rm dot}(\Sigma)$ for the Kashaev quantization, which are the category of ideal triangulations of $\Sigma$ and that of dotted triangulations of $\Sigma$, respectively; for each of these categories, there is unique morphism from any object to any object. Then we would want to construct projective representations of these categories, that is, projective functors from these categories to the category of Hilbert spaces. In order to compare these two functors for the two quantizations, we need to relate the two categories $Pt(\Sigma)$ and $Pt_{\rm dot}(\Sigma)$, for example try to build a functor between them in a natural way. However, this is not possible in general, and in the present paper we construct a functor between some full subcategories which are `orbits' of mapping class group actions, in the case of `universal' Teichm\"uller theory.

\vs

Amazingly, some difficulties and subtleties of Teichm\"uller theory and its quantization disappear when we consider a `universal' setting, in which case the relevant surface $\Sigma$ can be thought of as the open unit disc $\mathbb{D} \subset \mathbb{C}$ with a certain restriction on the behavior on the boundary $S^1 = \partial \mathbb{D}$, or a closed unit disc with countably many marked points on the boundary in the sense of \cite{FG06}. We consider infinite triangulations of $\mathbb{D}$, called {\em tessellations of $\mathbb{D}$}, with vertices at all rational points on $S^1$. The standard such tessellation is the well-known {\em Farey tessellation}, see Def.\ref{def:Farey_tessellation}. One considers the {\em universal Ptolemy groupoid $Pt$}, an analog of the Ptolemy groupoid $Pt(\Sigma)$, obtained by applying finite number of flips on the Farey tessellation; it is the category of tessellations of $\mathbb{D}$ whose edges are those of the Farey tessellation except for finitely many edges. We note that in this universal case it is necessary to introduce a decoration on such tessellations, namely the choice of a distinguished oriented edge (\cite{Penner2}). Then any two objects are connected by a finite sequence of elementary moves called $\alpha$ and $\beta$, described combinatorially in Def.\ref{def:alpha_and_beta}. There are some algebraic relations satisfied by $\alpha$ and $\beta$, so we can think of the formal group presented with generators $\alpha,\beta$ and these relations. Since this group is shown to be isomorphic to the well-known Thompson group $T$ of dyadic piecewise affine homeomorphisms of $S^1 = [0,1]/0 \hspace{-1mm} \sim \hspace{-1mm} 1$, it is called the {\em Ptolemy-Thompson group}, and we denote it by $T$ (see  \cite{FuKaS} for an exposition). A presentation computed by Lochak-Schneps \cite{LSc} is
\begin{align}
\label{eq:T_presentation_intro}
T = \left\langle \alpha,\beta \left| 
{\renewcommand{\arraystretch}{1.2} \begin{array}{l}
(\beta\alpha)^5 = 1, \quad \alpha^4 = 1,\quad \beta^3 = 1, \\
\left[\beta\alpha\beta, \, \alpha^2\beta\alpha\beta\alpha^2 \right] =1,\quad 
\left[\beta\alpha\beta, \, \alpha^2\beta\alpha^2\beta\alpha\beta\alpha^2\beta^2\alpha^2\right]=1
\end{array}} \right.
\right\rangle.
\end{align}
It is likely that there may be a certain profinite completion of this group containing the mapping class group of every Riemann surface of finite-type, i.e. Riemann surface having finite genus and $s$ punctures with $0<s<\infty$ (\cite{Penner2}). So $T$ could be viewed as a `discrete baby version' of  what can be called `the universal mapping class group'. Application of the Chekhov-Fock quantization yields a projective representation of $T$, and thus a dilogarithmic central extension of $T$ by $\mathbb{Z}$, denoted by $\wh{T}^{\rm CF}$ in the present paper (${\rm CF}$ for Chekhov-Fock).

\vs

Meanwhile, the {\em dotted universal Ptolemy groupoid $Pt_{\rm dot}$}, an analog of the dotted Ptolemy groupoid $Pt_{\rm dot}(\Sigma)$, is defined to be the category of tessellations of $\mathbb{D}$ coinciding with the Farey tessellation except for finitely many edges together with the choice of a corner for each triangle, while we require that the choice of corners differs from a fixed standard choice only on finitely many triangles. Then again any two objects are connected by a finite sequence of elementary moves. This time there are infinitely many elementary moves, with the advantage that their algebraic relations are much simpler. We define a formal group presented with generators and relations from these elementary moves, which is called the {\em Kashaev group} in \cite{FrKi}:
\begin{align}
\label{eq:K_presentation_intro}
K = \left\langle A_j, T_{jk}, P_{(jk)}  \left| 
\begin{array}{l}
A_j^3 = 1, \quad T_{k\ell} T_{jk} = T_{jk} T_{j\ell} T_{k\ell}, \quad A_j T_{jk} A_k = A_k T_{kj} A_j, \\
T_{jk} A_j T_{kj} = A_j A_k P_{(jk)}, \quad
\mbox{and trivial relations}
\end{array} \right.
\right\rangle,
\end{align}
where $j,k,\ell$ are mutually distinct elements of $\mathbb{Q}^\times$, and trivial relations mean that any generators whose subscript indices do not intersect commute, that conjugation by $P_{(jk)}$ acts as an index change $j\leftrightarrow k$, and that $P_{(jk)}$'s satisfy the permutation group relations. The Kashaev quantzation yields a projective representation of $K$ and thus a (dilogarithmic) central extension $\wh{K}$ of $K$ by $\mathbb{Z}$. 

\vs

For comparison with the Chekhov-Fock case, we should investigate how the groups $K$ and $T$ are related to each other. One of the key ideas of the present paper is the natural and essentially unique construction of a functor
$$
\mcal{F}: Pt \to Pt_{\rm dot}.
$$
This yields an injective group homomorphism
\begin{align}
\label{eq:T_to_K}
{\bf F}: T \to K,
\end{align}
which could be useful in the future projects, too. Pulling back the above central extension $\wh{K}$ of $K$ along ${\bf F}$ yields a dilogarithmic central extension $\wh{T}^{\rm Kash}$ of $T$ by $\mathbb{Z}$, coming from the Kashaev quantization (${\rm Kash}$ for Kashaev).  Now, one could ask for example if these two central extensions $\wh{T}^{\rm CF}$ and $\wh{T}^{\rm Kash}$ are equivalent as central extensions of $T$ by $\mathbb{Z}$, that is, if they correspond to the same class in the second cohomology group $H^2(T;\mathbb{Z})$. To say the result only, they correspond to different cohomology classes. However, there is another very interesting explicit way of manifesting the discrepancy between these two central extensions, using topological methods.

\vs

We shall first observe that $T$ can be viewed as a version of `asymptotically rigid' mapping class group of the unit disc $\mathbb{D}$. In order to talk about a mapping class group, we should settle which homotopies of $\mathbb{D}$ to use; we use {\em boundary-fixing} homotopies of $\mathbb{D}$, that is, homotopies of $\mathbb{D}$ that can be extended continuously to $\ol{\mathbb{D}} = \mathbb{D} \cup S^1$ and fix every point of $S^1 = \partial \mathbb{D}$ pointwise at all times. In the present paper, any homotopy of $\mathbb{D}$ is assumed to be boundary-fixing.

\begin{definition}
\label{def:mapping_class}
A {\em mapping class} of $\mathbb{D}$ is the homotopy class of homeomorphisms $\mathbb{D} \to \mathbb{D}$.
\end{definition}

Then $T$ can be identified with the group of all {\em asymptotically rigid} mapping classes of $\mathbb{D}$, defined as follows:

\begin{definition}
\label{def:asymptotically_rigid}
A homeomorphism $f : \mathbb{D} \to \mathbb{D}$ is said to be {\em asymptotically rigid} if 
\begin{enumerate}
\item it extends continuously to the boundary circle $S^1 = \partial \mathbb{D}$,  and
\item the homeomorphism $\mu^{-1} \circ (f|_{S^1}) \circ \mu : \mathbb{RP}^1 \to \mathbb{RP}^1$ is piecewise-${\rm PSL}(2,\mathbb{Z})$ with finitely many pieces whose endpoints are rational, where $\mu(x) = \frac{x-i}{x+i}$ is the Cayley transform.
\end{enumerate}
A homotopy class of asymptotically rigid homeomorphisms is called an {\em asymptotically rigid mapping class} of $\mathbb{D}$.
\end{definition}
Indeed, we can observe that elements of $T$ can be induced by asymptotically rigid mapping classes, and conversely any asymptotically rigid mapping class can be written as an element of $T$.  The terminology `asymptotically rigid' can be best justified from the fact that asymptotically rigid mapping classes `eventually' preserve the {\em Farey tessellation} (Def.\ref{def:Farey_tessellation}) of $\mathbb{D}$, i.e. they preserve the Farey tessellation except for finitely many edges of it. In this respect, one can see why we need piecewise-${\rm PSL}(2,\mathbb{Z})$ homeomorphisms\footnote{instead of, say piecewise-${\rm PSL}(2,\mathbb{Q})$ (as pointed out by a referee)}, for (globally) ${\rm PSL}(2,\mathbb{Z})$ fractional-linear homeomorphisms are what preserve the entire Farey tessellation. We note that, in fact, the identification of $T$ with the asymptotically rigid mapping class group of $\mathbb{D}$ is what makes the construction of the map \eqref{eq:T_to_K} natural and unique; see \S\ref{subsec:mcal_F}.

\vs

Now, we shall move one step further and introduce locally finite collection of countably infinite number of punctures inside $\mathbb{D}$; let $\mathbb{D}^\diamond$ be this infinitely-punctured unit disc. We require the homotopy of $\mathbb{D}^\diamond$ to pointwise fix the boundary $S^1$ {\em and} all punctures at all times, and suppose we have chosen a suitable `asymptotically rigid' condition for mapping classes of $\mathbb{D}^\diamond$ which refines Def.\ref{def:asymptotically_rigid}. Define $T^\diamond$ to be the group of all asymptotically rigid mapping classes of $\mathbb{D}^\diamond$. By forgetting the punctures, we obtain a natural map $T^\diamond \to T$, whose kernel is the group of homotopy classes of `braiding homeomorphisms' which permute the punctures (Def.\ref{def:braiding}). This kernel is isomorphic to the {\em infinite braid group} $B_\infty$ (Def.\ref{def:B_infty}), the inductive limit of the usual Artin braid group $B_n$ on $n$ strands. So, we get a short exact sequence:
\begin{align}
\label{eq:ses1}
\xymatrix{
1 \ar[r] & B_\infty \ar[r] & T^\diamond \ar[r] & T \ar[r] & 1,
}
\end{align}
and thus the group $T^\diamond$ is called a {\em braided} Ptolemy-Thompson group in \cite{FuKa2}. The abelianization of the kernel $B_\infty$ is $H_1(B_\infty) = B_\infty/[B_\infty,B_\infty] \cong \mathbb{Z}$. So, `dividing \eqref{eq:ses1} by $[B_\infty,B_\infty]$' yields another short exact sequence
\begin{align}
\label{eq:ses2}
\xymatrix{
1 \ar[r] & \mathbb{Z} \ar[r] & T^\diamond_{\rm ab} \ar[r] & T \ar[r] & 1,
}
\end{align}
where
$$
T^\diamond_{\rm ab} := T^\diamond/[B_\infty,B_\infty].
$$
This procedure is called the {\em relative abelianization} of the short exact sequence \eqref{eq:ses1}. In our case, it is easy to prove that in \eqref{eq:ses2} the kernel $\mathbb{Z}$ embeds into the center of $T^\diamond_{\rm ab}$, so that $T^\diamond_{\rm ab}$ is a {\em central} extension of $T$ by $\mathbb{Z}$ (Prop.\ref{prop:T_diamond_ab_is_central_extension}). This can be thought of as a topological method for producing a central extension of $T$ by $\mathbb{Z}$.

\vs

As a matter of fact, the choice of an `asymptotically rigid' structure of $\mathbb{D}^\diamond$ is crucial. There are two natural choices \cite{FuKa2}: $\mathbb{D}^*$ obtained by choosing a puncture on each edge of the Farey tessellation, and $\mathbb{D}^\sharp$ obtained by choosing a puncture in each triangle of the Farey tessellation. The resulting braided Ptolemy-Thompson groups are the group $T^*$ of all mapping classes of $\mathbb{D}^*$ eventually preserving the edge-punctured Farey tessellation, and the group $T^\sharp$ of all mapping classes of $\mathbb{D}^\sharp$ eventually preserving the triangle-punctured Farey tessellation. 
Funar and Sergiescu found out that the relative abelianization of $T^*$ is isomorphic to $\wh{T}^{\rm CF}$, the central extension of $T$ coming from the Chekhov-Fock quantization of the universal Teichm\"uller space:
\begin{proposition}[Funar-Sergiescu \cite{FuS}]
\label{prop:FS_T_star_ab}
One has a group isomorphism
$\wh{T}^{\rm CF} \cong T^*_{\rm ab}$.
\end{proposition}

In the present paper, we 
find that the relative abelianization of the other braided Ptolemy-Thompson group $T^\sharp$ is isomorphic to $\wh{T}^{\rm Kash}$, the central extension of $T$ coming from the Kashaev quantization of the universal Teichm\"uller space (the author acknowledges that this isomorphism is suggested to him by Louis Funar). This observation is the high point of the present paper:
\begin{proposition}
\label{prop:T_sharp_ab}
One has a group isomorphism 
$\wh{T}^{\rm Kash} \cong T^\sharp_{\rm ab}$.
\end{proposition}

It is remarkable that the phases $c_{g_1,g_2}$ \eqref{eq:general_rho_relation} appearing in the relations among the operators for the projective representations of the two quantizations are precisely captured by the topological information about braids for punctures of $\mathbb{D}$ introduced in two different natural ways. 

\vs

Meanwhile, one can be convinced that $\wh{T}^{\rm Kash}$ and $\wh{T}^{\rm CF}$ are indeed {\em distinct} central extensions of $T$, in the following sense. It is well-known that the set of equivalence classes of central extensions of $T$ by $\mathbb{Z}$ is in one-to-one correspondence with $H^2(T;\mathbb{Z})$. As said in \cite{FuS}, it is proved by Ghys-Sergiescu \cite{GS} that
\begin{align}
\label{eq:second_cohomology_group_of_T}
H^2(T;\mathbb{Z}) = \mathbb{Z} \chi \oplus \mathbb{Z} \alpha \cong \mathbb{Z} \oplus \mathbb{Z},
\end{align}
where $\chi$ and $\alpha$ are the Euler class and the discrete Godbillon-Vey class, respectively. Funar and Sergiescu devised a formula for computing the class in $H^2(T;\mathbb{Z})$ corresponding to each central extension of $T$ by $\mathbb{Z}$, if its finite presentation is given:

\begin{theorem}[Thm.1.2 of \cite{FuS}]
\label{thm:FS_classification}
Let $T_{n,p,q,r}$ be the group presented by the generators $\bar{\alpha}$, $\bar{\beta}$, $z$ and the relations
\begin{align}
\nonumber
\left\{ {\renewcommand{\arraystretch}{1.2}
\begin{array}{l}
(\bar{\beta} \bar{\alpha})^5 = z^n, \quad
\bar{\alpha}^4 = z^p, \quad
\bar{\beta}^3 = z^q, \\
\left[\bar{\beta} \bar{\alpha} \bar{\beta}, \, \bar{\alpha}^2 \bar{\beta} \bar{\alpha} \bar{\beta} \bar{\alpha}^2\right] = z^r, \quad
\left[\bar{\beta} \bar{\alpha} \bar{\beta}, \, \bar{\alpha}^2 \bar{\beta} \bar{\alpha}^2 \bar{\beta} \bar{\alpha} \bar{\beta} \bar{\alpha}^2 \bar{\beta}^2 \bar{\alpha}^2\right] = 1, \quad
\left[\bar{\alpha},z\right] = \left[\bar{\beta},z\right] = 1.
\end{array}} \right.
\end{align}
Then each central extension of $T$ by $\mathbb{Z}$ is isomorphic to $T_{n,p,q,r}$, for some $n,p,q,r\in \mathbb{Z}$. Moreover, the class $c_{T_{n,p,q,r}} \in H^2(T;\mathbb{Z})$ of the extension $T_{n,p,q,r}$ is given by
\begin{align}
\nonumber
c_{T_{n,p,q,r}} = (12n - 15p - 20q - 60r) \chi + r\alpha.
\end{align}
\end{theorem}

\begin{theorem}[\cite{FuS}: presentation of $\wh{T}^{\rm CF}$]
\label{thm:FS}
One has
\begin{align}
\label{eq:hat_T_CF}
\wh{T}^{\rm CF} \cong T_{1,0,0,0},
\end{align}
and hence the corresponding class in $H^2(T;\mathbb{Z})$ is $12\chi$.
\end{theorem}
The principal result of the present paper is the following:
\begin{theorem}[Main theorem of the present paper: presentation of $\wh{T}^{\rm Kash}$]
\label{thm:main}
One has
\begin{align}
\label{eq:hat_T_Kash}
\wh{T}^{\rm Kash} \cong T_{3,2,0,0},
\end{align}
and hence the corresponding class in $H^2(T;\mathbb{Z})$ is $6\chi$.
\end{theorem}
By looking at the corresponding classes in $H^2(T;\mathbb{Z})$, we deduce  that $\wh{T}^{\rm CF}$ and $\wh{T}^{\rm Kash}$ are not equivalent as central extensions of $T$ by $\mathbb{Z}$, and therefore, in particular, the two projective representations $\rho$ of $T$ coming from the two quantizations of the universal Teichm\"uller space are not equivalent to each other.

\vs

It is worthwhile to review how Theorems \ref{thm:FS} and \ref{thm:main} are proven in \cite{FuS} and in the present paper. Funar-Sergiescu proved \eqref{eq:hat_T_CF} using only the projective representation $\rho^{\rm CF}$ of the Chekhov-Fock quantization; namely, for each relation in \eqref{eq:T_presentation_intro}, replace each $\alpha$, $\beta$ by $\rho^{\rm CF}(\alpha)$, $\rho^{\rm CF}(\beta)$, and evaluate. Then they obtained Thm.\ref{thm:FS} with the help of Thm.\ref{thm:FS_classification} which they established separately. On the other hand, they also proved $T^*_{\rm ab} \cong T_{1,0,0,0}$, thus getting Prop.\ref{prop:FS_T_star_ab}.

\vs

For us, we have two options to reach Thm.\ref{thm:main}. With Thm.\ref{thm:FS_classification} in our hand, the key statement to prove is of course $\wh{T}^{\rm Kash} \cong T_{3,2,0,0}$ \eqref{eq:hat_T_Kash}, and one way to show this is to use only the projective representation $\rho^{\rm Kash}$ of the Kashaev quantization, like Funar and Sergiescu did. This amounts to computing the lifted $\alpha,\beta$-relations satisfied by the operators $\rho^{\rm Kash}(\alpha)$, $\rho^{\rm Kash}(\beta)$. We call such a proof an `algebraic' proof of Thm.\ref{thm:main}. In fact, complete calculation of the last two commutation relations is quite lengthy, so we only present the computation of the relations other than these two.
The other way to show \eqref{eq:hat_T_Kash} is to prove first $\wh{T}^{\rm Kash} \cong T^\sharp_{\rm ab}$ (Prop.\ref{prop:T_sharp_ab}) and then $T^\sharp_{\rm ab} \cong T_{3,2,0,0}$. It turns out that the main calculation in the proof of Prop.\ref{prop:T_sharp_ab} is short, and proof of $T^\sharp_{\rm ab} \cong T_{3,2,0,0}$ takes only small amount of topological checking (sketched in \S\ref{subsec:relative_abelianizations}). Thus we call this a `topological' proof of Thm.\ref{thm:main}; we prefer this to the algebraic proof, as it is more enlightening and does not require any clever algebraic manipulation. The reason why proving $\wh{T}^{\rm Kash} \cong T^\sharp_{\rm ab}$ (Prop.\ref{prop:T_sharp_ab}) is easy is because we do the main computation in the Kashaev group $K$ which is easier than a similar computation in $T$, and all that is left to do is to translate this computational result to the $T$ side by using a $\sharp$-punctured version ${\bf F}^\sharp$ of the map ${\bf F}$ \eqref{eq:T_to_K} which arises naturally (\S\ref{subsec:bolf_F_sharp}) and some elementary group theoretical argument.

\vs

The two quantizations of Teichm\"uller spaces have their own pros and cons, while their explicit relationship has still been somewhat mysterious. Guo and Liu \cite{GuLi} tried to build a bridge between the two constructions in a purely algebraic way. 
Namely, for each $g\in M(\Sigma)$, conjugation by $\rho(g)$ defines an automorphism of the algebra of operators on $\mathscr{H}$, and they studied how these algebra automorphisms in the two quantizations are related to each other. Since conjugation forgets multiplicative constants, the information on the phases $c_{g_1,g_2}$ for the projective representations $\rho$ as in \eqref{eq:general_rho_relation} is then lost. 
Only in the level of projective representations $\rho$, can we observe the discrepancy between the two quantizations as discussed above. 
We note that in the present paper, in addition to the topological interpretations of just the phases $c_{g_1,g_2}$, a natural relationship between the two projective representations $\rho$ themselves for the two quantizations of the universal Teichm\"uller space are also suggested via the map ${\bf F}$  \eqref{eq:T_to_K}, although this doesn't give an equivalence between these two.

\vs

Let us now list some possible directions for further research. First, one can try to mimic what is done in the present paper in the cases of finite-type Riemann surfaces $\Sigma$. Funar and Kashaev \cite{FuKas} have a result analogous to Thm.\ref{thm:main} using the Kashaev quantization, and Xu \cite{Xu14} has a result analogous to Thm.\ref{thm:FS} using the Chekhov-Fock quantization; both works rely on extensive algebraic proofs. As suggested by Funar, it is an interesting problem to search for topological interpretations of the resulting minimal central extensions of $M(\Sigma)$ by $\mathbb{Z}$, as done in Propositions \ref{prop:FS_T_star_ab} and \ref{prop:T_sharp_ab}, because na\"ive candidates do not work; the relative abelianization of extensions of $M(\Sigma)$ by braid groups on $\Sigma$ yields central extensions of $M(\Sigma)$ by $\mathbb{Z}/2\mathbb{Z}$, not by $\mathbb{Z}$, when $\Sigma$ is of positive genus (as pointed out to the author by Funar; see \cite{BeFu}).

\vs

In fact, more interesting is still the universal case. For example, the original problem posed to the author by Igor Frenkel is to interpret the projective representations and the central extensions of $T$ in the language of the representation theory of a rather basic Hopf algebra $\mcal{B}$, `the modular double of the quantum plane'. In \cite{FrKi} Frenkel and the author realized the quantum (universal) Teichm\"uller space $\mathscr{H}$ as the space of intertwiners of $\mcal{B}$, and the suggested problem is to realize the operators for $\alpha,\beta\in T$ corresponding to the projective representation $\rho$ as some kind of permutation (and `dualizing') operators on a certain infinite tensor power of the unique irreducible integrable representation of $\mcal{B}$. Meanwhile, recall that quantum Teichm\"uller theory provides genuine representations of the relative abelianizations $T^*_{\rm ab}$, $T^\sharp_{\rm ab}$ of the braided Ptolemy-Thompson groups $T^*$, $T^\sharp$. One can then ask if we can construct representations of $T^*$ and $T^\sharp$ themselves, on which $B_\infty$ acts faithfully; this is currently in progress, and will be published elsewhere. This problem, suggested by Funar and Frenkel, is important in two aspects. One is that such representations may be used to `flatten out' the known representations of $T^*_{\rm ab}$, $T^\sharp_{\rm ab}$ resulting from quantum Teichm\"uller theory, and the other is that there could be a connection to the Grothendieck-Teichm\"uller group $\wh{GT}$, as in Lochak-Schneps' work \cite{LSc}.

\begin{remark}
\label{rem:trivial_constant}
The projective representations of $T$ used in \cite{FuS} by Funar and Sergiescu which yield the central extension $\wh{T}^{\rm CF}$, as well as those of the mapping class groups of finite-type surfaces used in \cite{Xu14}, come from the quantization result of the paper \cite{FG}, which says that the operators representing the generators satisfy the algebraic relations up to complex constants of modulus 1. The author of the present paper computed these constants, in order to formulate a more precise way to `compare' the second cohomology classes coming from the two different quantizations. Namely, instead of considering the minimal central extension resolving a chosen `almost' group homomorphism $T\to {\rm GL}(V)$ and then computing the corresponding class in $H^2(T;\mathbb{Z})$, one can think of the group homomorphism $T\to {\rm PGL}(V)$ which is `more invariant', which yields a well-defined class in $H^2(T;{\rm U}(1))$. Then, the choice of an appropriate embedding $\mathbb{Z} \hookrightarrow {\rm U}(1)$ lets us find the class in $H^2(T;\mathbb{Z})$ that corresponds to the well-defined class in $H^2(T;{\rm U}(1))$ under the induced map $H^2(T;\mathbb{Z}) \to H^2(T;{\rm U}(1))$. Using the same embedding $\mathbb{Z} \hookrightarrow {\rm U}(1)$ for the two stories then lets us compare the two results more precisely. However, the author found out in \cite{K16} that the constants appearing in \cite{FG}, which are denoted by $\lambda$ there, are all $1$. Therefore, the Fock-Goncharov quantization of \cite{FG} yields genuine representations of $T$, not projective. So, in order to recover what are asserted in \cite{FuS} and \cite{Xu14}, one must construct another quantization of Teichm\"uller spaces that resembles Chekhov-Fock-Goncharov's result but does not yield trivial constants.
\end{remark}

\noindent{\bf Acknowledgments.} I am greatly indebted to Louis Funar and Vlad Sergiescu for abundant help, suggestions and discussions about this work, and therefore would like to warmly thank them. I thank Igor B. Frenkel for suggesting the original problem, and for his helpful comments. I thank all the referees for their suggestions to make the paper better.

\section{Decorated universal Ptolemy groupoids}
\label{sec:decorated_universal_Ptolemy_groupoids}

In this section, we study certain infinite tessellations (i.e. triangulations) of the open unit disc $\mathbb{D}$, and two kinds of decorations on the tessellations: marked tessellations and dotted tessellations. We study the groups of transformations of these enhanced tessellations, the Ptolemy-Thompson group $T$ and the Kashaev group $K$ respectively, and build a natural map ${\bf F}: T\to K$. These two groups are the main basic ingredients of the present paper. We shall introduce two groupoids, which lead to these groups.

\subsection{Tessellations of the unit disc $\mathbb{D}$}

Recall that any homotopy of $\mathbb{D}$ is assumed to be boundary-fixing, as noted in \S\ref{sec:introduction}.

\begin{definition}
\label{def:ideal_arc}
An {\em ideal arc} of the unit disc $\mathbb{D}$ connecting two distinct points on the unit circle $S^1 = \partial \mathbb{D}$ is a homotopy class of unoriented paths in $\mathbb{D}$ connecting the two points. The connected region bounded by three ideal arcs connecting three distinct points on $S^1$ is called an {\em ideal triangle}. The three ideal arcs bounding an ideal triangle are called the {\em sides} of the triangle.
\end{definition}

In the figures appearing in the present paper and usually in the literature, each ideal arc is often assumed to be stretched to the unique hyperbolic geodesic with respect to the usual Poincar\'e hyperbolic metric $ds^2 = \frac{|dz|^2}{(1-|z|^2)^2}$, so that it is a part of some circle, which intersects the unit circle at the right angle.

\begin{definition}
A {\em tessellation} $\tau$ of $\mathbb{D}$ is a countable locally finite collection of ideal arcs of $\mathbb{D}$ whose complementary region in $\mathbb{D}$ is the disjoint union of ideal triangles. 
The ideal arcs constituting $\tau$ are called {\em edges} of $\tau$.
The endpoints of the edges of $\tau$ are called {\em vertices} of $\tau$.
\end{definition}

Via the Cayley transformation $\mu$ (Def.\ref{def:asymptotically_rigid}), each point on $S^1 = \partial \mathbb{D}$ gets labeled by the corresponding element of $\mathbb{RP}^1 = \mathbb{R}\cup \{\infty\} = \partial \mathbb{H}$. In the present paper, we only study tessellations of $\mathbb{D}$ whose vertices are rational points of $S^1$, each of which can be written as $\mu(r) \in S^1$ for some $r\in \mathbb{Q} \cup \{\infty\} = \mathbb{QP}^1 \subset \mathbb{RP}^1$; we will denote this point just by $r$ in pictures, e.g. as in Fig.\ref{fig:marked_tessellations}. Since $\mathbb{Q}\cup \{\infty\}$ will come up often, we first settle the notation for its elements:

\begin{definition}
\label{def:reduced_expression}
A nonzero rational number $\frac{p}{q}$ is said to be in the {\em reduced expression} if $p,q\in \mathbb{Z}$, $q> 0$ and ${\rm gcd}(p,q)=1$. We set $\frac{0}{1}$ for the reduced expression for $0$, and $\frac{1}{0}$ or $\frac{-1}{0}$ for the reduced expressions for $\infty$. We call the elements of $\mathbb{Q}\cup \{\infty\}$ the {\em extended rationals}.
\end{definition}

The most important example of tessellations of $\mathbb{D}$ is the {\em Farey tessellation}: 

\begin{definition}
\label{def:Farey_tessellation}
The {\em Farey tessellation} $\tau^*$ is the tessellation whose vertices are all the rational points of $S^1$ (i.e. ${\tau^*}^{(1)} = \mathbb{Q}\cup\{\infty\}$ via $\mu$), in which the two rational points $\mu(\frac{a}{b})$ and $\mu(\frac{c}{d})$ (where $\frac{a}{b}$ and $\frac{c}{d}$ are reduced expressions) are connected by an ideal arc if and only if $|ad - bc| = 1$.

\end{definition}

One can show that if $\frac{a}{b}$ and $\frac{c}{d}$ satisfies $ad-bc=1$, then any $\frac{c'}{d'}$ satisfying $ad' - bc'=1$ can be written as $c' = c + na$, $d' = d+na$ for some integer $n$, and vice versa. Using this fact, one can argue that the collection of arcs among the rational points of $S^1$ defined in Def.\ref{def:Farey_tessellation} indeed defines a tessellation, i.e. no two arcs intersect in the interior of $\mathbb{D}$. We omit the detailed proof.

\begin{remark}
\label{rem:triangle_Farey}
Any ideal triangle of the Farey tessellation $\tau^*$ has the vertices $\mu(\frac{a}{b}), \mu(\frac{a+c}{b+d}), \mu(\frac{c}{d})$ for some extended rationals $\frac{a}{b}, \frac{c}{d}$ (in reduced expressions).
\end{remark}

See Fig.\ref{subfig:standard_marked_tessellation} for the Farey tessellation; ignore the arrowhead. More general tessellations that we are interested in are of the following type (see Fig.\ref{subfig:marked_tessellation_ex}, ignoring the arrowhead):

\begin{definition}
\label{def:Farey-type_tessellations}
A {\em Farey-type tessellation} is a tessellation $\tau$ whose vertices are all the rational points of $S^1$, all but finitely many of which ideal arcs are those of the Farey tessellation $\tau^*$. 
\end{definition}
In the present paper, 
a `tessellation' would automatically mean a `Farey-type tessellation'. 
The first kind of decoration on tessellations we need to consider is as follows.

\begin{definition}
\label{def:marked_tessellations}
A {\em marked tessellation} $(\tau,\vec{a})$ is a tessellation $\tau$ together with the choice of a {\em distinguished oriented edge (d.o.e.)} $\vec{a}$. The {\rm standard marked tessellation} $(\tau^*, \vec{a}^*)$ is the Farey tessellation $\tau^*$ together with the d.o.e. $\vec{a}^*$ being the arc connecting $\mu(0)$ and $\mu(\infty)$, with the direction $\mu(0)\to \mu(\infty)$ (see Fig.\ref{subfig:standard_marked_tessellation}).
We denote the standard marked tessellation $(\tau^*, \vec{a}^*)$ by $\tau^*_{\rm mark}$, and a general marked tessellation $(\tau,\vec{a})$ by $\tau_{\rm mark}$ if the d.o.e. $\vec{a}$ is clear from the context.
\end{definition}

The d.o.e. $\vec{a}$ is indicated by an arrow in the pictures, as in Fig.\ref{fig:marked_tessellations}. For a general example of marked tessellations, see Fig.\ref{subfig:marked_tessellation_ex}.


\begin{figure}[htbp!]
\begin{subfigure}[b]{0.5\textwidth}
\centering
\begin{pspicture}[showgrid=false,linewidth=0.5pt,unit=0.9](-2.6,-2.8)(2.6,2.6)
\psarc(0,0){2.4}{0}{360}
\psarc(-2.4,2.4){2.4}{-90}{0}
\psarc(-2.4,-2.4){2.4}{0}{90}
\psarc(2.4,-2.4){2.4}{90}{180}
\psarc(2.4,2.4){2.4}{180}{-90}
\psarc[arcsep=0.5pt](2.4,1.2){1.2}{142.5}{-90}
\psarc[arcsep=0.5pt](0.8,2.4){0.8}{180}{-40}
\psarc[arcsep=0.5pt](2.4,0.8){0.8}{127}{-90}
\psarc[arcsep=0.5pt](1.714,1.714){0.343}{140}{-51}
\psarc[arcsep=0.5pt](0.48,2.4){0.48}{180}{-22}
\psarc[arcsep=0.5pt](1.2,2.1){0.3}{160}{-34}
\psarc[arcsep=0.5pt](2.4,0.6){0.6}{117}{-90}
\psarc[arcsep=0.5pt](2.03,1.29){0.185}{127}{-60}
\rput(-2.6,0){$\frac{0}{1}$}
\rput(2.6,0){$\frac{1}{0}$}
\rput(0,2.7){-$\frac{1}{1}$}
\rput(1.5,2.2){-$\frac{2}{1}$}
\rput(2.09,1.58){-$\frac{3}{1}$}
\rput(0.87,2.5){-$\frac{3}{2}$}
\rput(2.35,1.15){-$\frac{4}{1}$}
\psarc[arcsep=0.5pt](-2.4,1.2){1.2}{-90}{37.5}
\psarc[arcsep=0.5pt](-0.8,2.4){0.8}{-140}{0}
\psarc[arcsep=0.5pt](-2.4,0.8){0.8}{-90}{53}
\psarc[arcsep=0.5pt](-1.714,1.714){0.343}{-129}{40}
\psarc[arcsep=0.5pt](-0.48,2.4){0.48}{-158}{0}
\psarc[arcsep=0.5pt](-1.2,2.1){0.3}{-146}{20}
\psarc[arcsep=0.5pt](-2.4,0.6){0.6}{-90}{63}
\psarc[arcsep=0.5pt](-2.03,1.29){0.185}{-120}{53}
\rput(-1.55,2.2){-$\frac{1}{2}$}
\rput(-2.03,1.69){-$\frac{1}{3}$}
\rput(-0.89,2.5){-$\frac{2}{3}$}
\rput(-2.33,1.26){-$\frac{1}{4}$}
\psarc[arcsep=0.5pt](-2.4,-1.2){1.2}{-37.5}{90}
\psarc[arcsep=0.5pt](-0.8,-2.4){0.8}{0}{140}
\psarc[arcsep=0.5pt](-2.4,-0.8){0.8}{-53}{90}
\psarc[arcsep=0.5pt](-1.714,-1.714){0.343}{-40}{129}
\psarc[arcsep=0.5pt](-0.48,-2.4){0.48}{0}{158}
\psarc[arcsep=0.5pt](-1.2,-2.1){0.3}{-20}{146}
\psarc[arcsep=0.5pt](-2.4,-0.6){0.6}{-63}{90}
\psarc[arcsep=0.5pt](-2.03,-1.29){0.185}{-53}{120}
\rput(-1.57,-2.13){$\frac{1}{2}$}
\rput(-2.02,-1.68){$\frac{1}{3}$}
\rput(-0.92,-2.5){$\frac{2}{3}$}
\rput(-2.30,-1.18){$\frac{1}{4}$}
\psarc[arcsep=0.5pt](2.4,-1.2){1.2}{90}{-142.5}
\psarc[arcsep=0.5pt](0.8,-2.4){0.8}{40}{-180}
\psarc[arcsep=0.5pt](2.4,-0.8){0.8}{90}{-127}
\psarc[arcsep=0.5pt](1.714,-1.714){0.343}{51}{-140}
\psarc[arcsep=0.5pt](0.48,-2.4){0.48}{22}{-180}
\psarc[arcsep=0.5pt](1.2,-2.1){0.3}{34}{-160}
\psarc[arcsep=0.5pt](2.4,-0.6){0.6}{90}{-117}
\psarc[arcsep=0.5pt](2.03,-1.29){0.185}{60}{-127}
\rput(0,-2.7){$\frac{1}{1}$}
\rput(1.5,-2.2){$\frac{2}{1}$}
\rput(1.98,-1.71){$\frac{3}{1}$}
\rput(0.92,-2.5){$\frac{3}{2}$}
\rput(2.30,-1.18){$\frac{4}{1}$}
\psline[linewidth=0.5pt, 
arrowsize=3pt 4, 
arrowlength=2, 
arrowinset=0.3] 
{->}(0,0)(0.2,0)
\psline{-}(-2.4,0)(2.4,0)
\rput{106}(2.15,0.52){\fontsize{12}{12} $\cdots$}
\rput{-106}(2.13,-0.52){\fontsize{12}{12} $\cdots$}
\rput{74}(-2.13,0.52){\fontsize{12}{12} $\cdots$}
\rput{-74}(-2.15,-0.52){\fontsize{12}{12} $\cdots$}
\rput{-10}(0.40,2.18){\fontsize{10}{10} $\cdots$}
\rput{10}(-0.45,2.18){\fontsize{10}{10} $\cdots$}
\rput{-10}(-0.45,-2.18){\fontsize{10}{10} $\cdots$}
\rput{10}(0.40,-2.18){\fontsize{10}{10} $\cdots$}
\rput{-43}(1.58,1.61){\fontsize{8}{8} $\cdots$}
\rput{43}(-1.62,1.59){\fontsize{8}{8} $\cdots$}
\rput{-43}(-1.62,-1.59){\fontsize{8}{8} $\cdots$}
\rput{43}(1.60,-1.61){\fontsize{8}{8} $\cdots$}
\rput{-29}(1.10,2.00){\fontsize{7}{7} $\cdots$}
\rput{29}(-1.16,1.98){\fontsize{7}{7} $\cdots$}
\rput{-29}(-1.16,-1.98){\fontsize{7}{7} $\cdots$}
\rput{29}(1.10,-2.00){\fontsize{7}{7} $\cdots$}
\rput{-59}(1.96,1.29){\fontsize{2}{2} $\cdots$}
\rput{59}(-1.99,1.25){\fontsize{2}{2} $\cdots$}
\rput{59}(1.96,-1.29){\fontsize{2}{2} $\cdots$}
\rput{-59}(-1.99,-1.25){\fontsize{2}{2} $\cdots$}
\end{pspicture}
\caption{The standard marked tessellation}
\label{subfig:standard_marked_tessellation}
\end{subfigure}
\begin{subfigure}[b]{65mm}
\centering
\begin{pspicture}[showgrid=false,linewidth=0.5pt,unit=0.9](-2.6,-2.8)(2.6,2.6)
\psarc(0,0){2.4}{0}{360}
\psarc(-2.4,2.4){2.4}{-90}{0}
\psarc(-2.4,-2.4){2.4}{0}{90}
\psarc[arcsep=0.5pt](2.4,1.2){1.2}{142.5}{-90}
\psarc[arcsep=0.5pt](0.8,2.4){0.8}{180}{-40}
\psarc[arcsep=0.5pt](2.4,0.8){0.8}{127}{-90}
\psarc[arcsep=0.5pt](1.714,1.714){0.343}{140}{-51}
\psarc[arcsep=0.5pt](0.48,2.4){0.48}{180}{-22}
\psarc[arcsep=0.5pt](1.2,2.1){0.3}{160}{-34}
\psarc[arcsep=0.5pt](2.4,0.6){0.6}{117}{-90}
\psarc[arcsep=0.5pt](2.03,1.29){0.185}{127}{-60}
\rput(-2.6,0){$\frac{0}{1}$}
\rput(2.6,0){$\frac{1}{0}$}
\rput(0,2.7){-$\frac{1}{1}$}
\rput(1.5,2.2){-$\frac{2}{1}$}
\rput(2.09,1.58){-$\frac{3}{1}$}
\rput(0.87,2.5){-$\frac{3}{2}$}
\rput(2.35,1.15){-$\frac{4}{1}$}
\psarc[arcsep=0.5pt](-2.4,1.2){1.2}{-90}{37.5}
\psarc[arcsep=0.5pt](-0.8,2.4){0.8}{-140}{0}
\psarc[arcsep=0.5pt](-2.4,0.8){0.8}{-90}{53}
\psarc[arcsep=0.5pt](-1.714,1.714){0.343}{-129}{40}
\psarc[arcsep=0.5pt](-0.48,2.4){0.48}{-158}{0}
\psarc[arcsep=0.5pt](-1.2,2.1){0.3}{-146}{20}
\psarc[arcsep=0.5pt](-2.4,0.6){0.6}{-90}{63}
\psarc[arcsep=0.5pt](-2.03,1.29){0.185}{-120}{53}
\rput(-1.55,2.2){-$\frac{1}{2}$}
\rput(-2.03,1.69){-$\frac{1}{3}$}
\rput(-0.89,2.5){-$\frac{2}{3}$}
\rput(-2.33,1.26){-$\frac{1}{4}$}
\psarc[arcsep=0.5pt](-0.8,-2.4){0.8}{0}{140}
\psarc[arcsep=0.5pt](-2.4,-0.8){0.8}{-53}{90}
\psarc[arcsep=0.5pt](-1.714,-1.714){0.343}{-40}{129}
\psarc[arcsep=0.5pt](-0.48,-2.4){0.48}{0}{158}
\psarc[arcsep=0.5pt](-1.2,-2.1){0.3}{-20}{146}
\psarc[arcsep=0.5pt](-2.4,-0.6){0.6}{-63}{90}
\psarc[arcsep=0.5pt](-2.03,-1.29){0.185}{-53}{120}
\rput(-1.57,-2.13){$\frac{1}{2}$}
\rput(-2.02,-1.68){$\frac{1}{3}$}
\rput(-0.92,-2.5){$\frac{2}{3}$}
\rput(-2.30,-1.18){$\frac{1}{4}$}
\psarc[arcsep=0.5pt](2.4,-1.2){1.2}{90}{-142.5}
\psarc[arcsep=0.5pt](0.8,-2.4){0.8}{40}{-180}
\psarc[arcsep=0.5pt](2.4,-0.8){0.8}{90}{-127}
\psarc[arcsep=0.5pt](1.714,-1.714){0.343}{51}{-140}
\psarc[arcsep=0.5pt](0.48,-2.4){0.48}{22}{-180}
\psarc[arcsep=0.5pt](1.2,-2.1){0.3}{34}{-160}
\psarc[arcsep=0.5pt](2.4,-0.6){0.6}{90}{-117}
\psarc[arcsep=0.5pt](2.03,-1.29){0.185}{60}{-127}
\rput(0,-2.7){$\frac{1}{1}$}
\rput(1.5,-2.2){$\frac{2}{1}$}
\rput(1.98,-1.71){$\frac{3}{1}$}
\rput(0.92,-2.5){$\frac{3}{2}$}
\rput(2.30,-1.18){$\frac{4}{1}$}
%
\psarc[arcsep=0.5pt](-2.4,4.8){4.8}{-90}{-37.5}
\psarc[arcsep=0.5pt](7.2,-2.4){7.2}{143.0}{180}
\psarc[arcsep=0.5pt](4.0,0.0){3.2}{142.7}{217.5}
\psarc[arcsep=0.5pt](-1.2,-2.4){1.2}{0}{127.5}
%
\psline[linewidth=0.5pt, 
arrowsize=3pt 4, 
arrowlength=2, 
arrowinset=0.3] 
{->}(-0.8,0.6)(-0.81,0.59) 
\rput{106}(2.15,0.52){\fontsize{12}{12} $\cdots$}
\rput{-106}(2.13,-0.52){\fontsize{12}{12} $\cdots$}
\rput{74}(-2.13,0.52){\fontsize{12}{12} $\cdots$}
\rput{-74}(-2.15,-0.52){\fontsize{12}{12} $\cdots$}
\rput{-10}(0.40,2.18){\fontsize{10}{10} $\cdots$}
\rput{10}(-0.45,2.18){\fontsize{10}{10} $\cdots$}
\rput{-10}(-0.45,-2.18){\fontsize{10}{10} $\cdots$}
\rput{10}(0.40,-2.18){\fontsize{10}{10} $\cdots$}
\rput{-43}(1.58,1.61){\fontsize{8}{8} $\cdots$}
\rput{43}(-1.62,1.59){\fontsize{8}{8} $\cdots$}
\rput{-43}(-1.62,-1.59){\fontsize{8}{8} $\cdots$}
\rput{43}(1.60,-1.61){\fontsize{8}{8} $\cdots$}
\rput{-29}(1.10,2.00){\fontsize{7}{7} $\cdots$}
\rput{29}(-1.16,1.98){\fontsize{7}{7} $\cdots$}
\rput{-29}(-1.16,-1.98){\fontsize{7}{7} $\cdots$}
\rput{29}(1.10,-2.00){\fontsize{7}{7} $\cdots$}
\rput{-59}(1.96,1.29){\fontsize{2}{2} $\cdots$}
\rput{59}(-1.99,1.25){\fontsize{2}{2} $\cdots$}
\rput{59}(1.96,-1.29){\fontsize{2}{2} $\cdots$}
\rput{-59}(-1.99,-1.25){\fontsize{2}{2} $\cdots$}
\end{pspicture}
\caption{A general marked tessellation}
\label{subfig:marked_tessellation_ex}
\end{subfigure}

\vspace{-2mm}

\caption{Examples of marked tessellations ($\mu$ is omitted in the vertex labels)}
\label{fig:marked_tessellations}
\end{figure}
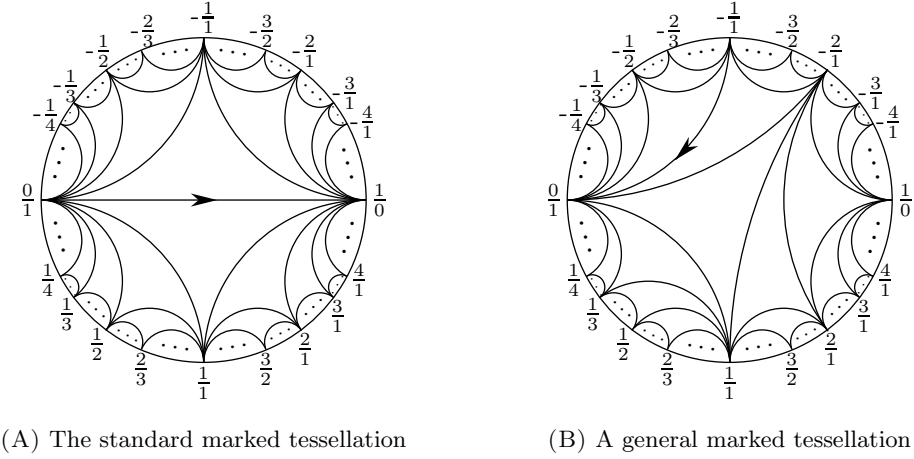

The second decoration of tessellations is the choice of a corner in each ideal triangle, together with a labeling rule of the triangles:

\begin{definition}
\label{def:dotted_tessellations}
A {\em dotted tessellation} $(\tau,D,L)$ is a tessellation $\tau$ together with a rule $D$ which assigns to each triangle a distinguished corner, indicated by a dot $(\bullet)$ in the picture (see Fig.\ref{fig:dotted_tessellations}), and a  choice $L$ of labeling of the triangles by $\mathbb{Q}^\times = \mathbb{Q} \setminus\{0\}$, i.e. a bijection from the set of all ideal triangles of $\tau$ to $\mathbb{Q}^\times$.

The {\em standard dotted tessellation} $(\tau^*,D^*,L^*)$ is the Farey tessellation $\tau^*$ with the dots on the `middle vertices' of the triangles (for a triangle of $\tau^*$ with the vertices $\mu(\frac{a}{b}), \mu(\frac{a+c}{b+d}), \mu(\frac{c}{d})$ as in Rem.\ref{rem:triangle_Farey}, the `middle vertex' is $\mu(\frac{a+c}{b+d})$), where the label of each triangle comes from the middle vertex; see Fig.\ref{subfig:standard_dotted_tessellation}.
We require that the dotting rule $D$ for all but finitely many ideal triangles of a dotted tessellation should coincide with that of $\tau^*$.
We denote the standard dotted tessellation $(\tau^*, D^*, L^*)$ by $\tau^*_{\rm dot}$, and $(\tau,D,L)$ by $\tau_{\rm dot}$ if $D$ and $L$ are clear from the context.
\end{definition}

\begin{remark}
How do we see that the above labeling rule $L^*$, which a priori is just a map from the set of triangles of $\tau^*$ to extended rationals, is a bijection to $\mathbb{Q}^\times$? One way of seeing this is via a recursive construction of the Farey tessellation $\tau^*$. Let us focus on the `lower half' of $\mathbb{D}$. We start from one triangle whose vertices are $\mu(\frac{0}{1})$, $\mu(\frac{1}{1})$, $\mu(\frac{1}{0})$, the `middle vertex' being $\mu(\frac{1}{1})$, hence labeled by $\frac{1}{1}$. A `procedure' takes as an input a triangle, whose vertices are $\mu(\frac{a}{b})$, $\mu(\frac{a+c}{b+d})$, $\mu(\frac{c}{d})$, hence $\mu(\frac{a+c}{b+d})$ being the `middle vertex', and yields two new triangles, one having vertices $\mu(\frac{a}{b})$, $\mu(\frac{2a+c}{2b+d})$, $\mu(\frac{a+c}{b+d})$, the `middle vertex' being $\mu(\frac{2a+c}{2b+d})$, and the other having vertices $\mu(\frac{a+c}{b+d})$, $\mu(\frac{a+2c}{b+2d})$, $\mu(\frac{c}{d})$, the `middle vertex' being $\mu(\frac{a+2c}{b+2d})$. We apply the `procedure' to the initial triangle to get $2$ more new triangles. Then we apply the `procedure' to each of the $2$ new triangles, to get $2^2$ more new triangles. Then we apply `procedure' to each of the $2^2$ new triangles to get $2^3$ more new triangles. And so on. It is clear that thus created triangles all have distinct `middle vertices', hence distinct labels. A standard story on the Farey tessellation, e.g. the one about the `continued fraction' expression of a rational number, tells us that any positive rational number appears as the label of one of the triangles thus obtained.
\end{remark}

In the pictures, we write $[j]$ inside the triangle labeled by $j\in \mathbb{Q}^\times$, as in Fig.\ref{fig:dotted_tessellations}. For a general example of dotted tessellations, see Fig.\ref{subfig:dotted_tessellation_ex}.

\input{fig-dotted_tessellations.tex}

\subsection{Ptolemy-Thompson group $T$ and Kashaev group $K$}
\label{subsec:T_and_K}

We now investigate groups of transformations of marked tessellations and dotted tessellations. A convenient and popular way of studying these is to consider groupoids of the decorated tessellations. Recall that a {\em groupoid} is a category in which every morphism has an inverse.

\begin{definition}[\cite{Penner2}]
\label{def:Ptolemy_groupoid}
Let the {\em universal Ptolemy groupoid} $Pt$ be the category whose objects are the marked tessellations (Def.\ref{def:marked_tessellations}) and for any objects $\tau_{\rm mark},\tau'_{\rm mark}$ there is exactly one morphism denoted by $[\tau_{\rm mark},\tau'_{\rm mark}]$. We set the composition of morphisms by
\begin{align}
\label{eq:composition_of_morphisms}
[\tau_{\rm mark}',\tau_{\rm mark}''] \circ [\tau_{\rm mark},\tau_{\rm mark}'] = [\tau_{\rm mark}, \tau_{\rm mark}''].
\end{align}
\end{definition}
\begin{definition}
\label{def:Pt_dot}
Analogously, define the {\em dotted universal Ptolemy groupoid} $Pt_{\rm dot}$ to be the category whose objects are dotted tessellations (Def.\ref{def:dotted_tessellations}) and for any objects $\tau_{\rm dot}$, $\tau'_{\rm dot}$ there is exactly one morphism denoted by $[\tau_{\rm dot}, \tau'_{\rm dot}]$. The composition rule is analogous to \eqref{eq:composition_of_morphisms}.
\end{definition}

\begin{remark}
Some authors, including Penner \cite{Penner2}, use the composition rule written in an opposite order to \eqref{eq:composition_of_morphisms}. 
\end{remark}

We first take a look into $Pt$. Each morphism $[\tau_{\rm mark},\tau_{\rm mark}']$ of $Pt$ can be thought of as a `transformation of marked triangulation' from $\tau_{\rm mark}$ to $\tau_{\rm mark}'$. Among these, there are `elementary' ones generating the whole groupoid $Pt$, which can be combinatorially described as follows:

\begin{definition}
\label{def:alpha_and_beta}
We label a morphism $[(\tau, \vec{a}), \, (\tau',\vec{a}')]$ of $Pt$ by $\alpha$ or $\beta$, if it falls into the relevant description as follows (see Fig.\ref{fig:alpha_beta}):

\begin{enumerate}
\item $\alpha$-move: Locate the ideal quadrilateral of $\tau$ formed by the two ideal triangles having $\vec{a}$ as one of their sides (Def.\ref{def:ideal_arc}). Then $\tau'$ is obtained by erasing the edge $\vec{a}$ from $\tau$ and adding the other ideal diagonal of this quadrilateral. This new edge is the new d.o.e. $\vec{a}'$, with the orientation given as if we obtained $\vec{a}'$ by rotating $\vec{a}$ counterclockwise.

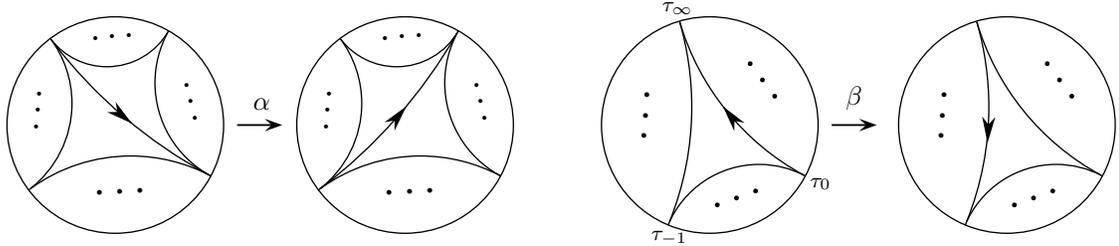
\begin{figure}[htbp!]
$\begin{array}{llll}
\begin{pspicture}[showgrid=false,linewidth=0.5pt,unit=6mm](-1.5,-1.2)(2.5,1.5)
\psarc(0,0){2.4}{0}{360}
%
\psarc[arcsep=0.5pt](0.343,-4.457){3.771}{61.7}{126.4}
\psarc[arcsep=0.5pt](-3.36,0.48){2.4}{-53.6}{37.4}
\psarc[arcsep=0.5pt](-0.185,2.862){1.569}{-142.6}{-29}
\psarc[arcsep=0.5pt](3.220,0.937){2.341}{151}{-118.3}
%
\psarc[arcsep=0.5pt](7.2,8.4){10.8}{-142.6}{-118.3}
%
\psline[linewidth=0.5pt, 
arrowsize=3pt 4, 
arrowlength=2, 
arrowinset=0.3] 
{->}(0.35,0.050)(0.37,0.034)
%
\rput{104.0}(1.7,0.5){\fontsize{15}{15} $\cdots$}
\rput{5.0}(-0.10,1.94){\fontsize{15}{15} $\cdots$}
\rput{85.0}(-1.7,0.3){\fontsize{15}{15} $\cdots$}
\rput{4.0}(0.08,-1.50){\fontsize{17}{17} $\cdots$}
%
\rput[l](2.7,0){\pcline[linewidth=0.7pt, arrowsize=2pt 4]{->}(0,0)(1;0)\Aput{\,$\alpha$}}
\end{pspicture}
%
%
%
%
& \begin{pspicture}[showgrid=false,linewidth=0.5pt,unit=6mm](-1.0,-1.2)(3.0,1.5)
\psarc(0,0){2.4}{0}{360}
%
\psarc[arcsep=0.5pt](0.343,-4.457){3.771}{61.7}{126.4}
\psarc[arcsep=0.5pt](-3.36,0.48){2.4}{-53.6}{37.4}
\psarc[arcsep=0.5pt](-0.185,2.862){1.569}{-142.6}{-29}
\psarc[arcsep=0.5pt](3.220,0.937){2.341}{151}{-118.3}
%
%
\psarc[arcsep=0.5pt](-8.8,7.733){11.467}{-53.0}{-29.7}
%
\psline[linewidth=0.5pt, 
arrowsize=3pt 4, 
arrowlength=2, 
arrowinset=0.3] 
{->}(0.01,0.393)(0.03,0.417)
%
\rput{104.0}(1.7,0.5){\fontsize{15}{15} $\cdots$}
\rput{5.0}(-0.10,1.94){\fontsize{15}{15} $\cdots$}
\rput{85.0}(-1.7,0.3){\fontsize{15}{15} $\cdots$}
\rput{4.0}(0.08,-1.50){\fontsize{17}{17} $\cdots$}
\end{pspicture}
%
%
%
%
& \begin{pspicture}[showgrid=false,linewidth=0.5pt,unit=6mm](-0.7,-1.2)(3.3,1.5)
\psarc(0,0){2.4}{0}{360}
%
\psarc[arcsep=0.5pt](1.091,-3.055){2.182}{61.7}{158.0}
\psarc[arcsep=0.5pt](-7.2,0.4){6.8}{-22}{16.3}
\psarc[arcsep=0.5pt](4.8,3.9){5.7}{-163.7}{-118.3}
%
\psline[linewidth=0.5pt, 
arrowsize=3pt 4, 
arrowlength=2, 
arrowinset=0.3] 
{->}(0.31,0.389)(0.29,0.414)
%
\rput{127.0}(1.2,1.0){\fontsize{19}{19} $\cdots$}
\rput{86.0}(-1.4,0.2){\fontsize{19}{19} $\cdots$}
\rput{20.0}(0.58,-1.65){\fontsize{17}{17} $\cdots$}
%
\rput(2.35,-1.3){\fontsize{8}{8} $\tau_0$}
\rput(-0.82,2.59){\fontsize{8}{8} $\tau_\infty$}
\rput(-1.00,-2.51){\fontsize{8}{8} $\tau_{-1}$}
\rput[l](2.7,0){\pcline[linewidth=0.7pt, arrowsize=2pt 4]{->}(0,0)(1;0)\Aput{$\beta$}}
\end{pspicture}
%
%
%
%
& \begin{pspicture}[showgrid=false,linewidth=0.5pt,unit=6mm](-0.3,-1.2)(3.4,1.5)
\psarc(0,0){2.4}{0}{360}
%
\psarc[arcsep=0.5pt](1.091,-3.055){2.182}{61.7}{158.0}
\psarc[arcsep=0.5pt](-7.2,0.4){6.8}{-22}{16.3}
\psarc[arcsep=0.5pt](4.8,3.9){5.7}{-163.7}{-118.3}
%
\psline[linewidth=0.5pt, 
arrowsize=3pt 4, 
arrowlength=2, 
arrowinset=0.3] 
{->}(-0.41,0.031)(-0.43,-0.4)
%
\rput{127.0}(1.2,1.0){\fontsize{19}{19} $\cdots$}
\rput{86.0}(-1.4,0.2){\fontsize{19}{19} $\cdots$}
\rput{20.0}(0.58,-1.65){\fontsize{17}{17} $\cdots$}
%
\end{pspicture}
\end{array} $
\caption{The action of $\alpha$ and $\beta$ on a marked tessellation}
\label{fig:alpha_beta}
\end{figure}

\vs

\item $\beta$-move: Locate the ideal triangle of $\tau$ having $\vec{a}$ as one of its sides and situated to the left of $\vec{a}$. Give labels $\tau_0,\tau_\infty,\tau_{-1}$ to the vertices of this triangle, so that $\vec{a}$ runs from $\tau_0$ to $\tau_\infty$. Then $\tau' = \tau$, and the new d.o.e. $\vec{a}'$ is the one running from $\tau_\infty$ to $\tau_{-1}$.

\end{enumerate}

These are called {\em elementary morphisms of $Pt$}. In each of the above cases, we say $\tau_{\rm mark}'$ is obtained from $\tau_{\rm mark}$ by {\em applying} the relevant move.
\end{definition}

\begin{proposition}
\label{prop:T_acts_transitively}
Any morphism of $Pt$ is a finite composition of elementary morphisms.
\end{proposition}

\begin{proof}
Recall that the $\alpha$ action replaces a diagonal of some ideal quadrilateral with the other diagonal. If we forget the choice of d.o.e. and just think of the underlying tessellations, we can call this transformation of tessellations a `flip'. We can associate a flip to any ideal arc of a tessellation. It is easy to see that any two tessellations are related by a finite number of flips, by observing that there exists a finite ideal polygon outside of which the two tessellations coincide. Meanwhile, given any underlying tessellation, we can change the d.o.e. to any ideal arc with any orientation while fixing the underlying tessellation, using a finite number of $\beta$'s and $\alpha^2$'s. Since $\alpha$ induces the flip of the underlying tessellation along the d.o.e., and since we know how to change the d.o.e. to any ideal arc in a given underlying tessellation by a finite number of elementary moves, we conclude that any two marked tessellations are connected by a finite sequence of $\alpha$-moves and $\beta$-moves. 
\end{proof}

By the requirement that there is only one morphism from any object to any object in $Pt$, the $\alpha$-moves and $\beta$-moves satisfy some algebraic relations. Easiest to see are $\beta^3={\rm id}$, $\alpha^4={\rm id}$, and the most famous is the {\em pentagon relation} $(\beta\alpha)^5={\rm id}$ which is not hard to check by pictures. Here, as usual, we read the composition of (i.e. `a word in') the elementary moves from the right; for example, $\alpha\beta\alpha^2$ means applying $\alpha^2$ first, then $\beta$, then $\alpha$. These three relations, together with two certain commutation relations, generate the whole set of algebraic relations of $\alpha,\beta$.
\begin{theorem}[Lochak-Schneps \cite{LSc}]
\label{thm:T}
Any algebraic relation of $\alpha,\beta$ is a consequence of the five relations in \eqref{eq:T_presentation_intro}. The free group generated by $\alpha,\beta$ modded out by these relations is isomorphic to Richard Thompson's group $T$ of dyadic piecewise affine homeomorphisms of $S^1$.
\end{theorem}

\begin{definition}[Funar, Kapoudjian, Sergiescu, and collaborators: the Ptolemy-Thompson group]
\label{def:Ptolemy-Thompson_group}
Let $F_{\rm mark}$ be the free group generated by the symbols $\alpha,\beta$, and let $R_{\rm mark}$ be the normal subgroup generated by the relations as in the RHS of \eqref{eq:T_presentation_intro} (that is, $(\beta\alpha)^5$, $\alpha^4$, etc). Then the quotient group $F_{\rm mark}/R_{\rm mark}$ is called the {\em Ptolemy-Thompson group} and denoted by $T$.
\end{definition}

\begin{remark}
One can replace $(\beta\alpha)^5=1$ by $(\alpha\beta)^5=1$, and $\left[\beta\alpha\beta, \, \alpha^2\beta\alpha^2\beta\alpha\beta\alpha^2\beta^2\alpha^2\right]=1$ by $\left[\beta\alpha\beta, \, \alpha^2\beta^2\alpha^2\beta\alpha\beta\alpha^2\beta\alpha^2\right]=1$. The version \eqref{eq:T_presentation_intro} is the one used by Funar-Sergiescu \cite{FuS}.
\end{remark}

\begin{remark}
As mentioned in \cite{LSc}, the isomorphism between the group of transformations of marked tessellations generated by the elementary moves $\alpha,\beta$ and the Thompson group $T$ is proved by Imbert \cite{I}, without determining the complete set of generating relations.
\end{remark}

Since both $\alpha$-move and $\beta$-move can be applied to any marked tessellation, so can any element of $T$ be. Note that Prop.\ref{prop:T_acts_transitively} says that $T$-action on the set of all marked tessellations is transitive, and that Thm.\ref{thm:T} implies that this action is free.
\begin{corollary}
\label{cor:T_acts_freely}
The group $T$ acts freely transitively on the set of all marked tessellations.
\end{corollary}

We proceed to $Pt_{\rm dot}$. Similar to $Pt$, there are elementary morphisms generating the whole groupoid $Pt_{\rm dot}$, which are combinatorially described. This time, there are infinitely many kinds, which can be grouped into three types.

\begin{definition}
\label{def:elementary_moves_of_Pt_dot}
We describe the {\rm elementary moves} $A_{[j]}$, $T_{[j][k]}$, $P_\gamma$ of $Pt_{\rm dot}$ for mutually distinct triangle labels $j,k\in \mathbb{Q}^\times$  and a permutation $\gamma$ of $\mathbb{Q}^\times$. A morphism $[\tau_{\rm dot}, \tau'_{\rm dot}]$ of $Pt_{\rm dot}$ is labeled by one of these names if it falls into the relevant description as follows:

\vs

\begin{itemize}
\item[\rm 1)] $A_{[j]}$-move: The dotted tessellation $\tau'_{\rm dot}$ is obtained from $\tau_{\rm dot}$ by moving the dot $\bullet$ (i.e. the distinguished corner) of the triangle of $\tau$ labeled by $j\in \mathbb{Q}^\times$ counterclockwise to the next corner in that triangle, while leaving all other information intact. 

\vs

\item[\rm 2)] $T_{[j][k]}$-move: The triangles of $\tau_{\rm dot}$ labeled by $j$ and $k$ (where $j\neq k$) must share exactly one side and the dots of these two triangles are placed exactly as in the LHS of Fig.\ref{fig:action_of_T_jk}, relative to the common edge of the two triangles. Then $\tau_{\rm dot}'$ is obtained from $\tau_{\rm dot}$ by replacing the common edge of the triangles labeled by $j, k$ by the other diagonal arc of the ideal quadrilateral formed by these two triangles, and setting the new dots and labels as in the RHS of Fig.\ref{fig:action_of_T_jk}, as if we rotate {\em clockwise} the diagonal arc of the quadrilateral while letting the dots $\bullet$ and triangle labels $[j],[k]$ be `floating' and thus pushed accordingly by the rotating arc, while leaving all the other information intact. 

\vs

\item[\rm 3)] $P_{\gamma}$-move: $\tau'_{\rm dot}$ is obtained from $\tau_{\rm dot}$ by relabeling the triangles, while leaving all other information intact. A triangle labeled by $j$ in $\tau_{\rm dot}$ is labeled by $\gamma(j)$ in $\tau'_{\rm dot}$. The $\gamma(j)$-triangle of $\tau'_{\rm dot}$ inherits the dotting rule of the $j$-triangle of $\tau_{\rm dot}$.

\end{itemize}

\vs

These are called {\em elementary morphisms} of $Pt_{\rm dot}$. In each of the above cases, we say that $\tau'_{\rm dot}$ is obtained from $\tau_{\rm dot}$ by {\em applying} the relevant move.

\end{definition}

\begin{figure}[htbp!]
$\begin{array}{llll}
\begin{pspicture}[showgrid=false,linewidth=0.5pt,unit=6mm](-1.3,-1.2)(2.7,1.3)
\psarc(0,0){2.4}{0}{360}
%
\psarc[arcsep=0.5pt](1.091,-3.055){2.182}{61.7}{158.0}
\psarc[arcsep=0.5pt](-7.2,0.4){6.8}{-22}{16.3}
\psarc[arcsep=0.5pt](4.8,3.9){5.7}{-163.7}{-118.3}
%
%
\rput{127.0}(1.2,1.0){\fontsize{19}{19} $\cdots$}
\rput{86.0}(-1.4,0.2){\fontsize{19}{19} $\cdots$}
\rput{20.0}(0.58,-1.65){\fontsize{17}{17} $\cdots$}
%
\rput(-0.5,-1.3){\fontsize{11}{11} $\bullet$}
%
\rput(0.1,-0.3){\fontsize{11}{11} $[\, j\, ]$}
%
\rput[l](2.7,0){\pcline[linewidth=0.7pt, arrowsize=2pt 4]{->}(0,0)(1;0)\Aput{\,$A_{[j]}$}}
\end{pspicture}
%
%
%
%
& \begin{pspicture}[showgrid=false,linewidth=0.5pt,unit=6mm](-0.8,-1.2)(2.7,1.3)
\psarc(0,0){2.4}{0}{360}
%
\psarc[arcsep=0.5pt](1.091,-3.055){2.182}{61.7}{158.0}
\psarc[arcsep=0.5pt](-7.2,0.4){6.8}{-22}{16.3}
\psarc[arcsep=0.5pt](4.8,3.9){5.7}{-163.7}{-118.3}
%
\rput(1.1,-0.7){\fontsize{11}{11} $\bullet$}
%
\rput(0.1,-0.3){\fontsize{11}{11} $[\, j\, ]$}
%
\rput{127.0}(1.2,1.0){\fontsize{19}{19} $\cdots$}
\rput{86.0}(-1.4,0.2){\fontsize{19}{19} $\cdots$}
\rput{20.0}(0.58,-1.65){\fontsize{17}{17} $\cdots$}
\end{pspicture}
%
%
%
%
& \begin{pspicture}[showgrid=false,linewidth=0.5pt,unit=6mm](-0.7,-1.2)(3.3,1.3)
\psarc(0,0){2.4}{0}{360}
%
\psarc[arcsep=0.5pt](0.343,-4.457){3.771}{61.7}{126.4}
\psarc[arcsep=0.5pt](-3.36,0.48){2.4}{-53.6}{37.4}
\psarc[arcsep=0.5pt](-0.185,2.862){1.569}{-142.6}{-29}
\psarc[arcsep=0.5pt](3.220,0.937){2.341}{151}{-118.3}
%
\psarc[arcsep=0.5pt](7.2,8.4){10.8}{-142.6}{-118.3}
%
\rput(-0.95,1.00){\fontsize{10}{10} $\bullet$}
\rput(0.72,1.42){\fontsize{10}{10} $\bullet$}
%
\rput(-0.5,-0.1){\fontsize{10}{10} $[\, j\, ]$}
\rput(0.3,0.8){\fontsize{10}{10} $[\, k\, ]$}
%
\rput{104.0}(1.7,0.5){\fontsize{15}{15} $\cdots$}
\rput{5.0}(-0.10,1.94){\fontsize{15}{15} $\cdots$}
\rput{85.0}(-1.7,0.3){\fontsize{15}{15} $\cdots$}
\rput{4.0}(0.08,-1.50){\fontsize{17}{17} $\cdots$}
%
\rput[l](2.8,0){\pcline[linewidth=0.7pt, arrowsize=2pt 4]{->}(0,0)(1;0)\Aput{$T_{[j][k]}$}}
\end{pspicture}
%
%
%
%
& \begin{pspicture}[showgrid=false,linewidth=0.5pt,unit=6mm](-0.3,-1.2)(3.4,1.3)
\psarc(0,0){2.4}{0}{360}
%
\psarc[arcsep=0.5pt](0.343,-4.457){3.771}{61.7}{126.4}
\psarc[arcsep=0.5pt](-3.36,0.48){2.4}{-53.6}{37.4}
\psarc[arcsep=0.5pt](-0.185,2.862){1.569}{-142.6}{-29}
\psarc[arcsep=0.5pt](3.220,0.937){2.341}{151}{-118.3}
%
%
\psarc[arcsep=0.5pt](-8.8,7.733){11.467}{-53.0}{-29.7}
%
\rput(-1.00,1.30){\fontsize{10}{10} $\bullet$}
\rput(0.67,1.10){\fontsize{10}{10} $\bullet$}
%
\rput(-0.50,0.8){\fontsize{10}{10} $[\, j\, ]$}
\rput(0.3,-0.1){\fontsize{10}{10} $[\, k\, ]$}
%
\rput{104.0}(1.7,0.5){\fontsize{15}{15} $\cdots$}
\rput{5.0}(-0.10,1.94){\fontsize{15}{15} $\cdots$}
\rput{85.0}(-1.7,0.3){\fontsize{15}{15} $\cdots$}
\rput{4.0}(0.08,-1.50){\fontsize{17}{17} $\cdots$}
\end{pspicture}
\end{array} $
\\
\begin{subfigure}[b]{0.48\textwidth}
\caption{The action of $A_{[j]}$ on a dotted tessellation}
\label{fig:action_of_A_j}
\end{subfigure}
\hfill
\begin{subfigure}[b]{0.5\textwidth}
\caption{The action of $T_{[j][k]}$ on a dotted tessellation}
\label{fig:action_of_T_jk}
\end{subfigure}
\vspace{-2mm}
\caption{Some elementary morphisms of $Pt_{\rm dot}$}
\label{fig:action_on_dotted_tessellations}
\end{figure}
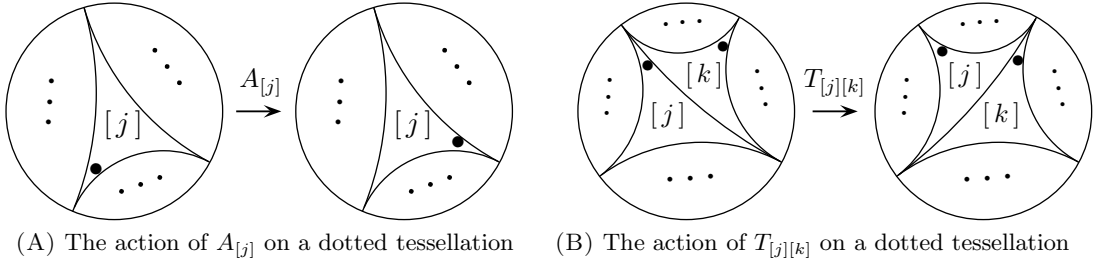

\begin{remark}
The alphabet $T$ for $T_{[j][k]}$ has nothing to do with the Ptolemy-Thompson group $T$, and this notational coincidence is just an unfortunate one.
\end{remark}
Note that the $T_{[j][k]}$-move is applicable only to certain dotted tessellations $\tau_{\rm dot}$. 
Thus, not all words of elementary moves are applicable to all dotted tessellations. So, an algebraic relation $({\rm word}_1) = ({\rm word}_2)$ means that whenever $({\rm word}_1)$ can be applied to some $\tau_{\rm dot}$ then so can $({\rm word}_2)$ be applied to $\tau_{\rm dot}$, and they yield the same result $\tau'_{\rm dot}$.

\vs

Analogously to Prop.\ref{prop:T_acts_transitively}, it is not hard to see the following:
\begin{proposition}
\label{prop:morphism_as_composition}
Any morphism of $Pt_{\rm dot}$ is a finite composition of elementary morphisms.
\end{proposition}

What is not so obvious is the complete generating set of algebraic relations among these elementary moves.

\begin{theorem}
\label{thm:algebraic_relations_of_elementary_moves}
All the nontrivial algebraic relations among the elementary moves of $Pt_{\rm dot}$ are the consequences of the four types of nontrivial relations in \eqref{eq:K_presentation_intro} (where we replace $A_j$, $T_{jk}$ of \eqref{eq:K_presentation_intro} by $A_{[j]}$, $T_{[j][k]}$), and the following `trivial relations':
\begin{align*}
{\rm [permutation]} \quad & P_{\rm id} = {\rm id}, \quad P_{\gamma_1} P_{\gamma_2} = P_{\gamma_1 \circ \gamma_2}, \\
{\rm [index~change]} \quad & P_\gamma A_{[j]} = A_{[\gamma(j)]} P_\gamma, \quad
P_\gamma T_{[j][k]} = T_{[\gamma(j) \, \gamma(k)]} P_\gamma, 
\\
{\rm [commutativity]} \quad & T_{[j][k]} T_{[\ell][m]} = T_{[\ell][m]} T_{[j][k]}, \quad
T_{[j][k]} A_{[\ell]} = A_{[\ell]} T_{[j][k]}, \quad
A_{[j]} A_{[k]} = A_{[k]} A_{[j]}, 
\end{align*}
where $\gamma_1,\gamma_2,\gamma$ are permutations of $\mathbb{Q}^\times$, and $j,k,\ell,m \in \mathbb{Q}^\times$ are mutually distinct.
\end{theorem}
The statement of this theorem is due to Kashaev (see e.g. \cite{Kash00}), and Teschner claimed a proof in \cite{T}; a more complete proof can be found in the author's another paper \cite{Ki14}. Now we define a group $K$ with these generators and relations.

\begin{definition}
\label{def:Kashaev_group}
Let $F_{\rm dot}$ be the free group generated by
\begin{align}
\label{eq:K_generators}
\{ A_{[j]}, T_{[j][k]}, P_\gamma : j,k\in \mathbb{Q}^\times, j\neq k, \mbox{ and } \gamma \mbox{ is a permutation of } \mathbb{Q}^\times\},
\end{align}
and $R_{\rm dot}$ be the normal subgroup of $F_{\rm dot}$ generated by all the relations mentioned in Thm.\ref{thm:algebraic_relations_of_elementary_moves} (that is, $A_j^3$, $T_{k\ell} T_{jk} (T_{jk} T_{j\ell} T_{k\ell})^{-1}$, etc). The {\em Kashaev group} $K$ is defined as $F_{\rm dot}/R_{\rm dot}$.
\end{definition}

\begin{remark}
The generating set \eqref{eq:K_generators} should have been used in \eqref{eq:K_presentation_intro}. In fact, we can restrict $\gamma$ in \eqref{eq:K_generators} to certain types of permutations of $\mathbb{Q}^\times$, but let us allow any permutation here. 
\end{remark}

\begin{remark}
A more general group $G_I$, defined in the same way as $K$, with $\mathbb{Q}^\times$ replaced by any index set $I$, was defined in {\em Frenkel-Kim} \cite{FrKi}, and called the Kashaev group there. 
\end{remark}

The group $K$ can be thought of as the {\em formal} group of transformations of dotted tessellations, as its elements may not be applied to all dotted tessellations. Still, by Thm.\ref{thm:algebraic_relations_of_elementary_moves}, the `action' of $K$ on the set of all dotted tessellations is `free', in the following sense:
\begin{corollary}
\label{cor:K_action_is_free}
If $g\in K$ fixes one dotted tessellation, i.e. $g.\tau_{\rm dot} = \tau_{\rm dot}$ for some $\tau_{\rm dot}$, then $g = 1$. Moreover, if $g, g' \in K$ are applicable to some $\tau_{\rm dot}$ and if $g.\tau_{\rm dot} = g'.\tau_{\rm dot}$ holds, then $g=g'$. Therefore any element of $K$ which can be applied to at least one dotted tessellation is completely characterized by its action on a dotted tessellation which it can be applied to.
\end{corollary}

\subsection{The natural functor $\mcal{F}: Pt \to Pt_{\rm dot}$}
\label{subsec:mcal_F}

Recall from \S\ref{sec:introduction} that the universal Ptolemy groupoid $Pt$ and the Ptolemy-Thompson group $T$ pertain to the Chekhov-Fock quantization of universal Teichm\"uller space, while the dotted universal Ptolemy groupoid $Pt_{\rm dot}$ and the Kashaev group $K$ are for the Kashaev quantization. To relate these two quantizations, we shall construct a natural functor
\begin{align}
\label{eq:map_from_marked_to_dotted}
\mcal{F} : Pt \to Pt_{\rm dot},
\end{align}
which will lead in the following subsection to a natural group homomorphism ${\bf F}: T\to K$. Recall that in each of the two categories $Pt$ and $Pt_{\rm dot}$, from any object to any object there is exactly one morphism. Therefore if we specify the images of objects of $Pt$ under $\mcal{F}$, the images of morphisms of $Pt$ under $\mcal{F}$ are then determined, making $\mcal{F}$ a functor. The question is how naturally and uniquely we can choose the images of the objects. We shall make such a choice using the following object:
\begin{definition}
\label{def:M}
The {\em asymptotically rigid mapping class group of $\mathbb{D}$}, denoted by $M$, is the group of all asymptotically rigid mapping classes of $\mathbb{D}$, defined in Def.\ref{def:asymptotically_rigid}. Its elements are denoted by $[\varphi]$, for an asymptotically rigid homeomorphism $\varphi : \mathbb{D} \to \mathbb{D}$ (Def.\ref{def:asymptotically_rigid}), and the composition rule is given by $[\varphi]\circ[\psi] = [\varphi\circ \psi]$.
\end{definition}
\begin{remark}
\label{rem:PPSL2Z}
Elements of $M$ are completely determined by its restriction on $S^1$, hence $M$ can be viewed as a subgroup of $Homeo^+(S^1)$, the group of orientation-preserving homeomorphisms of $S^1$. This subgroup of $Homeo^+(S^1)$ is referred to as ${\rm PPSL}(2,\mathbb{Z})$, standing for the `piecewise-${\rm PSL}(2,\mathbb{Z})$ homeomorphisms'.
\end{remark}
The group $M$ is a certain `universal' analog of the mapping class group $M(\Sigma)$ of a finite-type punctured Riemann surface $\Sigma$, which is a central and natural object in the story of $\Sigma$.  
The reason why we use boundary-fixing homotopies of $\mathbb{D}$ in the definition of $M$ is because we use the tessellations of $\mathbb{D}$ as analogs of ideal triangulations of $\Sigma$ whose vertices are punctures, so that the vertices of tessellations of $\mathbb{D}$ act like punctures of $\Sigma$.
Note that not every mapping class of $\mathbb{D}$ (Def.\ref{def:mapping_class}) preserves the set of objects of $Pt$.
We shall see that $M$ is precisely the group of all mapping classes of $\mathbb{D}$ preserving the set of objects of $Pt$. 
So $M$ can be viewed as the mapping class group of our `surface $\mathbb{D}$' having $Pt$ as its `groupoid of ideal triangulations', hence expect that it will play as natural and central role in our setting which uses the groupoid $Pt$, as the genuine mapping class group $M(\Sigma)$ does for $\Sigma$. 

\begin{remark}
See \cite{Penner2} for a discussion of why $Pt$ is a reasonable groupoid to start with.
\end{remark}

We first observe that $M$ acts naturally on the objects of $Pt$ and on those of $Pt_{\rm dot}$. Namely, regard a marked tessellation as a tessellation with an arrow on one edge, and a dotted tessellation as a tessellation with dots and triangle labels written on ideal triangles. Then one can see how asymptotically rigid mapping classes transform these graphical data. We will shortly see that this $M$-action on the objects of $Pt$ is free and transitive, while the $M$-action on the objects of $Pt_{\rm dot}$ is free. Later, we also prove that this leads to a group anti-isomorphism between $M$ and $T$. Then, by requiring the equivariance under the $M$-actions, the map $\mcal{F}$ on the set of objects of $Pt$ is completely determined by the image of one object. For example, if $\mcal{F}(\tau^*_{\rm mark})$ is chosen, then we know $\mcal{F}(\tau_{\rm mark})$ for any marked tessellation $\tau_{\rm mark}$. Choice of the image of one object can be called an {\em initial condition} for $\mcal{F}$. We will use the following initial condition
\begin{align}
\label{eq:our_initial_condition_for_F}
\mcal{F}(\tau^*_{\rm mark}) = \tau^*_{\rm dot},
\end{align}
just by convenience; the underlying tessellation of the standard dotted tessellation $\tau^*_{\rm dot}$ coincides with that of the standard marked tessellation $\tau^*_{\rm mark}$, and the dotting rule and the triangle-labeling rule of $\tau^*_{\rm dot}$ can easily be understood by means of the `middle vertices' (Def.\ref{def:dotted_tessellations}). But we may as well choose the image of $\tau^*_{\rm mark}$ to be any other dotted tessellation $\tau_{\rm dot}$, and get another $\mcal{F}$, hence another ${\bf F}$. We will see in Prop.\ref{prop:essential_uniqueness_of_bold_F} that the maps ${\bf F}: T\to K$ resulting from different choices of $\mcal{F}(\tau^*_{\rm mark})$ differ by conjugation in $K$, hence are `equivalent' to each other. Thus we justify the naturalness and essential uniqueness of the to-be-constructed map ${\bf F}: T\to K$.

\vs

We now prove the statements promised in the previous paragraph. We first establish how to record a marked tessellation by their vertices.

\begin{definition}[vertex function of a marked tessellation]
\label{def:vertex_function}
For a marked tessellation $(\tau,\vec{a})$, we construct a bijection $j \mapsto \tau_j$ from extended rationals to extended rationals, called the {\em vertex function} of $(\tau,\vec{a})$, by the following `inductive' process:

\begin{enumerate}
\item Let $\tau_0$ and $\tau_\infty$ be the two extended rational numbers such that $\mu(\tau_0)$ and $\mu(\tau_1)$ are the starting point and the ending point of the d.o.e. $\vec{a}$, where $\mu$ is the Cayley transform.

\item Among the ideal triangles of $\tau$, take the unique ideal triangle having $\vec{a}$ as one of its sides and situated to the right (resp. left) of $\vec{a}$, and let $\tau_1$ (resp. $\tau_{-1}$) be the extended rational such that $\mu(\tau_1)$ (resp. $\mu(\tau_{-1})$) is the third vertex of this triangle. 

\item For an ideal triangle of $\tau$ other than the two triangles appearing in {\rm (2)}, if two of its vertices are identified as $\mu(\tau_{\frac{a}{b}})$ and $\mu(\tau_{\frac{c}{d}})$ for some extended rationals $\frac{a}{b}, \frac{c}{d}$ in their reduced expressions but the third one is not identified yet, we let $\tau_{\frac{a+c}{b+d}}$ be the extended rational such that $\mu(\tau_{\frac{a+c}{b+d}})$ is the third vertex of this triangle. `Repeat' this step. 

\end{enumerate}

\end{definition}
For example, for the marked tessellation $(\tau,\vec{a})$ in Fig.\ref{subfig:marked_tessellation_ex}, we have $\tau_0 = -1$, $\tau_\infty = 0$, $\tau_1 = -\frac{1}{2}$, $\tau_{-1} = -2$. It is easy to see:

\begin{lemma}
\label{lem:vertex_function_determines_marked_tessellation}
The vertex function completely determines a marked tessellation.
\end{lemma}

We can now prove:
\begin{proposition}
\label{prop:M-actions}
The natural $M$-action on the objects of $Pt$ is free and transitive, and the natural $M$-action on the objects of $Pt_{\rm dot}$ is free.
\end{proposition}

\begin{proof}
Notice that an asymptotically rigid mapping class of $\mathbb{D}$ is completely determined by its restriction on $S^1$, and therefore by that on the rational points of $S^1$, because rational points are dense in $S^1$. This, together with Lem.\ref{lem:vertex_function_determines_marked_tessellation}, implies that the $M$-action on $Pt$ is free. Also, it is easy to see that the vertex function $j\mapsto \tau_j$ of a marked tessellation $(\tau,\vec{a})$ is piecewise-${\rm PSL}(2,\mathbb{Z})$ with finitely many breakpoints, which are rational. Therefore, with the help of Lem.\ref{lem:vertex_function_determines_marked_tessellation}, one can show that the standard marked tessellation is connected to any marked tessellation by the action of an element of $M$, thus implying the transitivity of the $M$-action on $Pt$. We can also easily show that the $M$-action on $Pt_{\rm dot}$ is free; for any dotted tessellation, the only element of $M$ fixing the underlying tessellation and the triangle labels is the identity.
\end{proof}

Since both $T$ and $M$ act freely transitively on the set of objects of $Pt$ (Cor.\ref{cor:T_acts_freely}, Prop.\ref{prop:M-actions}), we can construct a set bijection $M\to T$ by choosing one object of $Pt$, although this is not a group isomorphism. For example, if we use $\tau^*_{\rm mark}$ as a reference point, then for each $[\varphi] \in M$, the morphism $[\tau^*_{\rm mark}, [\varphi].\tau^*_{\rm mark}]$ corresponds to some element of $T$, giving a bijection $M \to T$.
\begin{proposition}
\label{prop:anti_isomorphism}
This bijection $M\to T$ is an anti-isomorphism of groups.
\end{proposition}
We postpone a proof until \S\ref{subsec:bold_F} where we collect some necessary notations. One can also see that $M$ is the group of {\em all} mapping classes of $\mathbb{D}$ inducing elements of $T$ by the action on $\tau^*_{\rm mark}$, as both the actions of elements of $T$ and mapping classes of $\mathbb{D}$ on $\tau^*_{\rm mark}$ are completely determined by the action on the rational points on $S^1$. Therefore the Ptolemy-Thompson group $T$ can be viewed as a substitute for $M$, and sometimes $T$ is referred to as the `asymptotically rigid mapping class group'. The advantage of $T$ over $M$ is that the elements of $T$ have combinatorial descriptions, namely as transformations of marked tessellations of $\mathbb{D}$.

\begin{remark}
It seems that one can obtain a proof of Prop.\ref{prop:anti_isomorphism} also by collecting some results of \cite{Penner2} and \cite{I}, which use the notion of the `(universal) Ptolemy group'.
\end{remark}

\begin{remark}
Meanwhile, the Kashaev group $K$ is much `larger' than $T$ or $M$, because not all its elements are induced by asymptotically rigid mapping classes of $\mathbb{D}$. Using the to-be-constructed natural injective map ${\bf F} : T \to K$, we may say that the group ${\bf F}(T)$ is a substitute for $M$, realized combinatorially as a group of transformations of dotted tessellations.
\end{remark}

Now, by requiring the natural $M$-actions on $Pt$ and $Pt_{\rm dot}$ be preserved, one gets a natural functor $\mcal{F} : Pt \to Pt_{\rm dot}$, which is uniquely determined by the choice of an initial condition.

\begin{proposition}[construction of $\mcal{F}$]
\label{prop:well-definedness_of_F}
For any marked tessellation $\tau_{\rm mark}^\circ$ and any dotted tessellation $\tau_{\rm dot}^\circ$, there is a unique functor $\mcal{F} : Pt \to Pt_{\rm dot}$ \eqref{eq:map_from_marked_to_dotted} which is equivariant under the $M$-actions on the objects and sends $\tau_{\rm mark}^\circ$ to $\tau_{\rm dot}^\circ$. This functor is injective on the set of objects.
\end{proposition}

\begin{proof}
Suppose $\mcal{F}(\tau_{\rm mark}^\circ) = \tau_{\rm dot}^\circ$. The $M$-equivariance implies $\mcal{F}([\varphi].\tau_{\rm mark}^\circ) = [\varphi].\tau_{\rm dot}^\circ$ for any $[\varphi] \in M$. By the freeness and the transitivity of the $M$-action on $Pt$, any marked tessellation can be written as $[\varphi].\tau_{\rm mark}^\circ$ for a unique $[\varphi]\in M$, so the image of each marked tessellation under $\mcal{F}$ is determined and well-defined. As mentioned already, the images of morphisms of $Pt$ under $\mcal{F}$ are then determined too. If $[\varphi_1].\tau_{\rm mark}^\circ \neq [\varphi_2].\tau_{\rm mark}^\circ$, then $[\varphi_1]\neq [\varphi_2]$, so $[\varphi_1].\tau_{\rm dot}^\circ \neq [\varphi_2].\tau_{\rm dot}^\circ$ by the freeness of the $M$-action on $Pt_{\rm dot}$. Hence the injectivity of $\mcal{F}$ on the set of objects.
\end{proof}

A concrete description of the functor $\mcal{F}$ on the set of objects of $Pt$ is available, if we use the initial condition \eqref{eq:our_initial_condition_for_F}.  To get an idea, observe that Fig.\ref{subfig:dotted_tessellation_ex} is the image of Fig.\ref{subfig:marked_tessellation_ex} under $\mcal{F}$.
\begin{proposition}[concrete description of a particular $\mcal{F}$]
\label{prop:concrete_description_of_F}
The functor $\mcal{F}:Pt \to Pt_{\rm dot}$ constructed in Prop.\ref{prop:well-definedness_of_F} with the initial condition \eqref{eq:our_initial_condition_for_F} can be explicitly described. For any object $(\tau,\vec{a})$ of $Pt$, its image under $\mcal{F}$ is given by $(\tau,D,L)$, where the dotting rule $D$ and the triangle-labeling rule $L$ are as follows. 

\vs

Let $j \mapsto \tau_j$ be the vertex function for $(\tau,\vec{a})$ (Def.\ref{def:vertex_function}). Then the vertices of any ideal triangle of $\tau$ are $\mu(\tau_{\frac{a}{b}})$, $\mu(\tau_{\frac{a+c}{b+d}})$, $\mu(\tau_{\frac{c}{d}})$, for some extended rationals $\frac{a}{b}, \frac{c}{d}$ in their reduced expressions. We choose the corner $\mu(\tau_{\frac{a+c}{b+d}})$ as the distinguished corner of this triangle (i.e. put the dot in that corner, for this triangle); this is the dotting rule $D$. And we label this triangle by $\frac{a+c}{b+d}$; this is the labeling rule $L$ for triangles.
\end{proposition}

\subsection{The natural map ${\bf F}: T \to K$}
\label{subsec:bold_F}

Now, out of the constructed functor $\mcal{F}: Pt \to Pt_{\rm dot}$, we should build a map ${\bf F}: T\to K$. In the end we will use \eqref{eq:our_initial_condition_for_F} as the initial condition for our $\mcal{F}$, but at the moment we can stay more general, and just assume that we use some initial condition $\mcal{F}(\tau_{\rm mark}^\circ) = \tau_{\rm dot}^\circ$ as in Prop.\ref{prop:well-definedness_of_F}. From \S\ref{subsec:T_and_K}, we can deduce that morphisms of $Pt$ can be represented as elements of $T$, and morphisms of $Pt_{\rm dot}$ as elements of $K$:
\begin{proposition}
\label{prop:morphisms_to_T_and_K}
For each morphism $[\tau_{\rm mark}, \tau'_{\rm mark}]$ of $Pt$ there is a unique element $g$ of $T$ such that $\tau'_{\rm mark} = g.\tau_{\rm mark}$. For each morphism $[\tau_{\rm dot}, \tau'_{\rm dot}]$ of $Pt_{\rm dot}$ there is a unique element $h$ of $K$ such that $\tau'_{\rm dot} = h.\tau_{\rm dot}$.
\end{proposition}
\begin{definition}
\label{def:leadsto}
In such situations, we say that $[\tau_{\rm mark}, \tau'_{\rm mark}]$ {\em corresponds to $g$} and that $[\tau_{\rm dot}, \tau'_{\rm dot}]$ {\em corresponds to $h$}, and write $[\tau_{\rm mark},\tau'_{\rm mark}] \leadsto g$, and $[\tau_{\rm dot}, \tau'_{\rm dot}] \leadsto h$.
\end{definition}
\begin{proposition}
\label{prop:composition_to_multiplication}
These correspondences carries composition of morphisms to group multiplications in $T$ and $K$.
\end{proposition}
The existence proof for Prop.\ref{prop:morphisms_to_T_and_K} is by transitivity of the $T$-action on objects of $Pt$ (Prop.\ref{prop:T_acts_transitively}) and of the $K$-action on objects of $Pt_{\rm dot}$ (Prop.\ref{prop:morphism_as_composition}), and the uniqueness proof is by the freeness of those actions (Cor.\ref{cor:T_acts_freely}, Cor.\ref{cor:K_action_is_free}). Prop.\ref{prop:composition_to_multiplication} is easy to show. 

\vs

In order to get a well-defined map ${\bf F}: T\to K$ from the functor $\mcal{F}: Pt\to Pt_{\rm dot}$, we should understand which morphisms of $Pt$ (resp. $Pt_{\rm dot}$) correspond to a same element of $T$ (resp. $K$). Here, the $M$-actions discussed in \S\ref{subsec:mcal_F} play a role again.
\begin{definition}
When two morphisms $[\tau_{\rm mark}, \tau'_{\rm mark}]$ and $[\tau_{\rm mark}'', \tau'''_{\rm mark}]$ of $Pt$ correspond to a same element of $T$, we write
$
[\tau_{\rm mark}, \tau'_{\rm mark}] \simeq_T [\tau_{\rm mark}'', \tau'''_{\rm mark}].
$
When two morphisms $[\tau_{\rm dot}, \tau'_{\rm dot}]$ and $[\tau''_{\rm dot}, \tau'''_{\rm dot}]$ of $Pt_{\rm dot}$ correspond to a same element of $K$, we write $[\tau_{\rm dot}, \tau'_{\rm dot}] \simeq_K [\tau''_{\rm dot}, \tau'''_{\rm dot}]$.
\end{definition}
\begin{proposition}
\label{prop:morphisms_lead_to_same_element}
We have $[\tau_{\rm mark}, \tau'_{\rm mark}] \simeq_T [\tau_{\rm mark}'', \tau'''_{\rm mark}]$ if and only if there is $[\varphi] \in M$ such that $\tau''_{\rm mark} = [\varphi].\tau_{\rm mark}$ and $\tau'''_{\rm mark} = [\varphi].\tau'_{\rm mark}$. We have $[\tau_{\rm dot}, \tau'_{\rm dot}] \simeq_K [\tau''_{\rm dot}, \tau'''_{\rm dot}]$ if there is $[\psi] \in M$ such that $\tau''_{\rm dot} = [\psi].\tau_{\rm dot}$ and $\tau'''_{\rm dot} = [\psi].\tau'_{\rm dot}$.
\end{proposition}
\begin{proof}
The `if' statements holds because homeomorphisms preserve the combinatorial/graphical data. For example, if $[\tau_{\rm mark},\tau'_{\rm mark}] \leadsto \alpha$, then one notes that $[[\varphi].\tau_{\rm mark}, [\varphi].\tau'_{\rm mark}] \leadsto \alpha$,  for any $[\varphi] \in M$. Now, suppose $[\tau_{\rm mark}, \tau'_{\rm mark}]$ and $[\tau''_{\rm mark}, \tau'''_{\rm mark}]$ correspond to a same element $g$ of $T$. By the transitivity of the $M$-action on $Pt$, we can find $[\varphi] \in M$ such that $\tau_{\rm mark}'' = [\varphi].\tau_{\rm mark}$. Then $[ [\varphi].\tau_{\rm mark}, \tau'''_{\rm mark} ] $ and $[ [\varphi].\tau_{\rm mark}, [\varphi].\tau'_{\rm mark} ]$ correspond to the same element $g$. This means $\tau'''_{\rm mark} = g. ([\varphi].\tau_{\rm mark})$ and $[\varphi].\tau'_{\rm mark} = g.([\varphi].\tau_{\rm mark})$, therefore $\tau'''_{\rm mark} = [\varphi]. \tau'_{\rm mark}$.
\end{proof}

With the help of the following lemma, we now present a promised proof of Prop.\ref{prop:anti_isomorphism}.

\begin{lemma}
\label{lem:T_M_commutes}
For any $g\in T$ and $[\varphi] \in M$, one has $g.([\varphi].\tau_{\rm mark}) = [\varphi].(g.\tau_{\rm mark})$.
\end{lemma}
\begin{proof}
Observe $[\tau_{\rm mark}, \, [\varphi]^{-1}.(g.([\varphi].\tau_{\rm mark}))] \simeq_T [[\varphi].\tau_{\rm mark}, \, g.([\varphi].\tau_{\rm mark})] \leadsto g$ by Prop.\ref{prop:morphisms_lead_to_same_element} and Def.\ref{def:leadsto}. Since $[\tau_{\rm mark}, \, g.\tau_{\rm mark}] \leadsto g$, by Prop.\ref{prop:morphisms_lead_to_same_element} one has $[\varphi]^{-1}.(g.([\varphi].\tau_{\rm mark})) = g.\tau_{\rm mark}$, so by applying $[\varphi]$ from the left ones get the desired result.
\end{proof} 

\begin{proof}[Proof of Propf.\ref{prop:anti_isomorphism}]
For any $g\in T$, denote by $[\varphi_g]$ the unique element of $M$ such that $g.\tau_{\rm mark}^* = [\varphi_g].\tau_{\rm mark}^*$. One should prove $[\varphi_{gh}] = [\varphi_h] \circ [\varphi_g]$. Note that
\begin{align*}
& [\tau_{\rm mark}^*, [\varphi_{gh}].\tau_{\rm mark}^*]
= [\tau_{\rm mark}^*, (gh).\tau_{\rm mark}^*]
= [\tau_{\rm mark}^*, g.(h.\tau_{\rm mark}^*)] 
= [\tau_{\rm mark}^*, g.([\varphi_h].\tau_{\rm mark}^*)] \\
& \stackrel{Lem.{\fontsize{7}{7} \ref{lem:T_M_commutes}} }{=} [\tau_{\rm mark}^*, [\varphi_h].(g.\tau_{\rm mark}^*)]
= [\tau_{\rm mark}^*, [\varphi_h].([\varphi_g].\tau_{\rm mark}^*)]
= [\tau_{\rm mark}^*, ( [\varphi_h]\circ [\varphi_g]).\tau_{\rm mark}^*],
\end{align*}
thus $[\varphi_{gh}].\tau_{\rm mark}^* = ( [\varphi_h]\circ [\varphi_g]).\tau_{\rm mark}^*$, and therefore $[\varphi_{gh}] = [\varphi_h] \circ [\varphi_g]$ as desired.
\end{proof}

Finally we construct the sought-for map ${\bf F}: T\to K$, using the above results.
\begin{proposition}
\label{prop:bold_F}
Suppose we have a functor $\mcal{F} : Pt \to Pt_{\rm dot}$ \eqref{eq:map_from_marked_to_dotted} as in Prop.\ref{prop:well-definedness_of_F}. For each morphism $[\tau_{\rm mark}, \tau'_{\rm mark}]$ of $Pt$, let $g \in T$ and $h \in K$ be such that
\begin{align}
\nonumber
[\tau_{\rm mark}, \tau'_{\rm mark}] \leadsto g \quad\mbox{and}\quad [\mcal{F}(\tau_{\rm mark}), \mcal{F}(\tau'_{\rm mark})] \leadsto h
\end{align}
(see Def.\ref{def:leadsto}). Then the map ${\bf F} : T\to K$ given by ${\bf F}(g) = h$ is a well-defined injective group homomorphism.
\end{proposition}
\begin{proof}
For any $g\in T$ we can find a morphism of $Pt$ corresponding to $g$, and thus we get some definition of a set map ${\bf F}$. To prove well-definedness, suppose the morphisms $[\tau_{\rm mark}, \tau'_{\rm mark}] $ and $[\tau''_{\rm mark}, \tau'''_{\rm mark}]$ of $Pt$ correspond to a same element $g\in T$. By the first statement of Prop.\ref{prop:morphisms_lead_to_same_element} there is $[\varphi] \in M$ such that $\tau''_{\rm mark} = [\varphi].\tau_{\rm mark}$ and $\tau'''_{\rm mark} = [\varphi].\tau'_{\rm mark}$, so
$$
[\mcal{F}(\tau''_{\rm mark}), \, \mcal{F}(\tau'''_{\rm mark})] = [\mcal{F}([\varphi].\tau_{\rm mark}), \, \mcal{F}([\varphi].\tau'_{\rm mark})]
= [\, [\varphi].( \mcal{F}(\tau_{\rm mark}) ), \, [\varphi]. ( \mcal{F}(\tau'_{\rm mark}) )\,],
$$
by the $M$-equivariance of $\mcal{F}$. Now, by the second statement of Prop.\ref{prop:morphisms_lead_to_same_element}, the morphisms $[\, [\varphi].( \mcal{F}(\tau_{\rm mark}) ), \, [\varphi]. ( \mcal{F}(\tau'_{\rm mark}) )\,]$ and $[\mcal{F}(\tau_{\rm mark}), \mcal{F}(\tau'_{\rm mark})]$ of $Pt_{\rm dot}$ correspond to a same element in $K$. So we showed $[\mcal{F}(\tau''_{\rm mark}), \, \mcal{F}(\tau'''_{\rm mark})]$ and $[\mcal{F}(\tau_{\rm mark}), \mcal{F}(\tau'_{\rm mark})]$ correspond to a same element in $K$, proving the well-definedness of ${\bf F}$.

\vs

We can see that ${\bf F}$ is a group homomorphism because the group multiplication structures of $T$ and $K$ are inherited from the compositions of morphisms (Prop.\ref{prop:composition_to_multiplication}), which are preserved by $\mcal{F}$ since it is a functor. Now suppose ${\bf F}(g)=1$ for some $g\in T$. Let $[\tau_{\rm mark},\tau'_{\rm mark}]$ be a morphism in $Pt$ corresponding to $g$. Then the morphism $[\mcal{F}(\tau_{\rm mark}), \mcal{F}(\tau'_{\rm mark})]$ of $Pt_{\rm dot}$ corresponds to $1 \in K$, meaning that $\mcal{F}(\tau'_{\rm mark}) = 1.\mcal{F}(\tau_{\rm mark}) = \mcal{F}(\tau_{\rm mark})$. Since $\mcal{F}$ is injective on the set of objects (Prop.\ref{prop:well-definedness_of_F}), we get $\tau'_{\rm mark}= \tau_{\rm mark}$, and therefore $g=1$. Hence the injectivity of ${\bf F}$.
\end{proof}

\begin{lemma}
\label{lem:alternative_construction_of_bold_F}
The construction in Prop.\ref{prop:bold_F} can also be written as 
$$
\mcal{F}(g.\tau_{\rm mark}) = ({\bf F}(g)).\mcal{F}(\tau_{\rm mark}), \quad \mbox{for any $g\in T$ and any marked tessellation $\tau_{\rm mark}$.} \qed
$$
\end{lemma}

What remains to be shown about the construction of ${\bf F}$ is its essential uniqueness. 
\begin{proposition}
\label{prop:essential_uniqueness_of_bold_F}
Let $\tau^\circ_{\rm mark}$, $\tau^{\circ\circ}_{\rm mark}$ be any marked tessellations and $\tau^\circ_{\rm dot}$, $\tau^{\circ\circ}_{\rm dot}$ be any dotted tessellations. Let $\mcal{F}^\circ : Pt \to Pt_{\rm dot}$ and $\mcal{F}^{\circ\circ} : Pt \to Pt_{\rm dot}$ be the functors constructed in Prop.\ref{prop:well-definedness_of_F} for initial conditions $\mcal{F}^\circ(\tau^\circ_{\rm mark}) = \tau^\circ_{\rm dot}$ and $\mcal{F}^{\circ\circ}(\tau^{\circ\circ}_{\rm mark}) = \tau^{\circ\circ}_{\rm dot}$ respectively. Let ${\bf F}^\circ, {\bf F}^{\circ\circ} : T \to K$ be the group homomorphisms constructed respectively from $\mcal{F}^\circ, \mcal{F}^{\circ\circ}$ by Prop.\ref{prop:bold_F}. Then there is $h \in K$ such that
\begin{align}
\label{eq:bold_F_conjugation}
{\bf F}^{\circ\circ}(g) = h^{-1} {\bf F}^\circ(g) h, \quad \forall g \in T.
\end{align}
\end{proposition}
\begin{proof}
There is a unique $[\varphi] \in M$ such that $\tau^\circ_{\rm mark} = [\varphi].\tau^{\circ\circ}_{\rm mark}$ (transitivity of $M$-action on $Pt$). Then $\mcal{F}^{\circ\circ}( \tau^\circ_{\rm mark} ) = \mcal{F}^{\circ\circ}( [\varphi].\tau^{\circ\circ}_{\rm mark} ) = [\varphi].\mcal{F}^{\circ\circ}(\tau^{\circ\circ}_{\rm mark})$ by the $M$-equivariance of $\mcal{F}^{\circ\circ}$. Thus we let $\tau_{\rm dot} = [\varphi].\tau^{\circ\circ}_{\rm dot}$, so that
$$
\mcal{F}^{\circ\circ}(\tau^\circ_{\rm mark}) = \tau_{\rm dot}.
$$
Now, let $g\in T$, and let $\tau'_{\rm mark} = g.\tau^\circ_{\rm mark}$, so that $[\tau^\circ_{\rm mark}, \tau'_{\rm mark}] \leadsto g$. Let $\tau'_{\rm dot} = \mcal{F}^\circ(\tau'_{\rm mark})$ and $\tau''_{\rm dot} = \mcal{F}^{\circ\circ}(\tau'_{\rm mark})$. Then by definition of ${\bf F}^\circ$ and ${\bf F}^{\circ\circ}$ we have
$$
[\tau^\circ_{\rm dot}, \tau'_{\rm dot}] = [\mcal{F}^\circ(\tau^\circ_{\rm mark}), \mcal{F}^\circ(\tau'_{\rm mark})] \leadsto {\bf F}^\circ(g), \quad
[\tau_{\rm dot}, \tau''_{\rm dot}] =[\mcal{F}^{\circ\circ}(\tau^\circ_{\rm mark}), \mcal{F}^{\circ\circ}(\tau'_{\rm mark})] \leadsto {\bf F}^{\circ\circ}(g).
$$
From the composition rule of $Pt_{\rm dot}$ (Def.\ref{def:Pt_dot}) we have
$$
[\tau_{\rm dot}, \tau''_{\rm dot}]
= [ \tau'_{\rm dot} , \tau''_{\rm dot} ] \circ [\tau^\circ_{\rm dot}, \tau'_{\rm dot}] \circ [ \tau_{\rm dot} , \tau^\circ_{\rm dot}].
$$
Let $h\in K$ be the element such that $[\tau_{\rm dot}, \tau^\circ_{\rm dot}] \leadsto h$. Then, in view of Prop.\ref{prop:composition_to_multiplication}, all we need to show for proving \eqref{eq:bold_F_conjugation} is $[\tau'_{\rm dot}, \tau''_{\rm dot}] \leadsto h^{-1}$, or equivalently, $[\tau''_{\rm dot}, \tau'_{\rm dot}] \leadsto h$. Observe that
$$
[\tau''_{\rm dot}, \tau'_{\rm dot}] = [\mcal{F}^{\circ\circ}(\tau'_{\rm mark}), \mcal{F}^\circ(\tau'_{\rm mark})]
= [\mcal{F}^{\circ\circ}(g.\tau^\circ_{\rm mark}), \mcal{F}^\circ(g.\tau^\circ_{\rm mark})]
$$
By the transitivity of the $M$-action on $Pt$, there is a unique $[\psi] \in M$ such that $g.\tau^\circ_{\rm mark} = [\psi].\tau^\circ_{\rm mark}$. So we have $\mcal{F}^{\circ\circ}(g.\tau^\circ_{\rm mark}) = \mcal{F}^{\circ\circ}([\varphi].\tau^\circ_{\rm mark}) = [\varphi]. \mcal{F}^{\circ\circ}(\tau^\circ_{\rm mark}) = [\varphi].\tau_{\rm dot}$ by the $M$-equivariance of $\mcal{F}^{\circ\circ}$, and similarly $\mcal{F}^\circ(g.\tau^\circ_{\rm mark}) = [\varphi].\mcal{F}^\circ(\tau^\circ_{\rm mark}) = [\varphi].\tau^\circ_{\rm dot}$ by the $M$-equivariance of $\mcal{F}^\circ$. Thus we get $[\tau''_{\rm dot},\tau'_{\rm dot}] = [ [\varphi].\tau_{\rm dot}, [\varphi].\tau^\circ_{\rm dot}]$ which corresponds to the same element in $K$ as $[\tau_{\rm dot}, \tau^\circ_{\rm dot}]$ does, namely $h$, by Prop.\ref{prop:morphisms_lead_to_same_element}. This $h$ doesn't depend on $g$.
\end{proof}

Now that we know that a different choice of an initial condition for $\mcal{F}$ does not essentially alter ${\bf F}:T \to K$, we are allowed to use any initial condition. As mentioned, we choose to use \eqref{eq:our_initial_condition_for_F}, and from now on, we implicitly assume this when we use $\mcal{F}$ and ${\bf F}$. For a concrete description of this map ${\bf F}:T\to K$, it suffices to describe the images of the generators $\alpha,\beta$ of $T$. The easiest way to actually compute these images is to apply on $\tau^*_{\rm mark}$. Namely, we have
$$
[\tau^*_{\rm dot}, \mcal{F}(\alpha.\tau^*_{\rm mark})] \leadsto {\bf F}(\alpha) \quad\mbox{and}\quad
[\tau^*_{\rm dot}, \mcal{F}(\beta.\tau^*_{\rm mark})] \leadsto {\bf F}(\beta),
$$
so we should investigate $\mcal{F}(\alpha.\tau^*_{\rm mark})$ and $\mcal{F}(\beta.\tau^*_{\rm mark})$, using Prop.\ref{prop:concrete_description_of_F}. To write this result down, we first need the following definition.

\begin{definition}
\label{def:P_gamma_alpha_beta}
We define permutations $\gamma_\alpha,\gamma_\beta$ of $\mathbb{Q}^\times$ as follows.

We first require that $\gamma_\alpha$ fixes $-1$ and $1$. 
If an ideal triangle of $\tau^*$ is labeled by $j \notin \{ -1,1 \}$ under the labeling rule of $\tau^*_{\rm dot}$ and by $j'$ under that of $\mcal{F}( \alpha.\tau_{\rm mark}^*)$, we set $\gamma_\alpha(j) = j'$. We thus establish a $\mathbb{Q}^\times$-permutation $\gamma_\alpha$.

For $\gamma_\beta$,
if an ideal triangle of $\tau^*$ is labeled by $j \in \mathbb{Q}^\times$ by the labeling rule of $\tau_{\rm dot}^*$ and by $j'$ under that of $\mcal{F}(\beta.\tau_{\rm mark}^*)$, we set $\gamma_\beta(j)=j'$. We thus get a $\mathbb{Q}^\times$-permutation $\gamma_\beta$. 
\end{definition}
The permutations $\gamma_\alpha$ and $\gamma_\beta$ are best seen in the pictures; see Figures \ref{fig:P_gamma_alpha} and \ref{fig:P_gamma_beta}. We can write some of the actions of the permutations $\gamma_\alpha$, $\gamma_\beta$ on $\mathbb{Q}^\times$:
\begin{align}
\label{eq:gamma_alpha}
& \left\{ {\renewcommand{\arraystretch}{1.2}
\begin{array}{l}
\gamma_\alpha(-1) = -1, \quad \gamma_\alpha(1)=1, \quad \gamma_\alpha(-2)=2, \quad
\gamma_\alpha (-\frac{1}{2}) = -2, \\
\gamma_\alpha(\frac{1}{2}) = -\frac{1}{2}, \quad \gamma_\alpha(2) = \frac{1}{2}, \quad \gamma_\alpha(\frac{1}{3}) = -\frac{3}{2}, \quad \gamma_\alpha(3) = \frac{2}{3}, \quad {\rm etc},
\end{array}} \right.
\\
\label{eq:gamma_beta}
& \left\{ {\renewcommand{\arraystretch}{1.2}
\begin{array}{l}
\gamma_\beta(-1) = -1, \quad \gamma_\beta(1) = -\frac{1}{2}, \quad \gamma_\beta(-2) =1, \quad
\gamma_\beta(-\frac{1}{2}) = -2, \\
\gamma_\beta(\frac{1}{2}) = -\frac{2}{3}, \quad \gamma_\beta(2) = -\frac{1}{3}, \quad \gamma_\beta(\frac{1}{3}) = -\frac{3}{4}, \quad \gamma_\beta(3) = -\frac{1}{4}, \quad {\rm etc}.
\end{array}} \right.
\end{align}

\begin{figure}[htbp!]
%
\centering
$\begin{array}{ccc}
\begin{pspicture}[showgrid=false,linewidth=0.5pt](-2.6,-2.8)(3.5,2.4)
\psarc(0,0){2.4}{0}{360}
\psarc(-2.4,2.4){2.4}{-90}{0}
\psarc(-2.4,-2.4){2.4}{0}{90}
\psarc(2.4,-2.4){2.4}{90}{180}
\psarc(2.4,2.4){2.4}{180}{-90}
\psarc[arcsep=0.5pt](2.4,1.2){1.2}{142.5}{-90}
\psarc[arcsep=0.5pt](0.8,2.4){0.8}{180}{-40}
\psarc[arcsep=0.5pt](2.4,0.8){0.8}{127}{-90}
\psarc[arcsep=0.5pt](1.714,1.714){0.343}{140}{-51}
\psarc[arcsep=0.5pt](0.48,2.4){0.48}{180}{-22}
\psarc[arcsep=0.5pt](1.2,2.1){0.3}{160}{-34}
\rput(-2.6,0){$\tau_0$}
\rput(2.7,0){$\tau_\infty$}
\rput(0,2.6){$\tau_{-1}$}
\rput(1.68,2.0){$\tau_{-2}$}
\psarc[arcsep=0.5pt](-2.4,1.2){1.2}{-90}{37.5}
\psarc[arcsep=0.5pt](-0.8,2.4){0.8}{-140}{0}
\psarc[arcsep=0.5pt](-2.4,0.8){0.8}{-90}{53}
\psarc[arcsep=0.5pt](-1.714,1.714){0.343}{-129}{40}
\psarc[arcsep=0.5pt](-0.48,2.4){0.48}{-158}{0}
\psarc[arcsep=0.5pt](-1.2,2.1){0.3}{-146}{20}
\rput(-1.65,2.15){\fontsize{7}{7} $\tau_{-\frac{1}{2}}$}
\psarc[arcsep=0.5pt](-2.4,-1.2){1.2}{-37.5}{90}
\psarc[arcsep=0.5pt](-0.8,-2.4){0.8}{0}{140}
\psarc[arcsep=0.5pt](-2.4,-0.8){0.8}{-53}{90}
\psarc[arcsep=0.5pt](-1.714,-1.714){0.343}{-40}{129}
\psarc[arcsep=0.5pt](-0.48,-2.4){0.48}{0}{158}
\psarc[arcsep=0.5pt](-1.2,-2.1){0.3}{-20}{146}
\rput(-1.57,-2.13){\fontsize{8}{8} $\tau_{\frac{1}{2}}$}
\rput(-2.05,-1.62){\fontsize{8}{8} $\tau_{\frac{1}{3}}$}
\psarc[arcsep=0.5pt](2.4,-1.2){1.2}{90}{-142.5}
\psarc[arcsep=0.5pt](0.8,-2.4){0.8}{40}{-180}
\psarc[arcsep=0.5pt](2.4,-0.8){0.8}{90}{-127}
\psarc[arcsep=0.5pt](1.714,-1.714){0.343}{51}{-140}
\psarc[arcsep=0.5pt](0.48,-2.4){0.48}{22}{-180}
\psarc[arcsep=0.5pt](1.2,-2.1){0.3}{34}{-160}
\rput(0,-2.6){$\tau_1$}
\rput(1.6,-2.15){$\tau_2$}
\rput(2.05,-1.63){$\tau_3$}
%
\rput(-0.05,1.6){\fontsize{9}{9} $\bullet$}
\rput(-0.05,-1.6){\fontsize{9}{9} $\bullet$}
\rput(1.07,1.52){\fontsize{8}{8} $\bullet$}
\rput(-1.15,1.52){\fontsize{8}{8} $\bullet$}
\rput(1.07,-1.52){\fontsize{8}{8} $\bullet$}
\rput(-1.15,-1.52){\fontsize{8}{8} $\bullet$}
\rput(1.59,1.28){\fontsize{7}{7} $\bullet$}
\rput(-1.69,1.28){\fontsize{7}{7} $\bullet$}
\rput(1.59,-1.28){\fontsize{7}{7} $\bullet$}
\rput(-1.69,-1.28){\fontsize{7}{7} $\bullet$}
\rput(0.78,1.96){\fontsize{6}{6} $\bullet$}
\rput(-0.89,1.96){\fontsize{6}{6} $\bullet$}
\rput(0.78,-1.96){\fontsize{6}{6} $\bullet$}
\rput(-0.89,-1.96){\fontsize{6}{6} $\bullet$}
%
\psline{-}(-2.4,0)(2.4,0)
\rput{110}(2.02,0.68){\fontsize{12}{12} $\cdots$}
\rput{-110}(2.02,-0.68){\fontsize{12}{12} $\cdots$}
\rput{74}(-1.96,0.68){\fontsize{12}{12} $\cdots$}
\rput{-74}(-1.96,-0.68){\fontsize{12}{12} $\cdots$}
\rput{-10}(0.40,2.18){\fontsize{10}{10} $\cdots$}
\rput{10}(-0.45,2.18){\fontsize{10}{10} $\cdots$}
\rput{-10}(-0.45,-2.18){\fontsize{10}{10} $\cdots$}
\rput{10}(0.40,-2.18){\fontsize{10}{10} $\cdots$}
\rput{-43}(1.58,1.61){\fontsize{8}{8} $\cdots$}
\rput{43}(-1.62,1.59){\fontsize{8}{8} $\cdots$}
\rput{-43}(-1.62,-1.59){\fontsize{8}{8} $\cdots$}
\rput{43}(1.60,-1.61){\fontsize{8}{8} $\cdots$}
\rput{-29}(1.10,2.00){\fontsize{7}{7} $\cdots$}
\rput{29}(-1.16,1.98){\fontsize{7}{7} $\cdots$}
\rput{-29}(-1.16,-1.98){\fontsize{7}{7} $\cdots$}
\rput{29}(1.10,-2.00){\fontsize{7}{7} $\cdots$}
%
\rput(-0.05,0.60){\fontsize{10}{10} $[-1]$}
\rput(-0.05,-0.50){\fontsize{10}{10} $[1]$}
\rput(-0.85,1.24){\fontsize{8}{8} $[-\frac{1}{2}]$}
\rput(-0.85,-1.17){\fontsize{9}{9} $[\frac{1}{2}]$}
\rput(-1.45,-1.05){\fontsize{7}{7} $[\frac{1}{3}]$}
\rput(0.75,1.24){\fontsize{8}{8} $[-2]$}
\rput(0.75,-1.17){\fontsize{8}{8} $[2]$}
\rput(1.37,-1.00){\fontsize{7}{7} $[3]$}
\psline[linewidth=0.5pt, 
arrowsize=3pt 4, 
arrowlength=2, 
arrowinset=0.3] 
{->}(0,0)(0.2,0)
%
\rput[l](3.2,0){\pcline[linewidth=0.7pt, arrowsize=2pt 4]{->}(0,0)(1;0)\Aput{$\alpha$}}
\end{pspicture}
%
%
%
%
& \begin{pspicture}[showgrid=false,linewidth=0.5pt](-3.5,-2.8)(2.6,2.4)
\psarc(0,0){2.4}{0}{360}
\psarc(-2.4,2.4){2.4}{-90}{0}
\psarc(-2.4,-2.4){2.4}{0}{90}
\psarc(2.4,-2.4){2.4}{90}{180}
\psarc(2.4,2.4){2.4}{180}{-90}
\psarc[arcsep=0.5pt](2.4,1.2){1.2}{142.5}{-90}
\psarc[arcsep=0.5pt](0.8,2.4){0.8}{180}{-40}
\psarc[arcsep=0.5pt](2.4,0.8){0.8}{127}{-90}
\psarc[arcsep=0.5pt](1.714,1.714){0.343}{140}{-51}
\psarc[arcsep=0.5pt](0.48,2.4){0.48}{180}{-22}
\psarc[arcsep=0.5pt](1.2,2.1){0.3}{160}{-34}
\rput(-2.7,0){\fontsize{8}{8} $\tau'_{-1}$}
\rput(2.7,0){$\tau'_1$}
\rput(0,2.65){$\tau'_\infty$}
\rput(1.6,2.1){$\tau'_2$}
\psarc[arcsep=0.5pt](-2.4,1.2){1.2}{-90}{37.5}
\psarc[arcsep=0.5pt](-0.8,2.4){0.8}{-140}{0}
\psarc[arcsep=0.5pt](-2.4,0.8){0.8}{-90}{53}
\psarc[arcsep=0.5pt](-1.714,1.714){0.343}{-129}{40}
\psarc[arcsep=0.5pt](-0.48,2.4){0.48}{-158}{0}
\psarc[arcsep=0.5pt](-1.2,2.1){0.3}{-146}{20}
\rput(-1.55,2.2){$\fontsize{8}{8} \tau'_{-2}$}
\psarc[arcsep=0.5pt](-2.4,-1.2){1.2}{-37.5}{90}
\psarc[arcsep=0.5pt](-0.8,-2.4){0.8}{0}{140}
\psarc[arcsep=0.5pt](-2.4,-0.8){0.8}{-53}{90}
\psarc[arcsep=0.5pt](-1.714,-1.714){0.343}{-40}{129}
\psarc[arcsep=0.5pt](-0.48,-2.4){0.48}{0}{158}
\psarc[arcsep=0.5pt](-1.2,-2.1){0.3}{-20}{146}
\rput(-1.67,-2.08){\fontsize{8}{8} $\tau'_{-\frac{1}{2}}$}
\rput(-2.27,-1.50){\fontsize{7}{7} $\tau'_{-\frac{2}{3}}$}
\psarc[arcsep=0.5pt](2.4,-1.2){1.2}{90}{-142.5}
\psarc[arcsep=0.5pt](0.8,-2.4){0.8}{40}{-180}
\psarc[arcsep=0.5pt](2.4,-0.8){0.8}{90}{-127}
\psarc[arcsep=0.5pt](1.714,-1.714){0.343}{51}{-140}
\psarc[arcsep=0.5pt](0.48,-2.4){0.48}{22}{-180}
\psarc[arcsep=0.5pt](1.2,-2.1){0.3}{34}{-160}
\rput(0,-2.7){$\tau'_0$}
\rput(1.65,-2.13){\fontsize{8}{8} $\tau'_{\frac{1}{2}}$}
\rput(2.08,-1.60){\fontsize{7}{7} $\tau'_{\frac{2}{3}}$}
%
%
\rput(1.6,0.03){\fontsize{9}{9} $\bullet$}
\rput(-1.6,-0.0){\fontsize{9}{9} $\bullet$}
\rput(1.07,1.52){\fontsize{8}{8} $\bullet$}
\rput(-1.15,1.52){\fontsize{8}{8} $\bullet$}
\rput(1.07,-1.52){\fontsize{8}{8} $\bullet$}
\rput(-1.15,-1.52){\fontsize{8}{8} $\bullet$}
\rput(1.59,1.28){\fontsize{7}{7} $\bullet$}
\rput(-1.69,1.28){\fontsize{7}{7} $\bullet$}
\rput(1.59,-1.28){\fontsize{7}{7} $\bullet$}
\rput(-1.69,-1.28){\fontsize{7}{7} $\bullet$}
\rput(0.78,1.96){\fontsize{6}{6} $\bullet$}
\rput(-0.89,1.96){\fontsize{6}{6} $\bullet$}
\rput(0.78,-1.96){\fontsize{6}{6} $\bullet$}
\rput(-0.89,-1.96){\fontsize{6}{6} $\bullet$}
%
\psline{-}(0,-2.4)(0,2.4)
\rput{110}(2.02,0.68){\fontsize{12}{12} $\cdots$}
\rput{-110}(2.02,-0.68){\fontsize{12}{12} $\cdots$}
\rput{74}(-1.96,0.68){\fontsize{12}{12} $\cdots$}
\rput{-74}(-1.96,-0.68){\fontsize{12}{12} $\cdots$}
\rput{-10}(0.40,2.18){\fontsize{10}{10} $\cdots$}
\rput{10}(-0.45,2.18){\fontsize{10}{10} $\cdots$}
\rput{-10}(-0.45,-2.18){\fontsize{10}{10} $\cdots$}
\rput{10}(0.40,-2.18){\fontsize{10}{10} $\cdots$}
\rput{-43}(1.58,1.61){\fontsize{8}{8} $\cdots$}
\rput{43}(-1.62,1.59){\fontsize{8}{8} $\cdots$}
\rput{-43}(-1.62,-1.59){\fontsize{8}{8} $\cdots$}
\rput{43}(1.60,-1.61){\fontsize{8}{8} $\cdots$}
\rput{-29}(1.10,2.00){\fontsize{7}{7} $\cdots$}
\rput{29}(-1.16,1.98){\fontsize{7}{7} $\cdots$}
\rput{-29}(-1.16,-1.98){\fontsize{7}{7} $\cdots$}
\rput{29}(1.10,-2.00){\fontsize{7}{7} $\cdots$}
%
\rput(0.65,-0.05){\fontsize{10}{10} $[1]$}
\rput(-0.80,-0.05){\fontsize{10}{10} $[-1]$}
\rput(-0.85,1.24){\fontsize{8}{8} $[-2]$}
\rput(-0.85,-1.17){\fontsize{9}{9} $[-\frac{1}{2}]$}
\rput(-1.55,-1.05){\fontsize{2}{2} $[-\frac{2}{3}]$}
\rput(0.75,1.24){\fontsize{8}{8} $[2]$}
\rput(0.75,-1.17){\fontsize{8}{8} $[\frac{1}{2}]$}
\rput(1.37,-1.00){\fontsize{7}{7} $[\frac{2}{3}]$}
\psline[linewidth=0.5pt, 
arrowsize=3pt 4, 
arrowlength=2, 
arrowinset=0.3] 
{->}(0,0)(0,0.2)
\end{pspicture}
&
\end{array} $

\vspace{-4mm}

\caption{The $\alpha$-move on $\tau_{\rm mark}^*$ and $\tau_{\rm dot}^*$}
\label{fig:P_gamma_alpha}
\end{figure}
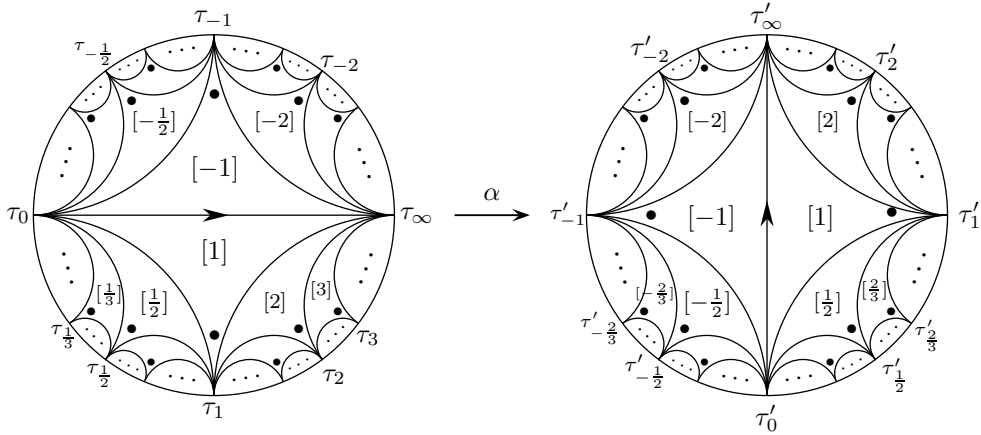
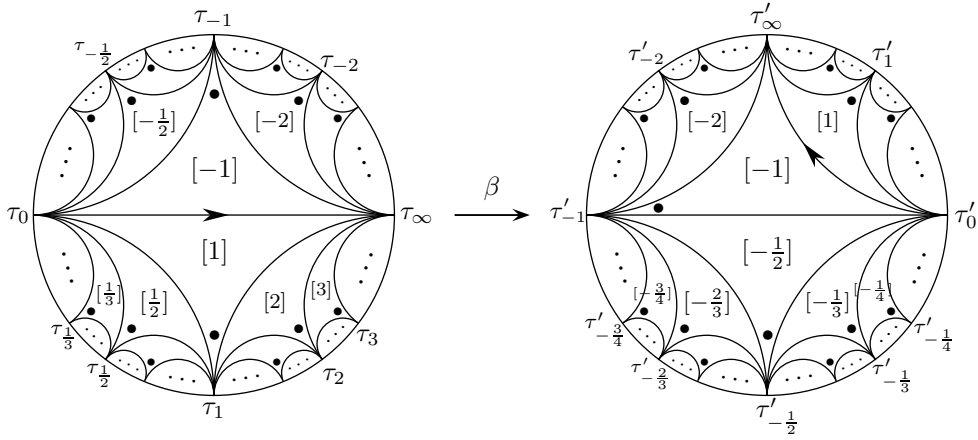
\begin{figure}[htbp!]
%
\centering
$\begin{array}{ccc}
\begin{pspicture}[showgrid=false,linewidth=0.5pt](-2.6,-2.8)(3.5,2.6)
\psarc(0,0){2.4}{0}{360}
\psarc(-2.4,2.4){2.4}{-90}{0}
\psarc(-2.4,-2.4){2.4}{0}{90}
\psarc(2.4,-2.4){2.4}{90}{180}
\psarc(2.4,2.4){2.4}{180}{-90}
\psarc[arcsep=0.5pt](2.4,1.2){1.2}{142.5}{-90}
\psarc[arcsep=0.5pt](0.8,2.4){0.8}{180}{-40}
\psarc[arcsep=0.5pt](2.4,0.8){0.8}{127}{-90}
\psarc[arcsep=0.5pt](1.714,1.714){0.343}{140}{-51}
\psarc[arcsep=0.5pt](0.48,2.4){0.48}{180}{-22}
\psarc[arcsep=0.5pt](1.2,2.1){0.3}{160}{-34}
\rput(-2.6,0){$\tau_0$}
\rput(2.7,0){$\tau_\infty$}
\rput(0,2.6){$\tau_{-1}$}
\rput(1.68,2.0){$\tau_{-2}$}
\psarc[arcsep=0.5pt](-2.4,1.2){1.2}{-90}{37.5}
\psarc[arcsep=0.5pt](-0.8,2.4){0.8}{-140}{0}
\psarc[arcsep=0.5pt](-2.4,0.8){0.8}{-90}{53}
\psarc[arcsep=0.5pt](-1.714,1.714){0.343}{-129}{40}
\psarc[arcsep=0.5pt](-0.48,2.4){0.48}{-158}{0}
\psarc[arcsep=0.5pt](-1.2,2.1){0.3}{-146}{20}
\rput(-1.65,2.15){\fontsize{7}{7} $\tau_{-\frac{1}{2}}$}
\psarc[arcsep=0.5pt](-2.4,-1.2){1.2}{-37.5}{90}
\psarc[arcsep=0.5pt](-0.8,-2.4){0.8}{0}{140}
\psarc[arcsep=0.5pt](-2.4,-0.8){0.8}{-53}{90}
\psarc[arcsep=0.5pt](-1.714,-1.714){0.343}{-40}{129}
\psarc[arcsep=0.5pt](-0.48,-2.4){0.48}{0}{158}
\psarc[arcsep=0.5pt](-1.2,-2.1){0.3}{-20}{146}
\rput(-1.57,-2.13){\fontsize{8}{8} $\tau_{\frac{1}{2}}$}
\rput(-2.05,-1.62){\fontsize{8}{8} $\tau_{\frac{1}{3}}$}
\psarc[arcsep=0.5pt](2.4,-1.2){1.2}{90}{-142.5}
\psarc[arcsep=0.5pt](0.8,-2.4){0.8}{40}{-180}
\psarc[arcsep=0.5pt](2.4,-0.8){0.8}{90}{-127}
\psarc[arcsep=0.5pt](1.714,-1.714){0.343}{51}{-140}
\psarc[arcsep=0.5pt](0.48,-2.4){0.48}{22}{-180}
\psarc[arcsep=0.5pt](1.2,-2.1){0.3}{34}{-160}
\rput(0,-2.6){$\tau_1$}
\rput(1.6,-2.15){$\tau_2$}
\rput(2.05,-1.63){$\tau_3$}
%
\rput(-0.05,1.6){\fontsize{9}{9} $\bullet$}
\rput(-0.05,-1.6){\fontsize{9}{9} $\bullet$}
\rput(1.07,1.52){\fontsize{8}{8} $\bullet$}
\rput(-1.15,1.52){\fontsize{8}{8} $\bullet$}
\rput(1.07,-1.52){\fontsize{8}{8} $\bullet$}
\rput(-1.15,-1.52){\fontsize{8}{8} $\bullet$}
\rput(1.59,1.28){\fontsize{7}{7} $\bullet$}
\rput(-1.69,1.28){\fontsize{7}{7} $\bullet$}
\rput(1.59,-1.28){\fontsize{7}{7} $\bullet$}
\rput(-1.69,-1.28){\fontsize{7}{7} $\bullet$}
\rput(0.78,1.96){\fontsize{6}{6} $\bullet$}
\rput(-0.89,1.96){\fontsize{6}{6} $\bullet$}
\rput(0.78,-1.96){\fontsize{6}{6} $\bullet$}
\rput(-0.89,-1.96){\fontsize{6}{6} $\bullet$}
%
\psline{-}(-2.4,0)(2.4,0)
\rput{110}(2.02,0.68){\fontsize{12}{12} $\cdots$}
\rput{-110}(2.02,-0.68){\fontsize{12}{12} $\cdots$}
\rput{74}(-1.96,0.68){\fontsize{12}{12} $\cdots$}
\rput{-74}(-1.96,-0.68){\fontsize{12}{12} $\cdots$}
\rput{-10}(0.40,2.18){\fontsize{10}{10} $\cdots$}
\rput{10}(-0.45,2.18){\fontsize{10}{10} $\cdots$}
\rput{-10}(-0.45,-2.18){\fontsize{10}{10} $\cdots$}
\rput{10}(0.40,-2.18){\fontsize{10}{10} $\cdots$}
\rput{-43}(1.58,1.61){\fontsize{8}{8} $\cdots$}
\rput{43}(-1.62,1.59){\fontsize{8}{8} $\cdots$}
\rput{-43}(-1.62,-1.59){\fontsize{8}{8} $\cdots$}
\rput{43}(1.60,-1.61){\fontsize{8}{8} $\cdots$}
\rput{-29}(1.10,2.00){\fontsize{7}{7} $\cdots$}
\rput{29}(-1.16,1.98){\fontsize{7}{7} $\cdots$}
\rput{-29}(-1.16,-1.98){\fontsize{7}{7} $\cdots$}
\rput{29}(1.10,-2.00){\fontsize{7}{7} $\cdots$}
%
\rput(-0.05,0.60){\fontsize{10}{10} $[-1]$}
\rput(-0.05,-0.50){\fontsize{10}{10} $[1]$}
\rput(-0.85,1.24){\fontsize{8}{8} $[-\frac{1}{2}]$}
\rput(-0.85,-1.17){\fontsize{9}{9} $[\frac{1}{2}]$}
\rput(-1.45,-1.05){\fontsize{7}{7} $[\frac{1}{3}]$}
\rput(0.75,1.24){\fontsize{8}{8} $[-2]$}
\rput(0.75,-1.17){\fontsize{8}{8} $[2]$}
\rput(1.37,-1.00){\fontsize{7}{7} $[3]$}
\psline[linewidth=0.5pt, 
arrowsize=3pt 4, 
arrowlength=2, 
arrowinset=0.3] 
{->}(0,0)(0.2,0)
%
\rput[l](3.2,0){\pcline[linewidth=0.7pt, arrowsize=2pt 4]{->}(0,0)(1;0)\Aput{$\beta$}}
\end{pspicture}
%
%
%
%
& \begin{pspicture}[showgrid=false,linewidth=0.5pt](-3.5,-2.8)(2.6,2.6)
\psarc(0,0){2.4}{0}{360}
\psarc(-2.4,2.4){2.4}{-90}{0}
\psarc(-2.4,-2.4){2.4}{0}{90}
\psarc(2.4,-2.4){2.4}{90}{180}
\psarc(2.4,2.4){2.4}{180}{-90}
\psarc[arcsep=0.5pt](2.4,1.2){1.2}{142.5}{-90}
\psarc[arcsep=0.5pt](0.8,2.4){0.8}{180}{-40}
\psarc[arcsep=0.5pt](2.4,0.8){0.8}{127}{-90}
\psarc[arcsep=0.5pt](1.714,1.714){0.343}{140}{-51}
\psarc[arcsep=0.5pt](0.48,2.4){0.48}{180}{-22}
\psarc[arcsep=0.5pt](1.2,2.1){0.3}{160}{-34}
\rput(-2.7,0.05){\fontsize{9}{9} $\tau'_{-1}$}
\rput(2.65,0){$\tau'_0$}
\rput(0,2.65){$\tau'_\infty$}
\rput(1.55,2.15){$\tau'_1$}
\psarc[arcsep=0.5pt](-2.4,1.2){1.2}{-90}{37.5}
\psarc[arcsep=0.5pt](-0.8,2.4){0.8}{-140}{0}
\psarc[arcsep=0.5pt](-2.4,0.8){0.8}{-90}{53}
\psarc[arcsep=0.5pt](-1.714,1.714){0.343}{-129}{40}
\psarc[arcsep=0.5pt](-0.48,2.4){0.48}{-158}{0}
\psarc[arcsep=0.5pt](-1.2,2.1){0.3}{-146}{20}
\rput(-1.66,2.15){\fontsize{9}{9} $\tau'_{-2}$}
\psarc[arcsep=0.5pt](-2.4,-1.2){1.2}{-37.5}{90}
\psarc[arcsep=0.5pt](-0.8,-2.4){0.8}{0}{140}
\psarc[arcsep=0.5pt](-2.4,-0.8){0.8}{-53}{90}
\psarc[arcsep=0.5pt](-1.714,-1.714){0.343}{-40}{129}
\psarc[arcsep=0.5pt](-0.48,-2.4){0.48}{0}{158}
\psarc[arcsep=0.5pt](-1.2,-2.1){0.3}{-20}{146}
\rput(-1.60,-2.10){\fontsize{7}{7} $\tau'_{-\frac{2}{3}}$}
\rput(-2.20,-1.52){\fontsize{8}{8} $\tau'_{-\frac{3}{4}}$}
\psarc[arcsep=0.5pt](2.4,-1.2){1.2}{90}{-142.5}
\psarc[arcsep=0.5pt](0.8,-2.4){0.8}{40}{-180}
\psarc[arcsep=0.5pt](2.4,-0.8){0.8}{90}{-127}
\psarc[arcsep=0.5pt](1.714,-1.714){0.343}{51}{-140}
\psarc[arcsep=0.5pt](0.48,-2.4){0.48}{22}{-180}
\psarc[arcsep=0.5pt](1.2,-2.1){0.3}{34}{-160}
\rput(0.15,-2.67){$\tau'_{-\frac{1}{2}}$}
\rput(1.62,-2.13){\fontsize{8}{8} $\tau'_{-\frac{1}{3}}$}
\rput(2.17,-1.58){\fontsize{8}{8} $\tau'_{-\frac{1}{4}}$}
%
\rput(-0.05,-1.6){\fontsize{9}{9} $\bullet$}
\rput(-1.50, 0.09){\fontsize{9}{9} $\bullet$}
\rput(1.07,1.52){\fontsize{8}{8} $\bullet$}
\rput(-1.15,1.52){\fontsize{8}{8} $\bullet$}
\rput(1.07,-1.52){\fontsize{8}{8} $\bullet$}
\rput(-1.15,-1.52){\fontsize{8}{8} $\bullet$}
\rput(1.59,1.28){\fontsize{7}{7} $\bullet$}
\rput(-1.69,1.28){\fontsize{7}{7} $\bullet$}
\rput(1.59,-1.28){\fontsize{7}{7} $\bullet$}
\rput(-1.69,-1.28){\fontsize{7}{7} $\bullet$}
\rput(0.78,1.96){\fontsize{6}{6} $\bullet$}
\rput(-0.89,1.96){\fontsize{6}{6} $\bullet$}
\rput(0.78,-1.96){\fontsize{6}{6} $\bullet$}
\rput(-0.89,-1.96){\fontsize{6}{6} $\bullet$}
%
\psline{-}(-2.4,0)(2.4,0)
\rput{110}(2.02,0.68){\fontsize{12}{12} $\cdots$}
\rput{-110}(2.02,-0.68){\fontsize{12}{12} $\cdots$}
\rput{74}(-1.96,0.68){\fontsize{12}{12} $\cdots$}
\rput{-74}(-1.96,-0.68){\fontsize{12}{12} $\cdots$}
\rput{-10}(0.40,2.18){\fontsize{10}{10} $\cdots$}
\rput{10}(-0.45,2.18){\fontsize{10}{10} $\cdots$}
\rput{-10}(-0.45,-2.18){\fontsize{10}{10} $\cdots$}
\rput{10}(0.40,-2.18){\fontsize{10}{10} $\cdots$}
\rput{-43}(1.58,1.61){\fontsize{8}{8} $\cdots$}
\rput{43}(-1.62,1.59){\fontsize{8}{8} $\cdots$}
\rput{-43}(-1.62,-1.59){\fontsize{8}{8} $\cdots$}
\rput{43}(1.60,-1.61){\fontsize{8}{8} $\cdots$}
\rput{-29}(1.10,2.00){\fontsize{7}{7} $\cdots$}
\rput{29}(-1.16,1.98){\fontsize{7}{7} $\cdots$}
\rput{-29}(-1.16,-1.98){\fontsize{7}{7} $\cdots$}
\rput{29}(1.10,-2.00){\fontsize{7}{7} $\cdots$}
%
\rput(-0.05,0.60){\fontsize{10}{10} $[-1]$}
\rput(-0.05,-0.50){\fontsize{10}{10} $[-\frac{1}{2}]$}
\rput(-0.85,1.24){\fontsize{8}{8} $[-2]$}
\rput(-0.85,-1.17){\fontsize{9}{9} $[-\frac{2}{3}]$}
\rput(-1.58,-1.05){\fontsize{2}{2} $[-\frac{3}{4}]$}
\rput(0.75,1.24){\fontsize{8}{8} $[1]$}
\rput(0.75,-1.17){\fontsize{8}{8} $[-\frac{1}{3}]$}
\rput(1.35,-1.00){\fontsize{2}{2} $[-\frac{1}{4}]$}
\psline[linewidth=0.5pt, 
arrowsize=3pt 4, 
arrowlength=2, 
arrowinset=0.3] 
{->}(0.5,0.936)(0.48,0.96)
\end{pspicture}
&
\end{array} $

\vspace{-2mm}

\caption{The $\beta$-move on $\tau_{\rm mark}^*$ and $\tau_{\rm dot}^*$}
\label{fig:P_gamma_beta}
\end{figure}

In particular, $\gamma_\beta$ fixes $-1$. Now, the formulas for ${\bf F}(\alpha)$ and ${\bf F}(\beta)$ are given as follows:
\begin{proposition}
\label{prop:our_bold_F}
The images of $\alpha,\beta$ under the map ${\bf F}:T \to K$ constructed in Prop.\ref{prop:bold_F} with the initial condition \eqref{eq:our_initial_condition_for_F} for the functor $\mcal{F}:Pt \to Pt_{\rm dot}$ (Prop.\ref{prop:well-definedness_of_F}) are given by:
\begin{align}
\label{eq:bold_F_alpha_beta}
{\bf F}(\alpha) = A_{[-1]} T_{[-1][1]}^{-1} A_{[1]} P_{\gamma_\alpha}, \qquad
{\bf F}(\beta) = A_{[-1]} P_{\gamma_\beta},
\end{align}
where the $\mathbb{Q}^\times$-permutations $\gamma_\alpha$, $\gamma_\beta$ are as described in Def.\ref{def:P_gamma_alpha_beta}.
\end{proposition}

This map ${\bf F}: T \to K$ will be used to build a relationship between the two quantizations of universal Teichm\"uller space. 
%
%

\section{Dilogarithmic central extensions of $T$}
\label{sec:dilogarithmic_central_extensions_of_T}

We first briefly review the major results of the quantum universal Teichm\"uller theory of Chekhov-Fock and of Kashaev, and see how they give `dilogarithmic' projective representations of the Ptolemy-Thompson group $T$. Then we develop some group theoretical argument for constructing the `minimal' central extension which `resolves' a projective representation of a group. Thus we will obtain two dilogarithmic central extensions of $T$ from the two quantizations, and state the main theorem of the present paper with an algebraic proof.

\subsection{Quantum universal Teichm\"uller space}

Since there is enough literature on the quantum Teichm\"uller theory, we only state the final results which we will use in the present paper, and refer more interested readers to standard references \cite{Kash98} \cite{Fo} \cite{FC} \cite{FG} (also \cite{T} \cite{FuS}) for a detailed reasoning. Funar and Sergiescu \cite{FuS} applied the Chekhov-Fock(-Goncharov) quantization applied to the universal case to get a one-parameter family of {\em projective functors}
\begin{align}
\label{eq:Pt_to_Hilb}
Pt \to {\rm Hilb},
\end{align}
where ${\rm Hilb}$ is the category of complex Hilbert spaces whose morphisms are unitary maps; a functor is called projective if it preserves the compositions of morphisms only up to multiplicative constants. These functors send each object $(\tau,\vec{a})$ to the Hilbert space $L^2_{\rm fin}(\mathbb{R}^{\tau^{(1)}})$, where $\tau^{(1)}$ is the set of edges of the triangulation $\tau$, and for any infinite index set $I$ the Hilbert space 
\begin{align}
\label{eq:L_2_fin}
L^2_{\rm fin}\left(\mathbb{R}^I, \, \bigwedge_{j \in I} dx_j \right)
\end{align}
is defined as follows, as appeared in \cite{FuS} (the symbol $L^2_{\rm fin}$ is adopted only in the present paper; it is not standard). The elements are represented as square-integrable functions with finite dimensional support, i.e. complex valued functions $f$ on $\mathbb{R}^I$ whose support is contained in $\mathbb{R}^J \times \{0\} \subset \mathbb{R}^I$ for some finite subset $J$ of $I$ depending on $f$, such that $\int_{\mathbb{R}^J} |f|^2 \, \bigwedge_{j \in J} dx_j < \infty$. For two elements $f,g$, choose a minimal finite subset $J$ such that $\mathbb{R}^J\times\{0\}$ contains the intersection of the supports of $f,g$; the inner product $\langle f,g\rangle := \int_{\mathbb{R}^J} f \, \ol{g} \, \bigwedge_{j\in J} dx_j$ makes this space a Hilbert space.

\vs

With the help of the vertex function $j\mapsto \tau_j$ for $(\tau,\vec{a})$ (Def.\ref{def:vertex_function}), one can identify $\tau^{(1)}$ canonically with $\mathbb{Q}\setminus\{0,1\}$. Namely, for the ideal triangle with vertices $\mu(\tau_{\frac{a}{b}})$, $\mu(\tau_{\frac{a+c}{b+d}})$, $\mu(\tau_{\frac{c}{d}})$, where $\mu$ is the Cayley transform (Def.\ref{def:asymptotically_rigid}), label the edge of this triangle opposite to the vertex $\mu(\tau_{\frac{a+c}{b+d}})$ by $\frac{a+c}{b+d}$; then only the label of the edge $\vec{a}$ is ambiguous, which we just label by $-1$. Then each morphism of $Pt$ is sent to an operator on the Hilbert space
$$
\mathscr{V} = L^2_{\rm fin}(\mathbb{R}^{\mathbb{Q}\setminus\{0,1\}})
$$
(see \eqref{eq:L_2_fin}). The construction is such that these functors descend to a well-defined family of set maps
\begin{align}
\label{eq:set_map_CF}
T \to {\rm GL}(\mathscr{V}),
\end{align}
which is `almost' a group homomorphism in the sense of \eqref{eq:general_rho_relation}, where the images of \eqref{eq:set_map_CF} are unitary operators on $\mathscr{V}$.

\vs

On the other hand, the Kashaev quantization applied to the universal case yields a one-parameter family of projective functors
$$
Pt_{\rm dot} \to {\rm Hilb},
$$
where this time each object $(\tau,D,L)$ is sent to the Hilbert space $L^2(\mathbb{R}^{\tau^{(2)}})$, where $\tau^{(2)}$ is the set of ideal triangles of $\tau$. With the help of the vertex function for $(\tau,\vec{a})$, one can identify $\tau^{(2)}$ canonically with $\mathbb{Q}^\times = \mathbb{Q}\setminus\{0\}$; namely, label the ideal triangle having vertices $\mu(\tau_{\frac{a}{b}})$, $\mu(\tau_{\frac{a+c}{b+d}})$, $\mu(\tau_{\frac{c}{d}})$ by $\frac{a+c}{b+d}$. Then each morphism of $Pt_{\rm dot}$ is sent to an operator on the Hilbert space
$$
\mathscr{M} = L^2_{\rm fin}(\mathbb{R}^{\mathbb{Q}^\times})
$$
(see \eqref{eq:L_2_fin}). The construction is such that these functors descend to a well-defined family of set maps
\begin{align}
\label{eq:set_map_Kash}
K \to {\rm GL}(\mathscr{M})
\end{align}
which is almost a group homomorphism in the sense of \eqref{eq:general_rho_relation}, where the images of \eqref{eq:set_map_Kash} are unitary operators on $\mathscr{M}$. We can then pullback \eqref{eq:set_map_Kash} by the natural group homomorphism ${\bf F}: T \to K$ constructred in \S\ref{subsec:bold_F}, to get an almost group homomorphism
\begin{align}
\nonumber
T \to {\rm GL}(\mathscr{M}),
\end{align}
which we can now think of `comparing' with \eqref{eq:set_map_CF}.

\vs

The maps \eqref{eq:set_map_CF} and \eqref{eq:set_map_Kash} satisfying \eqref{eq:general_rho_relation} can be thought of as the main final results of the construction of the quantum universal Teichm\"uller theory, and they are usually referred to as {\em (unitary) projective representations} of $T$ and $K$, respectively. Here, we shall be careful about the terminology `projective representation'.
\begin{definition}[projective and almost-linear representations]
\label{def:projective_and_almost-linear_representations}
A {\em projective representation} of a group $G$ on a vector space $V$ is the group homomorphism
$$
G \to {\rm PGL}(V),
$$
where ${\rm PGL}(V)$ is the quotient of ${\rm GL}(V)$ by the scalar matrices. An {\em almost-linear representation} of $G$ on $V$ is the set map
$$
\rho : G \to {\rm GL}(V)
$$
satisfying the condition \eqref{eq:general_rho_relation} for some constants $c_{g_1,g_2}$ in the underlying base field.
\end{definition}
So, it is easy to see that an almost-linear representation induces a projective representation, by post-composing with the canonical projection 
\begin{align}
\label{eq:p}
p : {\rm GL}(V) \to {\rm PGL}(V).
\end{align}
When $G$ is presented by generators and relations like our $T$ and $K$, we can write
$$
G = F/R,
$$
where $F$ is a free group and $R$ is a normal subgroup of $G$ generated by `relations'. Then, an almost-linear representation of $G$ on $V$ is usually given as a group homomorphism 
$$
\rho: F \to {\rm GL}(V)
$$
such that $\rho(R)$ lies in the group of scalar matrices in ${\rm GL}(V)$. To describe $\rho$, it suffices to describe the images of the generators of $F$ under $\rho$. This is in fact precisely how the result of the quantum Teichm\"uller theory is written as, which we will present in the next subsection.

\subsection{Dilogarithmic projective representations $\rho^{\rm CF}$, $\rho^{\rm Kash}$ of $T$}

In the present paper, we will need a slightly more general formulation than the almost-linear representations:
\begin{definition}[almost $G$-homomorphisms]
\label{def:almost_G-homomorphism}
Let $G$ be a group presented by generators and relations, i.e. $G = F/R$ where $F$ is a free group (for generators) and $R$ is the normal subgroup of $F$ generated by the relations. Let $H$ be a group. Now, a group homomorphism
\begin{align}
\label{eq:eta_F_H}
\eta : F \to H
\end{align}
is said to be an {\em almost $G$-homomorphism} if $\eta(R)$ is contained in the center of $\eta(F)$.

\vs

When $H = {\rm GL}(V)$ for some vector space $V$ and $\eta(R)$ lies in the group of scalar matrices, we call such $\eta$ \eqref{eq:eta_F_H} an {\em almost-linear representation of $G$ on $V$},  by abuse of notation.
\end{definition}

We now introduce the main technical tool used in both quantizations of Teichm\"uller spaces:
\begin{definition}[quantum dilogarithm function]
Let $b>0$, $b^2 \notin \mathbb{Q}$. Let the function $\Psi_b(z)$ on the complex plane be defined by 
\begin{align}
\label{eq:quantum_dilogarithm}
\Psi_b (z) = \exp\left(\frac{1}{4} \int_{\Omega_0} \frac{ e^{-2izw}}{ \sinh (wb) \sinh(w/b)} \frac{dw}{w} \right)
\end{align}
first in the strip $|{\rm Im} \, z| < (b+b^{-1})/2$, where $\Omega_0$ means the real line contour with a detour around $0$ (origin) along a small half circle above the real line, and analytically continued to a meromorphic function on the complex plane using the following functional equations:
\begin{align}
\nonumber
\left\{ {\renewcommand{\arraystretch}{1.2}
\begin{array}{l}
\Psi_b(z - ib/2) = (1+e^{2\pi b z}) \Psi_b(z+ ib/2), \\
\Psi_b(z-ib^{-1}/2) = (1 + e^{2\pi b^{-1} z}) \Psi_b(z+ib^{-1}/2).
\end{array} }
\right.
\end{align}
\end{definition}
This $\Psi_b$ is a `non-compact' version of the so-called {\em quantum dilogarithm} function, defined by Faddeev and Kashaev \cite{FK} \cite{F}. We note that the integral \eqref{eq:quantum_dilogarithm} was already known to Barnes \cite{B}. The generic real number $b$ is the parameter by which the final projective representations are parametrized by. The usual quantum parameter $q$ is related by $q = e^{\pi i b^2}$.

\begin{remark}
In the versions of quantum Teichm\"uller theory that are being used in the present paper\footnote{in particular, involving the `non-compact' quantum dilogarithm, instead of the `compact' version}, the condition $b^2\notin \mathbb{Q}$ ensures that the quantum algebras are represented on Hilbert spaces `strongly irreducibly', which makes the relevant analytical situation `rigid' (see e.g. \cite{FrKi}). Not much is known about what would happen when $b^2\in \mathbb{Q}$, which corresponds to $q$ being a root of unity. Perhaps, one related observation is that all zeros and poles of $\Psi_b(z)$ are simple only when $b^2$ is irrational. However, for the purposes of the present paper, maybe we are allowed to use rational $b^2$. Nevertheless, for the moment, we require $b^2 \notin \mathbb{Q}$, to be safe.
\end{remark}

\vs

Recall from Def.\ref{def:Ptolemy-Thompson_group} and Def.\ref{def:Kashaev_group} that
$$
T = F_{\rm mark}/R_{\rm mark} \quad\mbox{and}\quad K = F_{\rm dot}/R_{\rm dot}.
$$
Then the main result of the Kashaev quantization of universal Teichm\"uller space is written as a group homomorphism
\begin{align}
\label{eq:rho_F_dot_to_GL_mathscr_M}
\rho_{\rm dot} : F_{\rm dot} \to {\rm GL}(\mathscr{M}), \qquad\mathscr{M} = L^2_{\rm fin}(\mathbb{R}^{\mathbb{Q}^\times}, \bigwedge_{j\in \mathbb{Q}^\times} dx_j)
\end{align}
(see \eqref{eq:L_2_fin} for the meaning of $L^2_{\rm fin}$). The images of the generators of $F_{\rm dot}$ under $\rho_{\rm dot}$ are given by the following unitary operators:
\begin{align}
\label{eq:rho_A_j}
\rho_{\rm dot} (A_{[j]}) & = {\bf A}_{[j]} := e^{-\pi i/3} e^{3\pi i Q_j^2} e^{\pi i (P_j + Q_j)^2}, \\
\label{eq:rho_T_jk}
\rho_{\rm dot} (T_{[j][k]}) & = {\bf T}_{[j][k]} := e^{2\pi i P_j Q_k} \Psi_b(Q_j + P_k - Q_k)^{-1}, \\
\label{eq:rho_P_gamma}
(\rho_{\rm dot} (P_{\gamma}) f)( \{ x_j\}_{j \in \mathbb{Q}^\times} ) & = ({\bf P}_\gamma f)( \{ x_j\}_{j \in \mathbb{Q}^\times} ) := f(\{ x_{\gamma(j)} \}_{j\in \mathbb{Q}^\times} ), \qquad \forall f\in \mathscr{M},
\end{align}
for $j,k \in \mathbb{Q}^\times$ ($j\neq k$) and a $\mathbb{Q}^\times$-permutation $\gamma$, where $P_j, Q_j$ are self-adjoint operators defined by the formulas
\begin{align}
\nonumber
P_j f = \frac{1}{2\pi i} \frac{\partial}{\partial x_j} f, \qquad
Q_j f = x_j f,
\end{align}
on some dense subspaces of $\mathscr{M}$, as essentially self-adjoint operators there. 

\vs

The best way to describe these operators is to do so `locally'. Namely, for any chosen finite subset $J$ of $\mathbb{Q}^\times$, denote by $\mathscr{M}_J$ the set of all $f\in \mathscr{M}$ such that $J$ is the minimal finite subset of $\mathbb{Q}^\times$ such that the support of $f$ is contained in $\mathbb{R}^J \times \{0\} \subset \mathbb{R}^{Q^\times}$; we shall describe how the above operator acts on elements of $\mathscr{M}_J$. Notice that $\mathscr{M}_J$ is a subset of the Hilbert space $L^2(\mathbb{R}^J, \wedge_{j\in J} dx_j)$, inheriting the inner product structure. For the dense domains, we follow the arguments used by Fock and Goncharov in the quantization of Teichm\"uller spaces of finite-type surfaces, with a slight modification. As done in \cite{G}, let $W_{[j]}$ be the dense $\mathbb{C}$-vector subspace of $L^2(\mathbb{R}, dx_j)$ defined as the set of all finite $\mathbb{C}$-linear combinations of the functions of the form $p(x_j) \, e^{-Bx_j^2+Cx_j}$, where $p$ is a polynomial and $B>0$, $C\in \mathbb{C}$. The algebraic tensor product $W_J := \bigotimes_{j \in J} W_{[j]}$ is a dense subspace of the Hilbert space tensor product $\bigotimes_{j\in J} L^2(\mathbb{R}, dx_j)$, which is canonically isomorphic to $L^2(\mathbb{R}^J, \wedge_{j\in J} dx_j)$ as a Hilbert space. Each of the above formulas for $P_j$ and $Q_j$, $j\in J$, yields a well-defined symmetric operator on $W_J$, preserving $W_J$. It is well known that each of these is an `essentially self-adjoint' operator on $\mathscr{M}_J$, i.e. has a unique self-adjoint extension on $L^2(\mathbb{R}^J, \wedge_{j\in J} dx_j)$ (see \cite{K16}); by abuse of notation, let us denote these self-adjoint extensions by $P_j$ and $Q_j$. Then one makes sense of the formulas \eqref{eq:rho_A_j} and \eqref{eq:rho_T_jk} via the standard `functional calculus' of self-adjoint operators; see \cite{K16}, or a standard textbook, e.g. \cite{RSi80}. Such a description indeed suffices for our purposes, for each relation that is to be satisfied by the operators ${\bf A}_{[j]}, {\bf T}_{[j][k]}, {\bf P}_\gamma$ (see Prop.\ref{prop:lifted_Kashaev_relations}) involves only finitely many variables $x_j$ in a nontrivial way, and is just a permutation in the remaining variables, hence can be studied on $L^2(\mathbb{R}^J, \wedge_{j\in J} dx_j)$ for some finite set $J$.

\vs

Instead of just referring to the word `functional calculus', here we give a more down-to-earth description of these operators ${\bf A}_{[j]}$ and ${\bf T}_{[j][k]}$. Think of ${\bf A}_{[j]}$ \eqref{eq:rho_A_j} as an operator on the space $L^2(\mathbb{R}, dx_j)$ of square-integrable functions on one variable $x_j$; we can write
\begin{align}
\nonumber
L^2(\mathbb{R}, dx_j) \ni f(x_j) \longmapsto
({\bf A}_{[j]} f)(x_j) = e^{ - \pi i/12} \int_\mathbb{R} e^{2\pi i y_j x_j} e^{\pi i x_j^2} f(y_j)dy_j \, \in L^2(\mathbb{R}, dx_j),
\end{align}
which makes sense first on $W_{[j]} \subset L^2(\mathbb{R}, dx_j)$ for example, and then can be extended to the whole $L^2(\mathbb{R})$ by continuity, as it is unitary. If we write this as an operator ${\bf A} = e^{-\pi i/3} e^{3\pi i Q^2} e^{\pi i (P+Q)^2}$ acting on $L^2(\mathbb{R}, dx)$ for convenience, where $P$ and $Q$ are the self-adjoint operators on $L^2(\mathbb{R}, dx)$ given by $P = \frac{1}{2\pi i} \frac{d}{dx}$ and $Q = x$ on a dense subspace, then ${\bf A}$ is the unique unitary operator up to scalar multiplication by a complex number of modulus one that satisfies
\begin{align}
\label{eq:conjugation_action_of_bold_A}
{\bf A} Q {\bf A}^{-1} = P-Q, \quad
{\bf A} P {\bf A}^{-1} = -Q.
\end{align}
In a formal level, these two equations can be proved using the commutation relation $[P,Q]=\frac{1}{2\pi i}$. One can think of ${\bf A}$ as an analog of the Fourier transform $\mathscr{F} : f(x) \mapsto \int_{\mathbb{R}} e^{-2\pi i xy} f(y) dy$, which is characterized up to a constant by $\mathscr{F} Q \mathscr{F}^{-1} = -P$ and $\mathscr{F} P \mathscr{F}^{-1} = Q$.
For ${\bf T}_{[j][k]}$ \eqref{eq:rho_T_jk}, we view it as an operator on $L^2(\mathbb{R}^2, dx_j \, dx_k)$. First, the unitary operator $e^{2\pi i P_j Q_k}$ acts as
\begin{align}
\nonumber
(e^{2\pi i P_j Q_k} f)(x_j,x_k) = f(x_j + x_k, x_k).
\end{align}
One way to explain the factor $\Psi_b(Q_j+P_k-Q_k)^{-1}$ is as follows. First, write $\Psi_b(Q_j + P_k - Q_k)^{-1} = {\bf A}_{[k]} \Psi_b(Q_j + Q_k)^{-1} {\bf A}_{[k]}^{-1}$, using \eqref{eq:conjugation_action_of_bold_A}. We know how unitary operators ${\bf A}_{[k]}, {\bf A}_{[k]}^{-1}$ act. The operator $\Psi_b(Q_j + Q_k)^{-1}$ is just multiplication by $\Psi_b(x_j+x_k)^{-1}$. To see that this is indeed a unitary operator, we note that for $x\in \mathbb{R}$, the complex number $\Psi_b(x)$ is of modulus $1$.

\vs

The group homomorphism \eqref{eq:rho_F_dot_to_GL_mathscr_M} given by \eqref{eq:rho_A_j}, \eqref{eq:rho_T_jk} and \eqref{eq:rho_P_gamma} is indeed an almost-linear representation of $K$, in the sense of Def.\ref{def:almost_G-homomorphism}:
\begin{proposition}[\cite{Kash00}]
\label{prop:lifted_Kashaev_relations}
The operators ${\bf A}_{[j]}$, ${\bf T}_{[j][k]}$, ${\bf P}_\gamma$ in \eqref{eq:rho_A_j}, \eqref{eq:rho_T_jk}, \eqref{eq:rho_P_gamma} satisfy
\begin{align}
\nonumber
{\bf T}_{[j][k]} {\bf A}_{[j]} {\bf T}_{[k][j]} & = \zeta \, {\bf A}_{[j]} {\bf A}_{[k]} {\bf P}_{(jk)}, \quad j, k \in \mathbb{Q}^\times, \, j\neq k, \quad \mbox{where} \\
\label{eq:zeta}
\zeta & = e^{-\pi i (b+b^{-1})^2/12},
\end{align}
and strictly satisfy all other relations of $A_{[j]}$, $T_{[j][k]}$, $P_\gamma$ mentioned in Thm.\ref{thm:algebraic_relations_of_elementary_moves}:
\begin{align}
\label{eq:lifted_Kashaev_relations_major}
& {\bf A}_{[j]}^3 = {\rm id}, \quad
{\bf T}_{[k][\ell]} {\bf T}_{[j][k]} = {\bf T}_{[j][k]} {\bf T}_{[j][\ell]} {\bf T}_{[k][\ell]}, \quad
{\bf A}_{[j]} {\bf T}_{[j][k]} {\bf A}_{[k]} = {\bf A}_{[k]} {\bf T}_{[k][j]} {\bf A}_{[j]}, \\
\label{eq:lifted_Kashaev_relations_P}
& {\bf P}_{\rm id} = {\rm id},  \quad {\bf P}_{\gamma_1} {\bf P}_{\gamma_2} = {\bf P}_{\gamma_1 \circ \gamma_2}, \quad {\bf P}_\gamma {\bf A}_{[j]} = {\bf A}_{[\gamma(j)]} {\bf P}_\gamma, \quad
{\bf P}_\gamma {\bf T}_{[j][k]} = {\bf T}_{[\gamma(j) \, \gamma(k)]} {\bf P}_\gamma, 
\\
\label{eq:lifted_Kashaev_relations_commutation}
& {\bf T}_{[j][k]} {\bf T}_{[\ell][m]} = {\bf T}_{[\ell][m]} {\bf T}_{[j][k]}, \quad
{\bf T}_{[j][k]} {\bf A}_{[\ell]} = {\bf A}_{[\ell]} {\bf T}_{[j][k]}, \quad
{\bf A}_{[j]} {\bf A}_{[k]} = {\bf A}_{[k]} {\bf A}_{[j]}, 
\end{align}
for mutually distinct $j,k,\ell,m \in \mathbb{Q}^\times$ and any $\mathbb{Q}^\times$-permutations $\gamma,\gamma_1,\gamma_2$.
\end{proposition}
Construction of these operators ${\bf A}_{[j]}$, ${\bf T}_{[j][k]}$, ${\bf P}_\gamma$ is meaningful not just because they satisfy the above relations, but because they are `consistent' deformations of the coordinate change formulas for universal Teichm\"uller space that are induced by the transformations $A_{[j]}$, $T_{[j][k]}$, $P_\gamma$ of dotted tessellations of $\mathbb{D}$. However in the present paper, all that matter are the above relations in Prop.\ref{prop:lifted_Kashaev_relations}. Meanwhile, denote by ${\bf F}$ the natural map
\begin{align}
\label{eq:bold_F_on_free_group}
F_{\rm mark} \to F_{\rm dot}
\end{align}
given by the formula \eqref{eq:bold_F_alpha_beta}, by abuse of notation. Pulling back $\rho_{\rm dot}$ \eqref{eq:rho_F_dot_to_GL_mathscr_M} by ${\bf F}$ \eqref{eq:bold_F_on_free_group} yields
\begin{align}
\label{eq:rho_Kash}
\rho^{\rm Kash} = \rho_{\rm dot} \circ {\bf F}: F_{\rm mark} \to {\rm GL}(\mathscr{M}), \qquad\mathscr{M} = L^2_{\rm fin}(\mathbb{R}^{\mathbb{Q}^\times}),
\end{align}
(${\rm Kash}$ for Kashaev) whose images of the generators are
\begin{align}
\label{eq:rho_Kash_images}
\rho^{\rm Kash}(\alpha) = \wh{\alpha} := {\bf A}_{[-1]} {\bf T}^{-1}_{[-1][1]} {\bf A}_{[1]} {\bf P}_{\gamma_\alpha} \quad\mbox{and}\quad
\rho^{\rm Kash}(\beta) = \wh{\beta} := {\bf A}_{[-1]} {\bf P}_{\gamma_\beta}.
\end{align}
We can easily prove that $\rho^{\rm Kash}$ is an almost-linear representation of the Ptolemy-Thompson group $T$ on $\mathscr{M}$, which is to be compared with the almost-linear representation
\begin{align}
\label{eq:rho_CF}
\rho^{\rm CF} : F_{\rm mark} \to {\rm GL}(\mathscr{V}), \qquad \mathscr{V} = L^2_{\rm fin}(\mathbb{R}^{\mathbb{Q}\setminus\{0,1\}})
\end{align}
(${\rm CF}$ for Chekhov-Fock) of $T$ coming from the Chekhov-Fock(-Goncharov) quantization.

\vs

In order to obtain the explicit images of the generators of $F_{\rm mark}$ under $\rho^{\rm CF}$ \eqref{eq:rho_CF}, one should build the functor $Pt \to {\rm Hilb}$ \eqref{eq:Pt_to_Hilb} using Chekhov-Fock-Goncharov's analogous result $Pt(\Sigma) \to {\rm Hilb}$ for finite-type surfaces. Here the Ptolemy groupoid $Pt(\Sigma)$ for $\Sigma$ is the groupoid of all ideal triangulations of $\Sigma$, without distinguished oriented edges. Therefore morphisms of $Pt(\Sigma)$ are generated by `flips' along edges, so for quantum theory it suffices to know the operators which the flips are represented by, and this is how Chekhov-Fock-Goncharov's construction is written as. However, when we try to apply this idea to the groupoid $Pt$ of triangulations decorated with a d.o.e., some more work is needed if we want to give explicit formulas of the operators
$$
\rho^{\rm CF}(\alpha) = \wh{\alpha}_0 \quad\mbox{and}\quad
\rho^{\rm CF}(\beta) = \wh{\beta}_0
$$
representing $\alpha$ and $\beta$. Funar and Sergiescu \cite{FuS} deduced all the (lifted) $\alpha,\beta$-relations of \eqref{eq:T_presentation_intro} for the operators $\wh{\alpha}_0$, $\wh{\beta}_0$ from the relations shown in Fock-Goncharov \cite{FG}, without explicitly getting the formulas for the operators $\wh{\alpha}_0$, $\wh{\beta}_0$:
\begin{align}
\label{eq:lifted_T_relations_for_CF}
\left\{ \begin{array}{l}
(\wh{\beta}_0 \wh{\alpha}_0)^5 = \lambda, \quad \wh{\alpha}_0^4 = 1,\quad \wh{\beta}_0^3 = 1, \\
\left[\wh{\beta}_0 \wh{\alpha}_0 \wh{\beta}_0, \, \wh{\alpha}_0^2 \wh{\beta}_0 \wh{\alpha}_0 \wh{\beta}_0 \wh{\alpha}_0 ^2 \right] =1,\quad 
\left[\wh{\beta}_0 \wh{\alpha}_0 \wh{\beta}_0, \, \wh{\alpha}_0^2 \wh{\beta}_0 \wh{\alpha}_0^2 \wh{\beta}_0 \wh{\alpha}_0 \wh{\beta}_0 \wh{\alpha}_0^2 \wh{\beta}_0^2 \wh{\alpha}_0^2 \right]=1,
\end{array} \right.
\end{align}
where $\lambda$ is some complex number such that $|\lambda|=1$, depending on the parameter $b$.\footnote{As mentioned in Rem.\ref{rem:trivial_constant}, the author recently found by computation that $\lambda=1$. For the moment, regard $\lambda$ as a non-trivial formal variable.}

\vs

We call these almost-linear representations $\rho^{\rm Kash}$ and $\rho^{\rm CF}$ of $T$ `dilogarithmic', as both of them use the quantum dilogarithm function as a crucial tool.

\begin{remark}
All that matter for $\rho^{\rm CF}$ in the present paper are the relations \eqref{eq:lifted_T_relations_for_CF}, but it would be nice to get explicit formulas for $\wh{\alpha}_0 = \rho^{\rm CF}(\alpha)$ and $\wh{\beta}_0 = \rho^{\rm CF}(\beta)$.
\end{remark}

\subsection{Minimal central extensions resolving almost-linear representations}
\label{subsec:minimal_central_extensions}

It is well known (e.g. as pointed out in \cite{FuS} and \cite{FuKas}) that a projective representation $G \to {\rm PGL}(V)$ of $G$ on a vector space $V$ gives rise to a central extension $\til{G}$ of $G$, as the pullback by $p:{\rm GL}(V) \to {\rm PGL}(V)$ \eqref{eq:p} (that is, via fibre product). We can also construct central extensions of $G$ from an almost-linear representation $G\to {\rm GL}(V)$ of $G$, and among them we focus on the {\em minimal} central extension $\wh{G}$ of $G$ {\em resolving} the almost-linear representation, in the sense of Def.\ref{def:minimal_central_extension} (see e.g. \cite{FuS}, \cite{FuKas}). This extension $\wh{G}$ gives a more refined information on the almost-linear representation than $\til{G}$ does; in particular, $\til{G}$ is a central extension of $G$ by $k^\times= k\setminus\{0\}$ where $k$ is the base field of $V$, and $\wh{G}$ is an extension by a subgroup of $k^\times$. For our case, we have $G=T$, and $\til{T}$ is a central extension of $T$ by $\mathbb{C}^\times$, while the minimal extension $\wh{T}$ is a central extension of $T$ by $\mathbb{Z}$. Then we compute the class in $H^2(T;\mathbb{Z})$ corresponding to $\wh{T}$, and so finally we can compare different minimal central extensions $\wh{T}$ of $T$ inside $H^2(T;\mathbb{Z})$.

\vs

In the present subsection we gather and develop some group theoretical knowledge that is necessary for such constructions. For later use in the present paper, we use a bit more general setting than almost-linear representations; namely, we use almost $G$-homomorphisms, defined in Def.\ref{def:almost_G-homomorphism}. So, let $G$ be a group, presented with generators and relations
$$
G = F/R
$$
as in Def.\ref{def:almost_G-homomorphism}. A short way of stating the construction of the sought-for (`minimal') central extension of the group $G$ is as follows. Given an almost $G$-homomorphism
$$
\eta : F\to H
$$
for some group $H$, we get a central extension
\begin{align}
\label{eq:wh_G_first_construction}
\wh{G} = F/(R\cap \ker \eta)
\end{align}
of $G$ by the group isomorphic to $\eta(R)$. However, this construction is not really useful for presenting the resulting central extension $\widehat{G}$ by generators and relations.

\vs

We can construct $\widehat{G}$ \eqref{eq:wh_G_first_construction} more concretely, as follows. Let $Z$ be a group isomorphic to $\eta(R)$ and let us fix an isomorphism $\phi : \eta(R) \to Z$. Now we consider the free product $F*Z$, and let $R'$ be its normal subgroup generated by
\begin{align}
\label{eq:R_prime_generators}
r (\phi(\eta(r)))^{-1} ~~\mbox{(called the {\em lifted relations})}~~ \mbox{and}~~ [f,z]  ~~\mbox{(called the {\em commuting relations})},
\end{align}
for $r,f,z$ the generators of $R,F,Z$, respectively. This yields a central extension
\begin{align}
\label{eq:wh_G_second_construction}
\widehat{G} = F * Z/R'
\end{align}
of $G$ by $Z$. Using a presentation of $Z$ by generators and relations, this way we easily obtain the presentation for $\widehat{G}$ by generators and relations: lifted relations, commuting relations, and relations for (the central) $Z$. Moreover, we also obtain the natural lift of the original generators, by the group homomorphism
\begin{align}
\label{eq:lift_map}
\Psi : F \to F * Z/R'
\end{align}
induced by the inclusion $F\to F*Z$.

\vs

For completeness, let us prove:
\begin{lemma}
The two groups $F/(R\cap \ker\eta)$ \eqref{eq:wh_G_first_construction} and $F* Z/R'$ \eqref{eq:wh_G_second_construction} are isomorphic.
\end{lemma}
\begin{proof}
To avoid confusion, write $\wh{G} = F/(R\cap \ker\eta)$ and $\wh{G}_0 = F*Z / R'$ for the moment. As in \eqref{eq:lift_map} we have a group homomorphism $\Psi: F \to F * Z/R' = \wh{G}_0$. It suffices to prove that $\Psi$ is surjective, and that $\ker\Psi = R\cap \ker\eta$. For surjectivity, we should just show $z\in \Psi(F)$ for any $z\in Z$. We know from the relations $R'$ that for any $r\in R\subset F$ we have $r\equiv \phi(\eta(r))$ in $\wh{G}_0$. Since $\phi : \eta(R) \to Z$ is an isomorphism, for any $z\in Z$ there exists $r\in R$ such that $\phi(\eta(r)) \equiv z$ in $\wh{G}_0$. Then $\Psi(r)\equiv r\equiv \phi(\eta(r))\equiv z \in \wh{G}_0$, so $\Psi$ is surjective. Compose $\Psi$ with the map $\wh{G}_0 \to F/R$ which is quotienting by $Z$, to get $F \to F/R$ (which is easy to see), which coincides with just the natural projecting map from $F$ to $F/R$.
If $x\in \ker\Psi \subset F$, then $x$ is mapped by this composed map into the identity element of $F/R$, meaning that $x\in R$. Then, again using the relations $R'$ which say $r\equiv \phi(\eta(r))$ in $\wh{G}_0$ for any $r\in R$, we have $\Psi(x) \equiv x\equiv \phi(\eta(x))$ in $\wh{G}_0$. Now, $\Psi(x) \equiv 1$ if and only if $\eta(x)=1$, because $\phi : \eta(R) \to Z$ is an isomorphism and the natural map $Z\to \wh{G}_0$ is an injection ($\because$ it is not hard to see that the natural map from $Z$ to the subgroup $R* Z/(R'\cap(R*Z))$ of $\wh{G}_0$ is an isomorphism). Therefore $x\in \ker\eta$, hence $\ker\Psi = \ker\eta \cap R$, as desired.
\end{proof}

\vs

We can also formulate this construction of $\widehat{G} = F*Z/R'$ as follows. Let $G_1 := \eta(F) \le H$ and let ${\rm pr} : G_1 \to G_1/{\rm Center}(G_1)$ be the projection. We observe that ${\rm pr} \circ \eta : F \to G_1/{\rm Center}(G_1) = {\rm pr}(G_1)$ induces a well-defined group homomorphism
$$
G \to {\rm pr}(G_1),
$$
because ${\rm pr} \circ \eta$ sends $R$ to $1$. Now define $\widehat{G}$ as the pullback of $G$ along the map $G_1 \to {\rm pr}(G_1)$:
\begin{align}
\label{eq:conceptual_construction}
\begin{array}{l}
\xymatrix{
\widehat{G} \ar@{.>}[r] \ar@{.>}[d] & G \ar[d] \\
G_1 \ar[r] & {\rm pr}(G_1),
}
\end{array}
\end{align}
or the fibre product of $G\to {\rm pr}(G_1)$ and $G_1 \to {\rm pr}(G_1)$.
\begin{remark}
Suppose ${\rm Center}(G_1)\subset {\rm Center}(H)$. If we replace $G_1 \to {\rm pr}(G_1)$ by $H\to H/{\rm Center}(H)$ and $G\to {\rm pr}(G_1)$ by $G \to H/{\rm Center}(H)$ in the diagram \eqref{eq:conceptual_construction}, we get a possibly bigger central extension, as an analog of pulling back $G\to {\rm PGL}(V)$ along $p:{\rm GL}(V) \to {\rm PGL}(V)$.
\end{remark}

If the homomorphism $G\to {\rm pr}(G_1)$ is actually an isomorphism (we may call such $\eta$ a `faithful' almost $G$-homomorphism), it is easy to show that the resulting $\widehat{G} \to G_1$ is also an isomorphism. Thus for {\em any} central extension $\til{G}$ of $G$, by setting $H=\til{G}$ and choosing a set-map section $G \to \til{G}$, we have the notion of a {\em tautological} almost $G$-homomorphism $F\to \til{G}$, and it is not difficult to see that the central extension of $G$ obtained by the above procedure is indeed isomorphic to $\til{G}$:

\begin{definition}
\label{def:tautological_almost_group_homomorphisms}
Let $G = F/R$ be as in Def.\ref{def:almost_G-homomorphism}, and let $\til{G}$ be a central extension of $G$. An almost $G$-homomorphism (Def.\ref{def:almost_G-homomorphism}) $\eta: F \to \til{G}$ is said to be {\em tautological} if $\eta(F)=\til{G}$ holds and the induced map $G\to {\rm pr}(G_1) \cong G$ as described above is the identity map.
\end{definition}

\begin{proposition}
\label{prop:tautological_leads_to_itself}
The central extension of $G$ obtained by the above described procedure from a tautological almost $G$-homomorphism $F\to \til{G}$ is isomorphic to $\til{G}$.
\end{proposition}

We also introduce the notion of the equivalence of almost $G$-homomorphisms:

\begin{definition}
\label{def:equivalent_almost_group_homomorphisms}
Let $G=F/R$ be as in Def.\ref{def:almost_G-homomorphism} and $H_1,H_2$ be groups. Two almost $G$-homomorphisms (Def.\ref{def:almost_G-homomorphism}) $\eta_1: F\to H_1$ and $\eta_2 : F \to H_2$ are said to be {\em equivalent (via $\Phi_{12}$)} if the subgroups $G_j := \eta_j(F)$ of $H_j$ (for $j=1,2$) are isomorphic to each other and there is an isomorphism $\Phi_{12} : G_1 \to G_2$ such that $\Phi_{12} \circ \eta_1 = \eta_2$. In such a case, we write
\begin{align}
\nonumber
(\eta_1 : F \to H_1) \simeq_{\Phi_{12}} (\eta_2 : F\to H_2).
\end{align}
\end{definition}

It is easy to observe the following proposition and a lemma, which will be used later. 

\begin{proposition}
\label{prop:equivalent_almost_group_homomorphisms}
The equivalence of almost $G$-homomorphisms is an equivalence relation.
Equivalent almost $G$-homomorphisms yield isomorphic central extensions of $G$ via the above procedure. Also, the map $\Phi_{12}$ in Def.\ref{def:equivalent_almost_group_homomorphisms} provides an explicit isomorphism between the resulting central extensions.
\end{proposition}

\begin{lemma}
\label{lem:compositions_of_equivalent_almost_group_homomorphisms}
Let $G,F,R,H_1,H_2,\eta_1,\eta_2,G_1,G_2,\Phi_{12}$ be as in Def.\ref{def:equivalent_almost_group_homomorphisms}. Let $G' = F' / R'$ be another group presented with generators and relations. Then, if $\phi : F' \to F$ is a group homomorphism with $\phi(R') \subset R$, then the pre-compositions of $\eta_1,\eta_2$ with $\phi$ are equivalent almost $G'$-homomorphisms, i.e. $\eta_1\circ \phi \simeq \eta_2\circ \phi$, via an appropriate restriction of $\Phi_{12}$.

For a group $H_1'$, suppose that $\psi_1: G_1 \to H_1'$ is an injective group homomorphism. Then the post-composition of $\eta_1$ with $\psi_1$, i.e. $\psi_1\circ \eta_1 : F \to H_1'$, is an almost $G$-homomorphism and is equivalent to $\eta_1 : F \to H_1$ via $\psi_1^{-1} : \psi_1(G_1) \to G_1$.
\end{lemma}

Finally, we assert that the procedure described in the present subsection which is for almost $G$-homomorphisms is indeed an appropriate generalization of what we want from almost-linear representations:

\begin{proposition}
\label{prop:appropriate_generalization}
If we apply the procedure described in the present subsection to an almost-linear representation $\eta : F \to {\rm GL}(V)$ of $G=F/R$ (Def.\ref{def:almost_G-homomorphism}), we get the {\em minimal} central extension of $G$ resolving $\bar{\eta}$ (Def.\ref{def:minimal_central_extension}), where $\bar{\eta}: G \to {\rm GL}(V)$ is the map induced by $\eta$.
\end{proposition}

\subsection{Algebraic proof of the main theorem}
\label{subsec:algebraic_proof}

Coming back to our situation, we can now say that $\rho^{\rm CF}$ \eqref{eq:rho_CF} and $\rho^{\rm Kash}$ \eqref{eq:rho_Kash} are almost $T$-homomorphisms, where $T = F_{\rm mark}/R_{\rm mark}$ is presented with generators and relations. By the procedure described in \S\ref{subsec:minimal_central_extensions}, one obtains minimal central extensions of $T$ resolving $\rho^{\rm CF}$ and $\rho^{\rm Kash}$ ($\because$Prop.\ref{prop:appropriate_generalization}), denoted by
$$
\wh{T}^{\rm CF} \quad \mbox{and} \quad \wh{T}^{\rm Kash}
$$
respectively, which are given by some presentations with generators and relations. The main task of constructing these presentations is computing the `lifted relations'  \eqref{eq:R_prime_generators}, using the almost $T$-homomorphisms $\rho^{\rm CF}$ and $\rho^{\rm Kash}$. For $\rho^{\rm CF}$, such result is already written in \eqref{eq:lifted_T_relations_for_CF}, which is due to Funar and Sergiescu \cite{FuS}. In particular, we see by inspection that $\rho^{\rm CF}(R_{\rm mark}) \subset {\rm GL}(\mathscr{V})$ is generated by a single complex scalar $\lambda$, which means that $\wh{T}^{\rm CF}$ is a central extension of $T$ by $\mathbb{Z}$; namely, by replacing $\lambda$ in \eqref{eq:lifted_T_relations_for_CF} by $z$ and adding the commuting relations $[\wh{\alpha}_0,z]=[\wh{\beta}_0,z]=1$, we get a presentation for $\wh{T}^{\rm CF}$ generated by the symbols $\wh{\alpha}_0,\wh{\beta}_0,z$. In a similar way, we can also see that $\wh{T}^{\rm Kash}$ is a central extension by $\mathbb{Z}$. First, note that $\rho_{\rm dot}$ \eqref{eq:rho_F_dot_to_GL_mathscr_M} is an almost $K$-homomorphism, where $K = F_{\rm dot}/R_{\rm dot}$. Denote by $\wh{K}$ the central extension of $K$ resulting from $\rho_{\rm dot}$ by the procedure in \S\ref{subsec:minimal_central_extensions}. Prop.\ref{prop:lifted_Kashaev_relations} gives a complete list of lifted relations, and by inspection one observes that $\rho_{\rm dot}(R_{\rm dot}) \subset {\rm GL}(\mathscr{M})$ is generated by a single complex scalar $\zeta$, which means that $\wh{K}$ is a central extension of $K$ by $\mathbb{Z}$. Since $\wh{T}^{\rm Kash}$ is the pullback of $\wh{K}$ by ${\bf F} : T\to K$, we can indeed deduce that $\wh{T}^{\rm Kash}$ is a central extension of $T$ by $\mathbb{Z}$. A complete presentation for $\wh{T}^{\rm Kash}$ is the main result of the present paper (Thm.\ref{thm:main}), and one way of getting it is to translate the lifted version of the relations of the Kashaev group $K$ given in Prop.\ref{prop:lifted_Kashaev_relations} into a lifted version of the relations of the Ptolemy-Thompson group $T$, using ${\bf F}: T\to K$ stated in Prop.\ref{prop:our_bold_F}. We call this an `algebraic' proof of the main theorem, Thm.\ref{thm:main}; we present a major part of it here. 

\vs

Recall from \eqref{eq:rho_Kash_images} that we denoted by $\wh{\alpha}$ and $\wh{\beta}$ the operators corresponding to the generators $\alpha$, $\beta$ of $F_{\rm mark}$ by the almost $T$-homomorphism $\rho^{\rm Kash}$ \eqref{eq:rho_Kash}. So we have $\wh{\alpha} := {\bf A}_{[-1]} {\bf T}^{-1}_{[-1][1]} {\bf A}_{[1]} {\bf P}_{\gamma_\alpha}$ and $\wh{\beta} := {\bf A}_{[-1]} {\bf P}_{\gamma_\beta}$, where the $\mathbb{Q}^\times$-permutations $\gamma_\alpha$ and $\gamma_\beta$ are as described in Def.\ref{def:P_gamma_alpha_beta}. We now take each $\alpha,\beta$-relation in \eqref{eq:T_presentation_intro} and evaluate using $\wh{\alpha}$, $\wh{\beta}$. The strategy is to use the relations in Prop.\ref{prop:lifted_Kashaev_relations}. The following variants of the relations \eqref{eq:lifted_Kashaev_relations_major} will become handy:
\begin{align}
\label{eq:lifted_Kashaev_relations_more}
\left\{ {\renewcommand{\arraystretch}{1.4}
\begin{array}{l}
{\bf T}_{[j][k]}^{-1} {\bf A}_{[j]} {\bf A}_{[k]} = \zeta^{-1} {\bf A}_{[j]} {\bf P}_{(jk)} {\bf T}_{[j][k]}, \quad
{\bf T}_{[j][k]}^{-1} {\bf A}_{[j]}^2 {\bf A}_{[k]} = {\bf A}_{[j]}^2 {\bf A}_{[k]} {\bf T}_{[k][j]}^{-1}, \\
{\bf T}_{[j][k]}^{-1} {\bf T}_{[k][\ell]}^{-1} = {\bf T}_{[k][\ell]}^{-1} {\bf T}_{[j][\ell]}^{-1} {\bf T}_{[j][k]}^{-1}, \quad {\bf T}_{[k][j]}^{-1} {\bf A}_{[j]}^2 {\bf T}_{[j][k]}^{-1} = \zeta^{-1} {\bf P}_{(jk)} {\bf A}_{[j]}^2 {\bf A}_{[k]}^2.
\end{array}} \right.
\end{align}
We begin from the easier relations:
\begin{lemma}
\label{lem:easy_relations_representation}
The operators $\widehat{\alpha}$ and $\widehat{\beta}$ defined in \eqref{eq:rho_Kash_images} satisfy
\begin{align*}
\widehat{\beta}^3 = 1 \quad\mbox{and}\quad
\widehat{\alpha}^4 = \zeta^{-2}.
\end{align*}
\end{lemma}

\begin{proof}
Observe from \eqref{eq:gamma_beta} that $\gamma_\beta$ fixes $-1$; so ${\bf P}_{\gamma_\beta}$ commutes with ${\bf A}_{[-1]}$ ($\because$ \eqref{eq:lifted_Kashaev_relations_P}). Then
\begin{align*}
\widehat{\beta}^3 \, \stackrel{\eqref{eq:rho_Kash_images}}{=} \, ({\bf A}_{[-1]} {\bf P}_{\gamma_\beta})^3 = {\bf A}_{[-1]}^3 {\bf P}_{\gamma_\beta}^3 \, \stackrel{\eqref{eq:lifted_Kashaev_relations_major}}{=} \, {\bf P}_{\gamma_\beta}^3
\, \stackrel{\eqref{eq:lifted_Kashaev_relations_P}}{=} \, {\bf P}_{\gamma_\beta^3}.
\end{align*}
By applying the definition of $\gamma_\beta$ (as in Def.\ref{def:P_gamma_alpha_beta}) three times, it is easy to see that $\gamma_\beta^3$ is the identity permutation of $\mathbb{Q}^\times$, because it is the triangle-label permutation associated to the move $\beta^3 = {\rm id}$. Hence ${\bf P}_{\gamma_\beta^3} = {\bf P}_{\rm id} \stackrel{\eqref{eq:lifted_Kashaev_relations_P}}{=} {\rm id}$. This yields $\widehat{\beta}^3 = 1$.

\vs

We now take $\wh{\alpha}^2$. Note from \eqref{eq:gamma_alpha} that $\gamma_\alpha$ fixes $-1$ and $1$; so ${\bf P}_{\gamma_\alpha}$ commutes with ${\bf A}_{[-1]} {\bf T}_{[-1][1]}^{-1} {\bf A}_{[1]}$ ($\because$ \eqref{eq:lifted_Kashaev_relations_P}). Then
\begin{align}
\nonumber
\widehat{\alpha}^2 & \, \stackrel{\eqref{eq:rho_Kash_images}}{=} \, ({\bf A}_{[-1]} {\bf T}_{[-1][1]}^{-1} {\bf A}_{[1]} {\bf P}_{\gamma_\alpha})^2
= {\bf A}_{[-1]} {\bf T}_{[-1][1]}^{-1} \ul{ {\bf A}_{[1]}{\bf A}_{[-1]} } {\bf T}_{[-1][1]}^{-1} {\bf A}_{[1]} {\bf P}_{\gamma_\alpha}^2 \\
\nonumber
& \, \stackrel{\eqref{eq:lifted_Kashaev_relations_commutation}}{=} \, {\bf A}_{[-1]} \ul{ {\bf T}_{[-1][1]}^{-1} ({\bf A}_{[-1]}  {\bf A}_{[1]}) }{\bf T}_{[-1][1]}^{-1} {\bf A}_{[1]} {\bf P}_{\gamma_\alpha}^2 \\
\nonumber
& \, \stackrel{\eqref{eq:lifted_Kashaev_relations_more}}{=} \,
{\bf A}_{[-1]} (\zeta^{-1} {\bf A}_{[-1]} \ul{ {\bf P}_{(-1\, 1)} \cancel{ {\bf T}_{[-1][1]} } ) \cancel{ {\bf T}_{[-1][1]}^{-1} }  {\bf A}_{[1]}  } {\bf P}_{\gamma_\alpha}^2 \\
\nonumber
& \stackrel{\eqref{eq:lifted_Kashaev_relations_P}}{=} \zeta^{-1} \ul{ {\bf A}_{[-1]}^3 }  {\bf P}_{(-1 \, 1)} {\bf P}_{\gamma_\alpha}^2
\stackrel{\eqref{eq:lifted_Kashaev_relations_major}}{=}
\zeta^{-1} \, {\bf P}_{(-1\, 1)} {\bf P}_{\gamma_\alpha}^2,
\end{align}
where in each step we underlined the part which is being replaced in the next step. Since $\gamma_\alpha$ fixes $-1$ and $1$, we know ${\bf P}_{\gamma_\alpha}$ commutes with ${\bf P}_{(-1 \, 1)}$, and therefore
\begin{align*}
\widehat{\alpha}^4 = (\widehat{\alpha}^2)^2 = (\zeta^{-1} \, {\bf P}_{(-1\, 1)} {\bf P}_{\gamma_\alpha}^2)^2
= \zeta^{-2} \, \ul{ {\bf P}_{(-1\, 1)}^2 }  \, \ul{ {\bf P}_{\gamma_\alpha}^4 }
\stackrel{\eqref{eq:lifted_Kashaev_relations_P}}{=} \zeta^{-2} \, {\bf P}_{(-1 \, 1)^2} {\bf P}_{\gamma_\alpha^4} = \zeta^{-2} {\bf P}_{\gamma_\alpha^4}.
\end{align*}
By applying the definition of $\gamma_\alpha$ (as in Def.\ref{def:P_gamma_alpha_beta}) four times, it is easy to see that $\gamma_\alpha^4$ is the identity permutation of $\mathbb{Q}^\times$, because it is the triangle-label permutation associated to the move $\alpha^4 = {\rm id}$. Hence ${\bf P}_{\gamma_\alpha^4} = {\bf P}_{\rm id} = {\rm id}$. This yields $\widehat{\alpha}^4 = \zeta^{-2}$.
\end{proof}

From now on, the trivial step of switching the order of commuting factors  such as ${\bf A}_{[1]} {\bf A}_{[-1]} = {\bf A}_{[-1]} {\bf A}_{[1]}$ may not be explicitly shown. Proof of the following result is a little bit more involved.

\begin{lemma}
\label{lem:pentagon_representation}
The operators $\widehat{\alpha}$ and $\widehat{\beta}$ defined in \eqref{eq:rho_Kash_images} satisfy
\begin{align*}
(\widehat{\beta} \widehat{\alpha})^5 = \zeta^{-3}.
\end{align*}
\end{lemma}

\begin{proof}
From the definition \eqref{eq:rho_Kash_images} of $\wh{\alpha}$ and $\wh{\beta}$ we get
\begin{align}
\nonumber
\widehat{\beta} \widehat{\alpha}
& = ({\bf A}_{[-1]} \ul{ {\bf P}_{\gamma_\beta}) ({\bf A}_{[-1]} {\bf T}_{[-1][1]}^{-1} {\bf A}_{[1]} } {\bf P}_{\gamma_\alpha}) \\
\nonumber
& \stackrel{\eqref{eq:lifted_Kashaev_relations_P}}{=}
{\bf A}_{[-1]} {\bf A}_{[\gamma_\beta(-1)]} {\bf T}_{[\gamma_\beta(-1)][\gamma_\beta(1)]}^{-1} {\bf A}_{[\gamma_\beta(1)]}  \ul{ {\bf P}_{\gamma_\beta} {\bf P}_{\gamma_\alpha}} 
\stackrel{\eqref{eq:gamma_beta}, \, \eqref{eq:lifted_Kashaev_relations_P}}{=}
{\bf A}_{[-1]}^2 {\bf T}_{[-1][-\frac{1}{2}]}^{-1} {\bf A}_{[-\frac{1}{2}]}  {\bf P}_{\gamma_\beta \circ \gamma_\alpha}.
\end{align}
From \eqref{eq:gamma_alpha} and \eqref{eq:gamma_beta} we get $\gamma_\beta \circ \gamma_\alpha: -1 \mapsto -1$, $1\to -\frac{1}{2}$, $-\frac{1}{2} \mapsto 1$, 
and hence we can write
\begin{align}
\label{eq:def_gamma}
\gamma_\beta \circ \gamma_\alpha = (-\frac{1}{2} \, 1) \circ \gamma
\end{align}
for some permutation $\gamma$ of $\mathbb{Q}^\times$ which fixes $-1,-\frac{1}{2},1$. Now using \eqref{eq:lifted_Kashaev_relations_P} we write ${\bf P}_{\gamma_\beta \circ \gamma_\alpha} = {\bf P}_{(-\frac{1}{2} \, 1)} {\bf P}_\gamma$. Since $\gamma$ fixes $-1,-\frac{1}{2},1$, from \eqref{eq:lifted_Kashaev_relations_P} we know that ${\bf P}_\gamma$ commutes with the expression
\begin{align}
\label{eq:hat_beta_alpha_knot}
(\widehat{\beta}\widehat{\alpha})_\star := {\bf A}_{[-1]}^2 {\bf T}_{[-1][-\frac{1}{2}]}^{-1} {\bf A}_{[-\frac{1}{2}]} {\bf P}_{(-\frac{1}{2} \, 1)}.
\end{align}
So we have
\begin{align*}
\widehat{\beta} \widehat{\alpha} = (\widehat{\beta} \widehat{\alpha})_\star \, {\bf P}_\gamma = {\bf P}_\gamma \, (\widehat{\beta} \widehat{\alpha})_\star,
\end{align*}
and therefore
\begin{align*}
(\widehat{\beta} \widehat{\alpha})^5 = (\widehat{\beta} \widehat{\alpha})_\star^5 \, {\bf P}_\gamma^5
\,\stackrel{\eqref{eq:lifted_Kashaev_relations_P}}{=} \, (\widehat{\beta} \widehat{\alpha})_\star^5 \, {\bf P}_{\gamma^5}.
\end{align*}
Think of applying the move $(\beta\alpha)^5$ to the standard marked tessellation $\tau_{\rm mark}^*$ (ten elementary moves in total). By drawing the picture for each step, we can observe that all ideal triangles of $\tau_{\rm mark}^*$ (here ideal triangles are viewed as subsets of $\mathbb{D}$ without labels) remain intact during this whole process of ten moves, except the three which are labeled by $-1,-\frac{1}{2},1$ according to the labeling rule $L^*$ of $\mcal{F}(\tau_{\rm mark}^*) = \tau^*_{\rm dot} = (\tau^*, D^*, L^*)$.  And we know that $(\beta\alpha)^5$ is the identity move on the set of marked tessellations. Now, by following the definitions of $\gamma_\alpha$, $\gamma_\beta$ (Def.\ref{def:P_gamma_alpha_beta}) and $\gamma$ \eqref{eq:def_gamma}, one can deduce from pictures that $\gamma^5$ is the identity permutation of $\mathbb{Q}^\times$, thus ${\bf P}_{\gamma^5} = 1$ (in particular, we note that $(\gamma_\beta \circ\gamma_\alpha)^5 \neq {\rm id}$).

\vs

So it remains to prove $(\widehat{\beta} \widehat{\alpha})_\star^5 = \zeta^{-3}$. From its definition \eqref{eq:hat_beta_alpha_knot}, $(\widehat{\beta}\widehat{\alpha})_\star$ can be thought of as an operator on $L^2(\mathbb{R}^3, dx_{-1} \, dx_{-\frac{1}{2}} \, dx_1)$. For the ease of notation, we replace the subscripts $[-1],[-\frac{1}{2}],[1]$ with $1,2,3$ respectively. For example, ${\bf A}_{[-1]}$ will now be denoted by ${\bf A}_1$, and ${\bf T}_{[-1][-\frac{1}{2}]}$ by ${\bf T}_{12}$. The permutation operators will be denoted without the parentheses, e.g. ${\bf P}_{(-\frac{1}{2} \, 1)}$ will be denoted by ${\bf P}_{23}$. Then we now can rewrite \eqref{eq:hat_beta_alpha_knot} as:
\begin{align}
\label{eq:hat_beta_alpha_knot2}
(\widehat{\beta} \widehat{\alpha})_\star = {\bf A}_1^2 {\bf T}_{12}^{-1} {\bf A}_2 {\bf P}_{23} : L^2(\mathbb{R}^3, dx_1 \, dx_2\, dx_3) \longrightarrow L^2(\mathbb{R}^3, dx_1 \, dx_2\, dx_3).
\end{align}
We first note that
\begin{align}
\label{eq:hat_beta_alpha_knot_squared}
\left\{ {\renewcommand{\arraystretch}{1.4} \begin{array}{rl}
(\widehat{\beta} \widehat{\alpha})_\star^2
& = ({\bf A}_1^2 {\bf T}_{12}^{-1} {\bf A}_2 \ul{ {\bf P}_{23}) ({\bf A}_1^2 {\bf T}_{12}^{-1} {\bf A}_2 {\bf P}_{23} })
\stackrel{\eqref{eq:lifted_Kashaev_relations_P}}{=}
{\bf A}_1^2 \ul{ {\bf T}_{12}^{-1} {\bf A}_2 {\bf A}_1^2} {\bf T}_{13}^{-1} {\bf A}_3  \\
& \stackrel{\eqref{eq:lifted_Kashaev_relations_more}}{=} 
\ul{ {\bf A}_1^2 {\bf A}_1^2 } {\bf A}_2  {\bf T}_{21}^{-1}  {\bf T}_{13}^{-1} {\bf A}_3 
\stackrel{\eqref{eq:lifted_Kashaev_relations_major}}{=}
{\bf A}_1 {\bf A}_2 {\bf T}_{21}^{-1} {\bf T}_{13}^{-1} {\bf A}_3.
\end{array} } \right.
\end{align}
Putting together \eqref{eq:hat_beta_alpha_knot2} and \eqref{eq:hat_beta_alpha_knot_squared}, we get
\begin{align*}
(\widehat{\beta} \widehat{\alpha})_\star^5
& = (\widehat{\beta} \widehat{\alpha})_\star^2 (\widehat{\beta} \widehat{\alpha})_\star (\widehat{\beta} \widehat{\alpha})_\star^2
= (\ul{ {\bf A}_1 {\bf A}_2 {\bf T}_{21}^{-1} {\bf T}_{13}^{-1} {\bf A}_3)
({\bf A}_1^2 {\bf T}_{12}^{-1} {\bf A}_2 {\bf P}_{23} })
({\bf A}_1 {\bf A}_2 {\bf T}_{21}^{-1} {\bf T}_{13}^{-1} {\bf A}_3) \\
& \stackrel{\eqref{eq:lifted_Kashaev_relations_P}}{=}
( {\bf P}_{23} {\bf A}_1 {\bf A}_3 {\bf T}_{31}^{-1} {\bf T}_{12}^{-1} {\bf A}_2
{\bf A}_1^2 \ul{ {\bf T}_{13}^{-1} {\bf A}_3 )
{\bf A}_1 } {\bf A}_2 \ul{ {\bf T}_{21}^{-1} {\bf T}_{13}^{-1} } {\bf A}_3 \\
& \stackrel{\eqref{eq:lifted_Kashaev_relations_more}}{=}
{\bf P}_{23} {\bf A}_1 {\bf A}_3 {\bf T}_{31}^{-1} {\bf T}_{12}^{-1} {\bf A}_2
{\bf A}_1^2 (\zeta^{-1} {\bf A}_1 \underbrace{ {\bf P}_{13} \ul{ \cancel{ {\bf T}_{13} } ) {\bf A}_2 ( \cancel{ {\bf T}_{13}^{-1} } } {\bf T}_{23}^{-1} {\bf T}_{21}^{-1}  ) {\bf A}_3 } \\
& \stackrel{\eqref{eq:lifted_Kashaev_relations_P}}{=}
\zeta^{-1} {\bf P}_{23} {\bf A}_1 {\bf A}_3 {\bf T}_{31}^{-1} {\bf T}_{12}^{-1} \ul{ {\bf A}_2
\cancel{ {\bf A}_1^2 } \cancel{  {\bf A}_1 } ( {\bf A}_2 }  {\bf T}_{21}^{-1} {\bf T}_{23}^{-1}  {\bf A}_1 {\bf P}_{13} ) \\
& \stackrel{\eqref{eq:lifted_Kashaev_relations_major}}{=}
\zeta^{-1} {\bf P}_{23} {\bf A}_1 {\bf A}_3 {\bf T}_{31}^{-1} \ul{ {\bf T}_{12}^{-1} ({\bf A}_2^2)  {\bf T}_{21}^{-1} } {\bf T}_{23}^{-1}  {\bf A}_1 {\bf P}_{13}  \\
& \stackrel{\eqref{eq:lifted_Kashaev_relations_more}}{=}
\zeta^{-1} {\bf P}_{23} {\bf A}_1 {\bf A}_3 {\bf T}_{31}^{-1} (\zeta^{-1} \ul{ {\bf P}_{21} {\bf A}_2^2 {\bf A}_1^2 ) {\bf T}_{23}^{-1}  {\bf A}_1 } {\bf P}_{13}  \\
& \stackrel{\eqref{eq:lifted_Kashaev_relations_P}}{=}
\zeta^{-2} {\bf P}_{23} {\bf A}_1 {\bf A}_3 {\bf T}_{31}^{-1}  (  {\bf A}_1^2 \ul{ {\bf A}_2^2 {\bf T}_{13}^{-1} }  {\bf A}_2  {\bf P}_{12}) {\bf P}_{13}  \\
& \,\,\,\, = \,\,\,
\zeta^{-2} {\bf P}_{23} {\bf A}_1 {\bf A}_3 \ul{ {\bf T}_{31}^{-1}  {\bf A}_1^2 ( {\bf T}_{13}^{-1} } \,\,  \ul{ \cancel{ {\bf A}_2^2 }  ) \cancel{ {\bf A}_2 } } {\bf P}_{12} {\bf P}_{13}  \\
& \hspace{-4mm}\stackrel{\eqref{eq:lifted_Kashaev_relations_major}, \, \eqref{eq:lifted_Kashaev_relations_more}}{=}
\zeta^{-2} {\bf P}_{23} {\bf A}_1 {\bf A}_3 (\zeta^{-1} \ul{ {\bf P}_{13} {\bf A}_1^2 {\bf A}_3^2 ) {\bf P}_{12} {\bf P}_{13} }  \\
& \stackrel{\eqref{eq:lifted_Kashaev_relations_P}}{=}
\zeta^{-3} {\bf P}_{23} \underbrace{ \cancel{ {\bf A}_1 } \ul{ \cancel{ {\bf A}_3 }  ( \cancel{ {\bf A}_3^2 } } \cancel{ {\bf A}_1^2 }  } {\bf P}_{32} )
\stackrel{\eqref{eq:lifted_Kashaev_relations_major}}{=}
\zeta^{-3} \ul{ \cancel{ {\bf P}_{23} } \cancel{ {\bf P}_{32} } } 
\stackrel{\eqref{eq:lifted_Kashaev_relations_P}}{=}
\zeta^{-3}.
\end{align*}
\end{proof}

A key observation in the above proof is that the operator $(\widehat{\beta} \widehat{\alpha})^5$ defined on the space $\mathscr{M} = L^2_{\rm fin}(\mathbb{R}^{\mathbb{Q}^\times}, \wedge_{j\in \mathbb{Q}^\times} dx_j)$  (see \eqref{eq:L_2_fin}) of functions in variables $\{x_j\}_{j\in \mathbb{Q}^\times}$ acts in an interesting way (i.e. involving ${\bf A}_{\cdot}$'s and ${\bf T}_{\cdot\, \cdot}$'s) only for the three variables $x_{-1}, x_{-\frac{1}{2}}, x_1$, while acting as a permutation operator for the other variables, where this permutation is in fact the identity permutation. Similarly, when checking the remaining two relations, the relevant operators will act in an interesting way only for a finite number of variables, and act as a certain permutation operator for the others. Hence we can focus on those few variables as we have done above.

\begin{lemma}
\label{lem:remaining_relations}
The operators $\widehat{\alpha}$ and $\widehat{\beta}$ defined in \eqref{eq:rho_Kash_images} satisfy
\begin{align*}
(\widehat{\beta} \widehat{\alpha} \widehat{\beta}) (\widehat{\alpha}^2 \widehat{\beta} \widehat{\alpha} \widehat{\beta} \widehat{\alpha}^2) & = (\widehat{\alpha}^2 \widehat{\beta} \widehat{\alpha} \widehat{\beta} \widehat{\alpha}^2) (\widehat{\beta} \widehat{\alpha} \widehat{\beta}), \\
(\widehat{\beta} \widehat{\alpha} \widehat{\beta}) (\widehat{\alpha}^2 \widehat{\beta} \widehat{\alpha}^2 \widehat{\beta} \widehat{\alpha} \widehat{\beta} \widehat{\alpha}^2 \widehat{\beta}^2 \widehat{\alpha}^2)
& = (\widehat{\alpha}^2 \widehat{\beta} \widehat{\alpha}^2 \widehat{\beta} \widehat{\alpha} \widehat{\beta} \widehat{\alpha}^2 \widehat{\beta}^2 \widehat{\alpha}^2) (\widehat{\beta} \widehat{\alpha} \widehat{\beta}).
\end{align*}
\end{lemma}
Proof of these relations can be done similarly as in the proof of Lem.\ref{lem:easy_relations_representation} and Lem.\ref{lem:pentagon_representation}. In fact, we only need to use the relations \eqref{eq:lifted_Kashaev_relations_P} and \eqref{eq:lifted_Kashaev_relations_commutation}, and the only tricky part is to keep good track of the subscript indices. We consider this computation to be trivial, and since it can be checked by any interested reader, we omit it here. See \cite{Ki} for a full calculation.

\vs

From Lemmas \ref{lem:easy_relations_representation}, \ref{lem:pentagon_representation} and \ref{lem:remaining_relations} one can easily see that the image $\rho^{\rm Kash}(R_{\rm mark}) \subset \mathbb{C}^\times$ of the group $R_{\rm mark} \subset F_{\rm mark}$ of $\alpha,\beta$-relations of $T$ under the almost $T$-homomorphism $\rho^{\rm Kash}$ \eqref{eq:rho_Kash} is generated by a single complex number $\zeta^{-1} \in {\rm U}(1)$, which depends on the parameter $b$ of the quantization. Applying the procedure in \S\ref{subsec:minimal_central_extensions} to $\rho^{\rm Kash}$, we obtain the central extension $\wh{T}^{\rm Kash}$ of $T$ by $\mathbb{Z}$, presented with generators $\ol{\alpha}$, $\ol{\beta}$, $z$, where $\ol{\alpha}$, $\ol{\beta}$ are the lifts of $\alpha$, $\beta$ of $T$ and $z$ is the generator of the center $\mathbb{Z}$ (i.e. the kernel of the central extension), where the lifted relations is obtained by replacing $\wh{\alpha}, \wh{\beta}, \zeta^{-1}$ in the results of Lemmas \ref{lem:easy_relations_representation}, \ref{lem:pentagon_representation}, \ref{lem:remaining_relations} by $\ol{\alpha}, \ol{\beta}, z$, respectively, and the commutation relations $\left[\bar{\alpha},z\right]=\left[\bar{\beta},z\right]=1$ are added:
\begin{align}
\label{eq:wh_T_Kash_presentation}
\left\{ {\renewcommand{\arraystretch}{1.2}
\begin{array}{l}
(\bar{\beta}\bar{\alpha})^5 = z^3, \qquad \bar{\alpha}^4 = z^2, \qquad \bar{\beta}^3= 1, \\
\left[\bar{\beta}\bar{\alpha}\bar{\beta}, \, \bar{\alpha}^2 \bar{\beta}\bar{\alpha}\bar{\beta}\bar{\alpha}^2\right]= \left[\bar{\beta}\bar{\alpha}\bar{\beta}, \, \bar{\alpha}^2\bar{\beta}\bar{\alpha}^2\bar{\beta}\bar{\alpha}\bar{\beta}\bar{\alpha}^2\bar{\beta}^2\bar{\alpha}^2\right]= \left[\bar{\alpha},z\right]=\left[\bar{\beta},z\right]=1.
\end{array} } \right.
\end{align}

\vs

To finish the proof of Thm.\ref{thm:main}, we recall the result of Funar and Sergiescu \cite{FuS}, which gives a classification of all possible central extensions of $T$ by $\mathbb{Z}$, their presentations, and a way of computing the corresponding extension classes in $H^2(T; \mathbb{Z})$. Their result is gathered in Thm.\ref{thm:FS_classification} in \S\ref{sec:introduction} of the present paper. From the above presentation of $\wh{T}^{\rm Kash}$ we see that it is isomorphic to the group $T_{3,2,0,0}$ appearing in Thm.\ref{thm:FS_classification}, and using the  formula in Thm.\ref{thm:FS_classification} we can compute the extension class of this central extension $\wh{T}^{\rm Kash}$ of $T$ to be $6\chi \in H^2(T;\mathbb{Z})$, where $\chi$ is the `Euler class' (see \eqref{eq:second_cohomology_group_of_T}). This finishes the `algebraic' proof of our main theorem. Analogous result for $\wh{T}^{\rm CF}$ obtained in \cite{FuS} is written in Thm.\ref{thm:FS} in \S\ref{sec:introduction}, which in particular shows that $\wh{T}^{\rm Kash}$ and $\wh{T}^{\rm CF}$ are inequivalent central extensions of $T$, as the extension class of $\wh{T}^{\rm CF}$ is $12\chi$.

\section{Topological proof of the main theorem}

In this section we present a `topological' proof of the main theorem: $\wh{T}^{\rm Kash} \cong T_{3,2,0,0}$. We introduce infinitely-punctured unit discs $\mathbb{D}^*$, $\mathbb{D}^\sharp$, and their asymptotically rigid mapping class groups $T^*$, $T^\sharp$, which are extensions of $T$ by the infinite braid group $B_\infty$. By abelianizing the kernel $B_\infty$ we get central extensions $T^*_{\rm ab}$, $T^\sharp_{\rm ab}$ of $T$ by $\mathbb{Z}$. The strategy is to prove $\wh{T}^{\rm Kash} \cong T^\sharp_{\rm ab}$ (Prop.\ref{prop:T_sharp_ab}) and $T^\sharp_{\rm ab} \cong T_{3,2,0,0}$ separately. The latter can be easily checked topologically, so the main point is to check $\wh{T}^{\rm Kash} \cong T^\sharp_{\rm ab}$. For this, we introduce a version of Kashaev group $K^\sharp$ for $\mathbb{D}^\sharp$, which is an extension of $K$ by $B_\infty$ and yields a central extension $K^\sharp_{\rm ab}$ of $K$ by abelianizing $B_\infty$. Natural group homomorphisms ${\bf F}^\sharp : T^\sharp \to K^\sharp$ and ${\bf F}^\sharp_{\rm ab} : T^\sharp_{\rm ab} \to K^\sharp_{\rm ab}$ are constructed in a similar way as in \S\ref{subsec:mcal_F} and \S\ref{subsec:bold_F}, which enable us to deduce the desired result.

\subsection{Infinitely-punctured unit discs and their tessellations}

Following Funar and Kapoudjian \cite{FuKa2}, we introduce infinite number of punctures in the unit disc $\mathbb{D}$ in certain ways, and consider a construction analogous to what is done in \S\ref{sec:decorated_universal_Ptolemy_groupoids} of the present paper, for these new infinitely-punctured surfaces.

\begin{definition}
\label{def:punctured_discs}
Denote by $\mathbb{D}^*$ (resp. $\mathbb{D}^\sharp$) the open unit disc $\mathbb{D}$ with infinitely many punctures (depicted as $\circ$ in the pictures, to avoid confusion with the dots $\bullet$ for dotted tessellations), where the position of the punctures are chosen once and for all, in the following way. Choose one point in the interior of each ideal arc (resp. one point in the interior of each ideal triangle) of the Farey tessellation of Def.\ref{def:Farey_tessellation} where we assume here that all the ideal arcs are stretched to geodesics with respect to the Poincar\'e hyperbolic metric; these points comprise the punctures for $\mathbb{D}^*$ (resp. for $\mathbb{D}^\sharp$). We call the punctures of $\mathbb{D}^*$ (resp. $\mathbb{D}^\sharp$) the {\em $*$-punctures} (resp. {\em $\sharp$-punctures}).
\end{definition}
Any homotopy of $\mathbb{D}^*$ (resp. $\mathbb{D}^\sharp$) is assumed to pointwise fix every point on the boundary $S^1 = \partial \mathbb{D}$ {\em and} every $*$-puncture (resp. $\sharp$-puncture) at all times.
\begin{definition}
\label{def:ideal_arc_in_punctured_discs}
An {\em ideal arc in $\mathbb{D}^*$} (resp. {\em ideal arc in $\mathbb{D}^\sharp$}) connecting two given distinct rational points on $S^1=\partial \mathbb{D}$ is a homotopy class of unoriented paths connecting the two points, while the homotopy requires that each ideal arc should pass through exactly one $*$-puncture at all times (resp. that each ideal arc should not pass through any $\sharp$-puncture at any time). An {\em ideal triangle} in $\mathbb{D}^*$ (resp. in $\mathbb{D}^\sharp$) is a triangle with three distinct vertices on $S^1$ whose sides are ideal arcs in $\mathbb{D}^*$ (resp. in $\mathbb{D}^\sharp$).
\end{definition}

\begin{remark}
It may be more natural to view the $*$-punctures of $\mathbb{D}^*$ as distinguished points instead of punctures, if an ideal arc in $\mathbb{D}^*$ is described as above.
\end{remark}

\begin{remark}
There are infinitely many distinct ideal arcs in $\mathbb{D}^*$ (resp. ideal arcs in $\mathbb{D}^\sharp$) connecting given two distinct rational points on $S^1$. Recall that for the non-punctured case there is a unique ideal arc connecting any given two distinct points on $S^1$.
\end{remark}

\begin{definition}
\label{def:diamond_Farey}
A {\em $*$-Farey ideal arc} (resp. {\em $\sharp$-Farey ideal arc}) is the homotopy class of ideal arcs in $\mathbb{D}^*$ (resp. in $\mathbb{D}^\sharp$) homotopic  to an ideal arc of the Farey tesssellation of $\mathbb{D}$ (Def.\ref{def:Farey_tessellation}) stretched to the hyperbolic geodesic for the Poincar\'e metric.

\vs

A {\em Farey-type $*$-punctured tessellation of $\mathbb{D}^*$} (resp. {\em Farey-type $\sharp$-punctured tessellation of $\mathbb{D}^\sharp$}) is a Farey-type tessellation of $\mathbb{D}$ (Def.\ref{def:Farey-type_tessellations}) such that each ideal arc passes through exactly one $*$-puncture of $\mathbb{D}^*$ while every $*$-puncture is being passed by one arc (resp. such that each ideal triangle contains in its interior exactly one $\sharp$-puncture), and such that all but finitely many ideal arcs are $*$-Farey ideal arcs (resp. $\sharp$-Farey ideal arcs). In this definition, each ideal arc should be thought of as an ideal arc in $\mathbb{D}^*$ (resp. in $\mathbb{D}^\sharp$), in the sense of Def.\ref{def:ideal_arc_in_punctured_discs}. Ideal arcs constituting Farey type $*$-punctured or $\sharp$-punctured tessellation are also called {\em edges}.
\end{definition}

From now on, we only use Farey type $*$-punctured or $\sharp$-punctured tessellations, so we omit the word `Farey-type'.

\begin{definition}
Let $\diamond = *$ or $\sharp$. A {\em marked $\diamond$-punctured tessellation of $\mathbb{D}^\diamond$} is a $\diamond$-punctured tessellation of $\mathbb{D}^\diamond$ together with the choice of a distinguished oriented edge, called {\em d.o.e.}

\vs

A {\em dotted $\diamond$-punctured tessellation of $\mathbb{D}^\diamond$} is a $\diamond$-punctured tessellation of $\mathbb{D}^\diamond$ together with the choice of a distinguished corner for each ideal triangle denoted by a filled dot $\bullet$ in the pictures, and the choice of a way of labeling ideal triangles by $\mathbb{Q}^\times$, i.e. a bijection between the ideal triangles and $\mathbb{Q}^\times$, where the triangle labeled by $j \in \mathbb{Q}^\times$ is indicated by $[j]$ in the picture.

\vs

For $\diamond\in \{ *,\sharp\}$, we denote by $Pt^\diamond$ (resp. $Pt^\diamond_{\rm dot}$) the groupoid whose objects are all possible $\diamond$-punctured marked tessellations (resp. $\diamond$-punctured dotted tessellations) of $\mathbb{D}^\diamond$, where, from any object to any object there is a unique morphism. 
\end{definition}

Funar and Kapoudjian \cite{FuKa2} in fact used another infinite surface, namely the {\em ribbon tree}, obtained by thickening the binary tree in the plane. There is a one-to-one correspondence between hexagon decompositions of the ribbon tree and tessellations of the unit disc, which can easily be understood via Fig.\ref{fig:dualizing_tessellation_to_ribbon_graph}. Choice of a d.o.e. of a tessellation can be realized as choice of an ordered pair of two adjacent hexagons of the ribbon tree, and a dotted tessellation also can be realized by some combinatorial decoration on the hexagon decomposition of the ribbon tree. We refer the readers to \cite{Ki} for a more detailed discussion on the correspondence between these two models. 

\begin{figure}[htbp!]
\centering
\includegraphics[width=40mm]{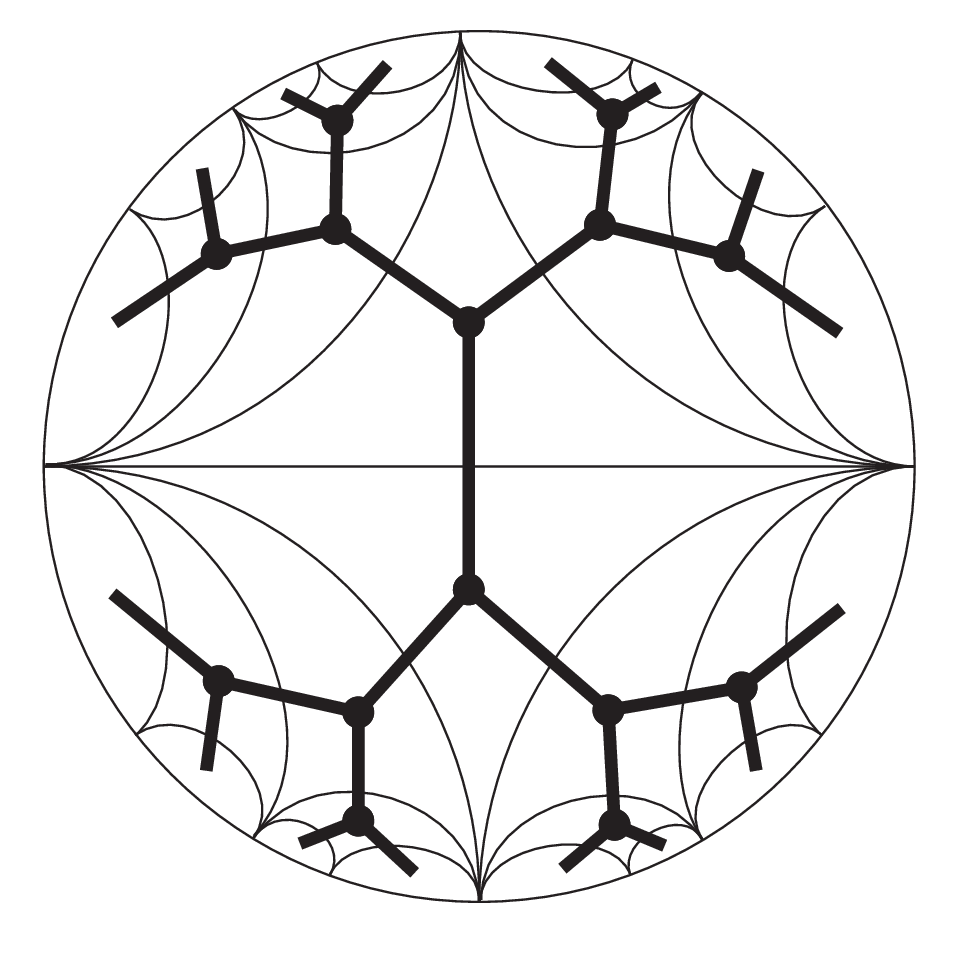}
\begin{pspicture}[showgrid=false,linewidth=0.5pt,unit=7.5mm](-0.5,-2)(0.8,2.0)
\rput[l](-0.5,0.3){\pcline[linewidth=0.7pt, arrowsize=2pt 4]{<->}(0,0)(1.5;0)\Aput{{\rm dualize}}}
\end{pspicture}
\includegraphics[width=40mm]{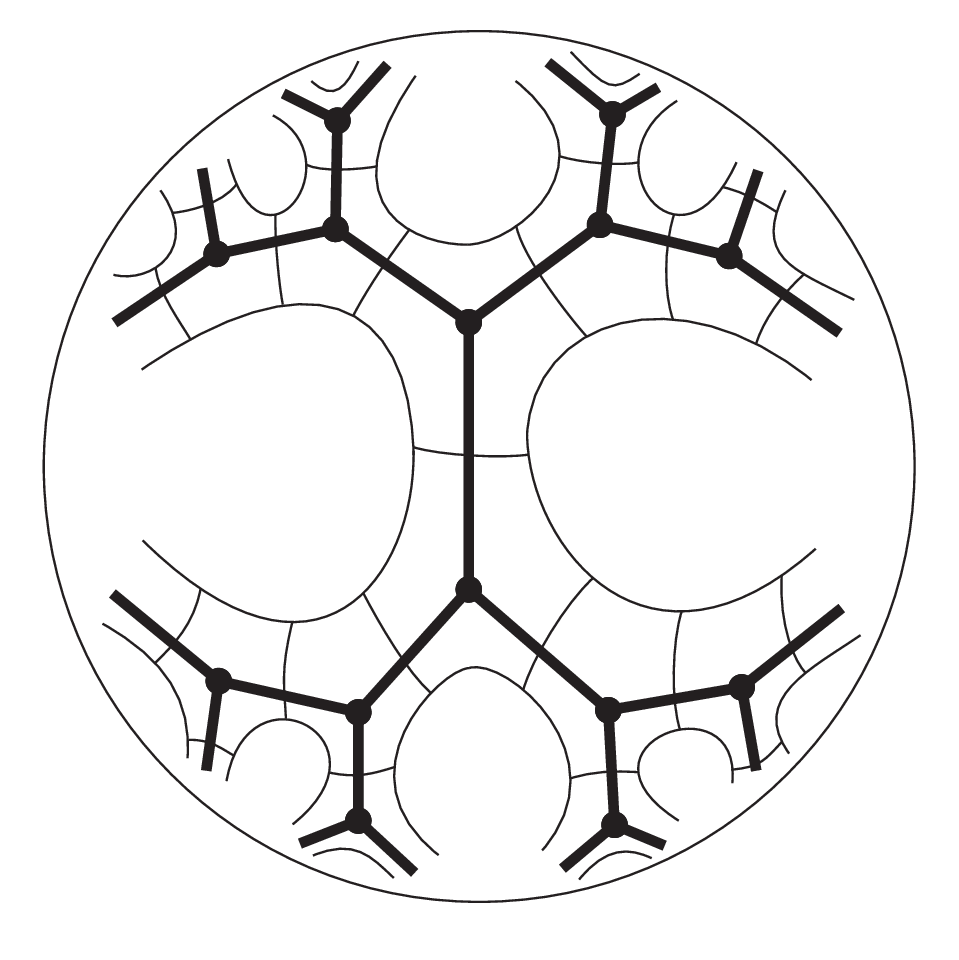}

\vspace{-5mm}

\caption{Dualizing between a tessellation of $\mathbb{D}$ and an infinite ribbon tree}
\label{fig:dualizing_tessellation_to_ribbon_graph}
\end{figure}

\vspace{-3mm}

\subsection{Braided Ptolemy-Thompson groups $T^*$, $T^\sharp$}
\label{subsec:braided_Ptolemy-Thompson_groups}

Analogously to the non-punctured case, there are some `elementary' morphisms of $Pt^\diamond$ (for $\diamond\in \{ *,\sharp\}$) which generate all the morphisms of the whole groupoid, and which have combinatorial descriptions. As in \S\ref{subsec:T_and_K}, each morphism of $Pt^\diamond$ can be viewed as a transformation of a marked $\diamond$-punctured tessellation of $\mathbb{D}^\diamond$ into another. We first define analogs of the $\alpha$-move and the $\beta$-move in $Pt^*$ and $Pt^\sharp$.
\begin{definition}
\label{def:braided_alpha_and_beta}
We label a morphism of $Pt^*$ by $\alpha^*$ (resp. $\beta^*$) if it transforms a marked $*$-punctured tessellation of $\mathbb{D}^*$ as in Fig.\ref{fig:action_of_alpha_star} (resp. as in Fig.\ref{fig:action_of_beta_star}), leaving all other parts indicated by triple dots `$\cdots$' in Fig.\ref{fig:action_of_alpha_star} (resp. in Fig.\ref{fig:action_of_beta_star}) intact. In other words, $\alpha^*$-move rotates the d.o.e. counterclockwise to the other diagonal of the ideal quadrilateral containing the d.o.e., and the $\beta^*$-move just alters the choice of d.o.e. as the $\beta$-move of $Pt$ does.

\vspace{-4mm}
\begin{figure}[htbp!]
$\begin{array}{ll}
\hspace{-5mm}
\includegraphics[width=35mm]{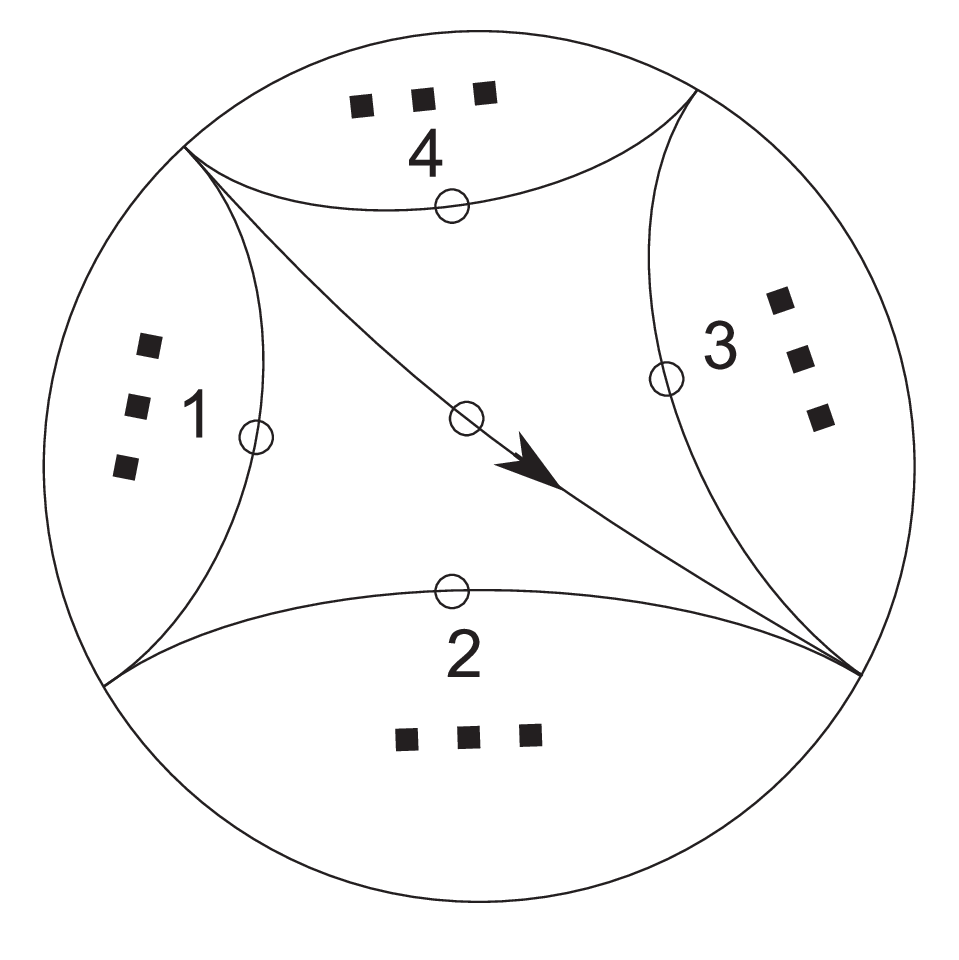}
\begin{pspicture}[showgrid=false,linewidth=0.5pt,unit=7.5mm](-0.5,-1.5)(0.1,2.0)
\rput[l](-0.7,0.6){\pcline[linewidth=0.7pt, arrowsize=2pt 4]{->}(0,0)(1.0;0)\Aput{$\alpha^*$}}
\end{pspicture}
\includegraphics[width=35mm]{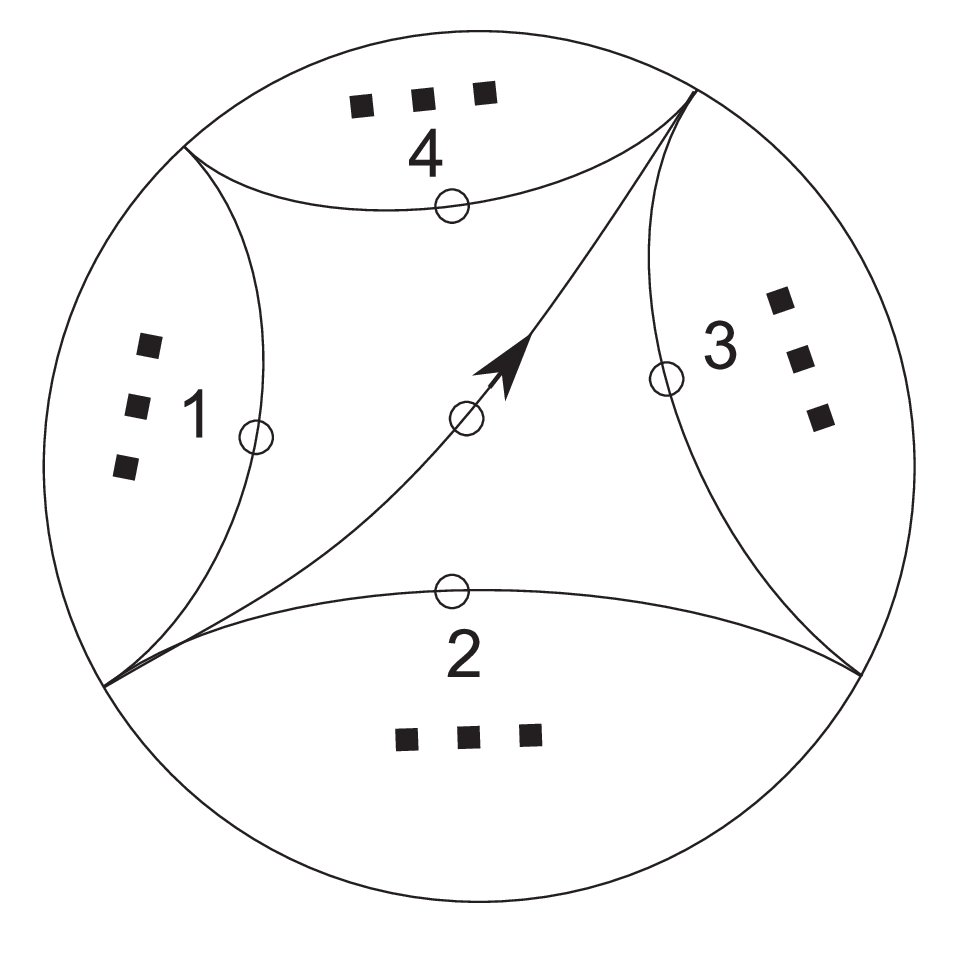}
&
%
\centering
\includegraphics[width=35mm]{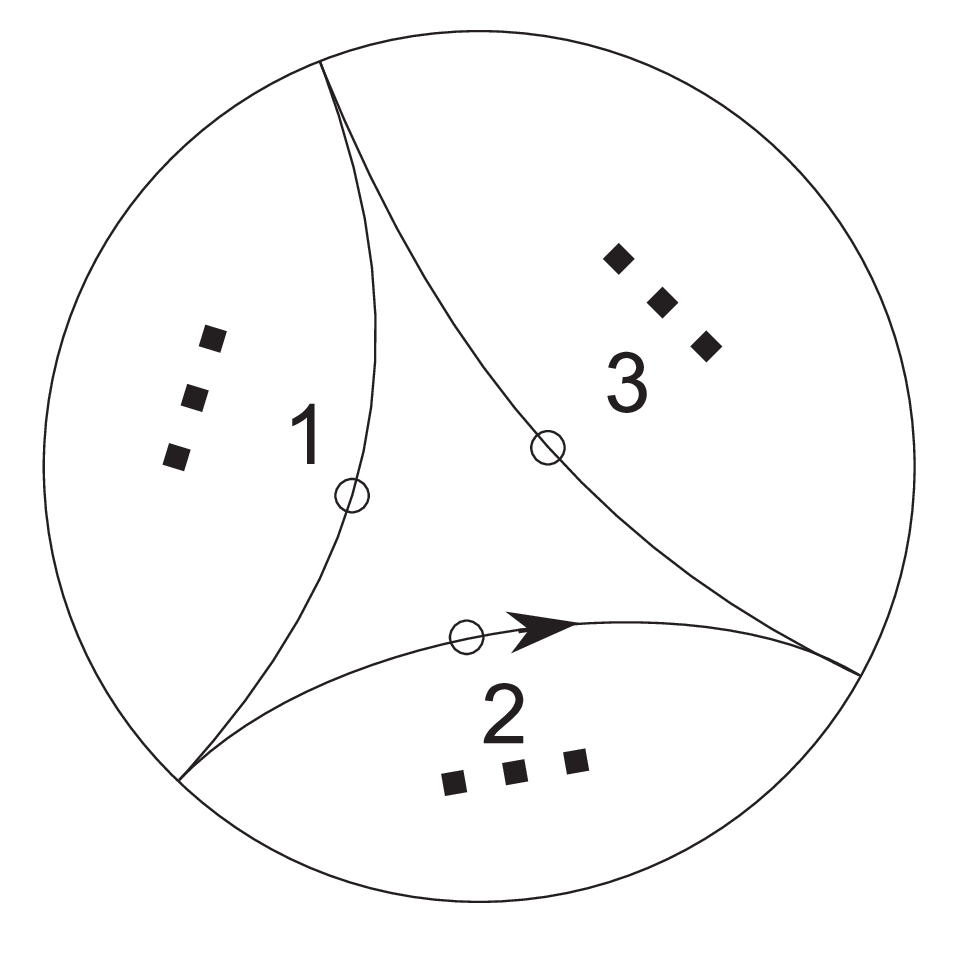}
\begin{pspicture}[showgrid=false,linewidth=0.5pt,unit=7.5mm](-0.5,-1.5)(0.1,2.0)
\rput[l](-0.7,0.6){\pcline[linewidth=0.7pt, arrowsize=2pt 4]{->}(0,0)(1.0;0)\Aput{$\beta^*$}}
\end{pspicture}
\includegraphics[width=35mm]{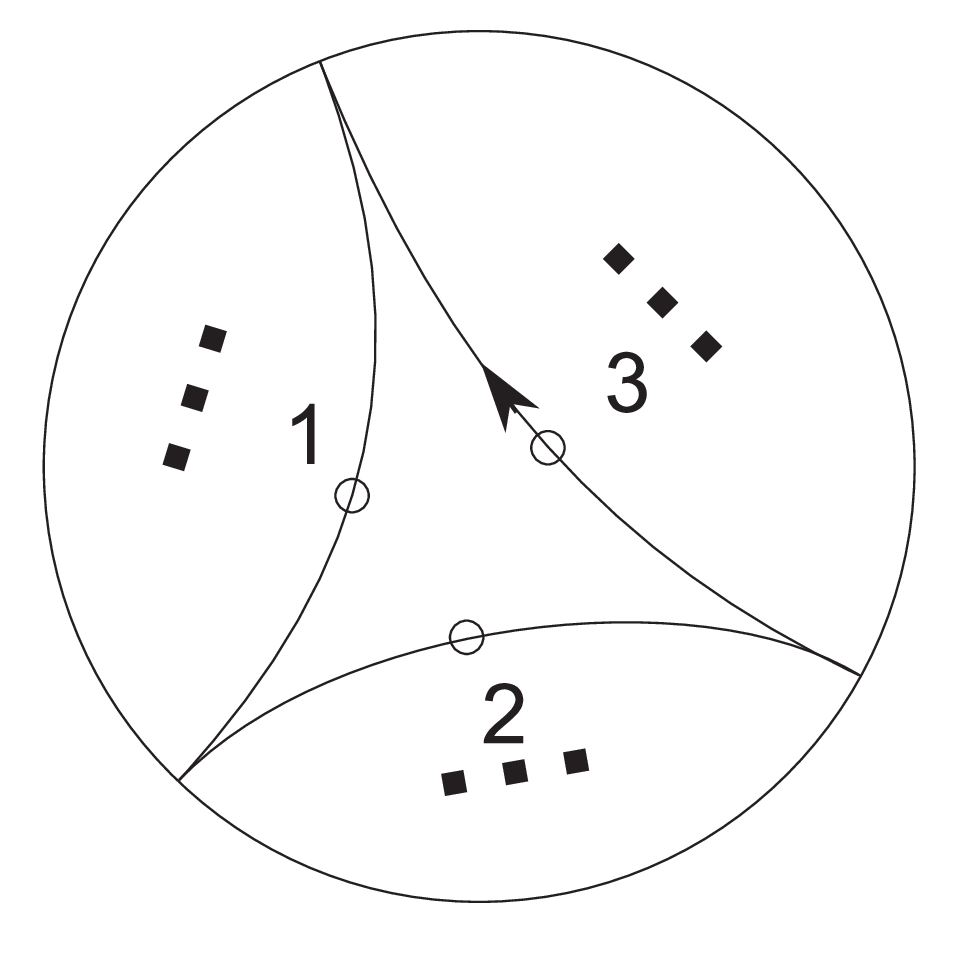}
\end{array}$
\\
\vspace{-4mm}

\begin{subfigure}[b]{0.48\textwidth}
\caption{The action of $\alpha^*$ on $Pt^*$}
\label{fig:action_of_alpha_star}
\end{subfigure}
\hfill
\begin{subfigure}[b]{0.5\textwidth}
\caption{The action of $\beta^*$ on $Pt^*$}
\label{fig:action_of_beta_star}
\end{subfigure}

\vspace{-4mm}

\caption{Some elementary morphisms of $Pt^*$}
\label{fig:some_elementary_morphisms_of_Pt_star}
\end{figure}

\vspace{-2mm}

\vs

We label a morphism of $Pt^\sharp$ by $\alpha^\sharp$ (resp. $\beta^\sharp$) if it transforms a marked $\sharp$-punctured tessellation of $\mathbb{D}^\sharp$ as in Fig.\ref{fig:action_of_alpha_sharp} (resp. as in Fig.\ref{fig:action_of_beta_sharp}), leaving all other parts indicated by triple dots `$\cdots$' in Fig.\ref{fig:action_of_alpha_sharp} (resp. in Fig.\ref{fig:action_of_beta_sharp}) intact. In other words, $\alpha^\sharp$-move rotates the d.o.e. counterclockwise to the other diagonal of the ideal quadrilateral containing the d.o.e. without touching the $\sharp$-punctures, and the $\beta^\sharp$-move just alters the choice of d.o.e. as the $\beta$-move of $Pt$ does.

\begin{figure}[htbp!]
$\begin{array}{ll}
\hspace{-5mm}
\includegraphics[width=35mm]{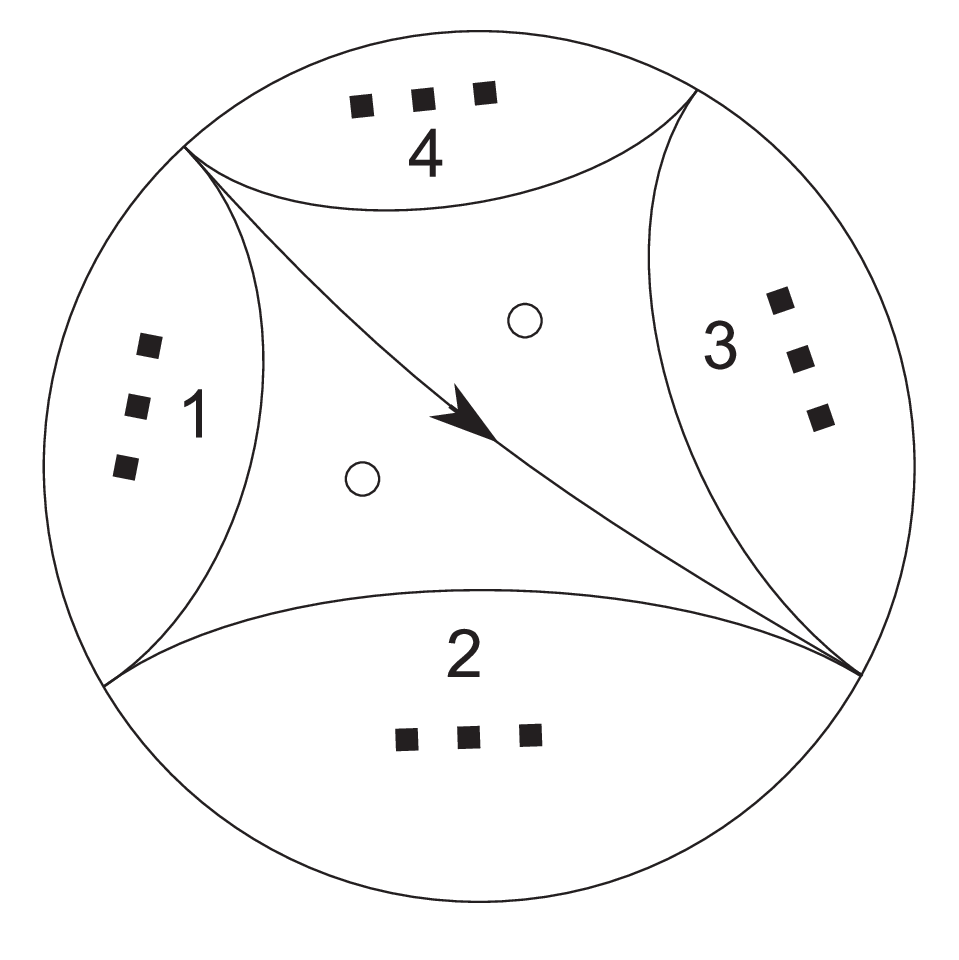}
\begin{pspicture}[showgrid=false,linewidth=0.5pt,unit=7.5mm](-0.5,-1.5)(0.1,2.0)
\rput[l](-0.7,0.6){\pcline[linewidth=0.7pt, arrowsize=2pt 4]{->}(0,0)(1.0;0)\Aput{$\alpha^\sharp$}}
\end{pspicture}
\includegraphics[width=35mm]{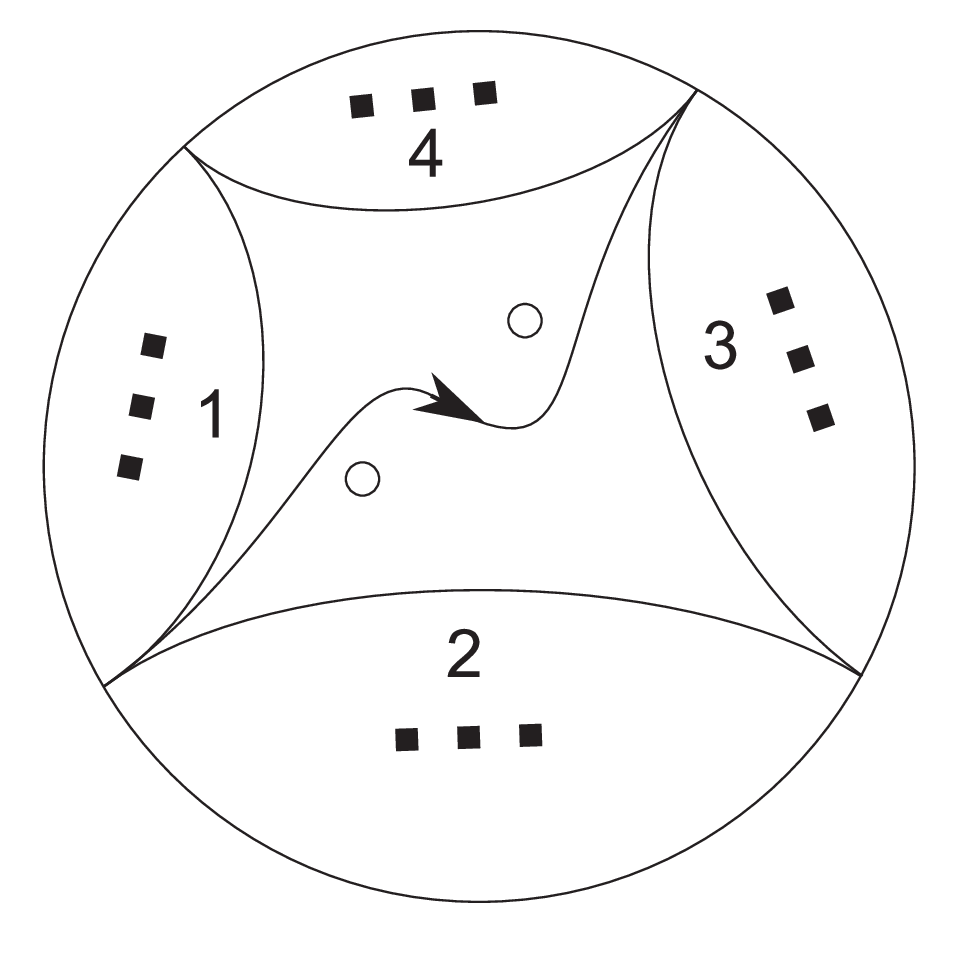}
&
%
\centering
\includegraphics[width=35mm]{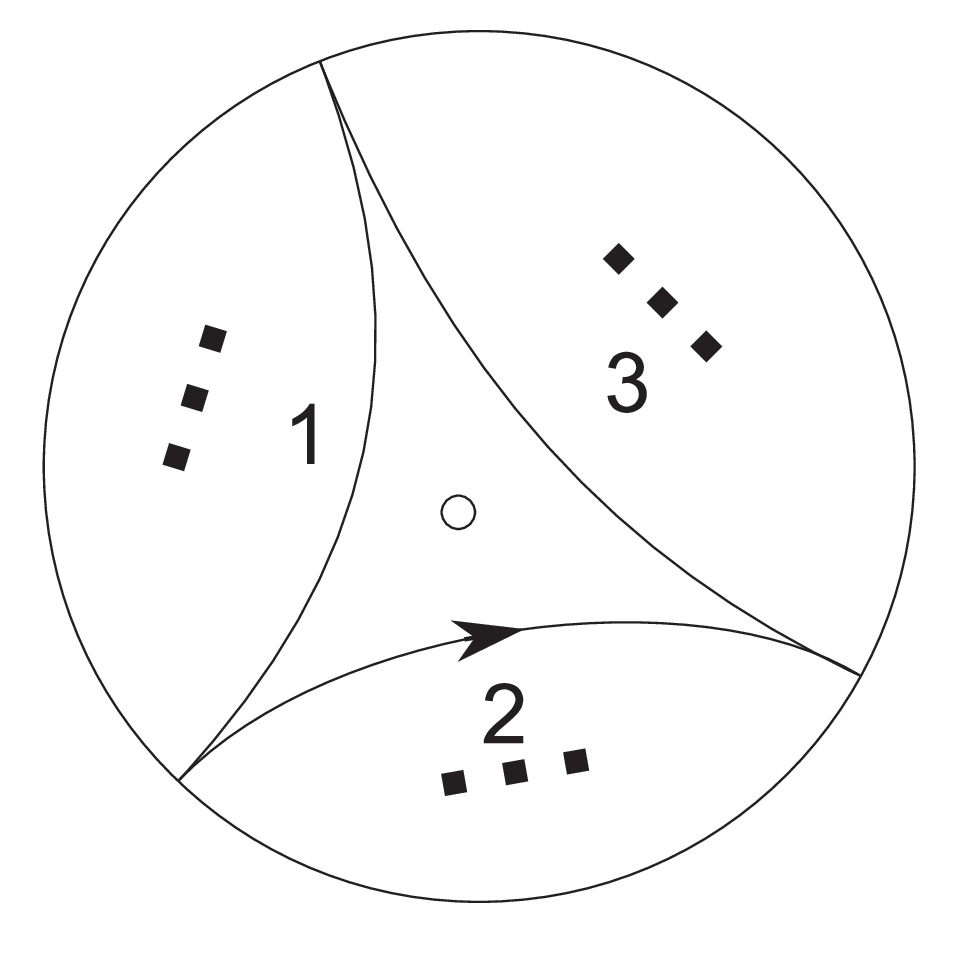}
\begin{pspicture}[showgrid=false,linewidth=0.5pt,unit=7.5mm](-0.5,-1.5)(0.1,2.0)
\rput[l](-0.7,0.6){\pcline[linewidth=0.7pt, arrowsize=2pt 4]{->}(0,0)(1.0;0)\Aput{$\beta^\sharp$}}
\end{pspicture}
\includegraphics[width=35mm]{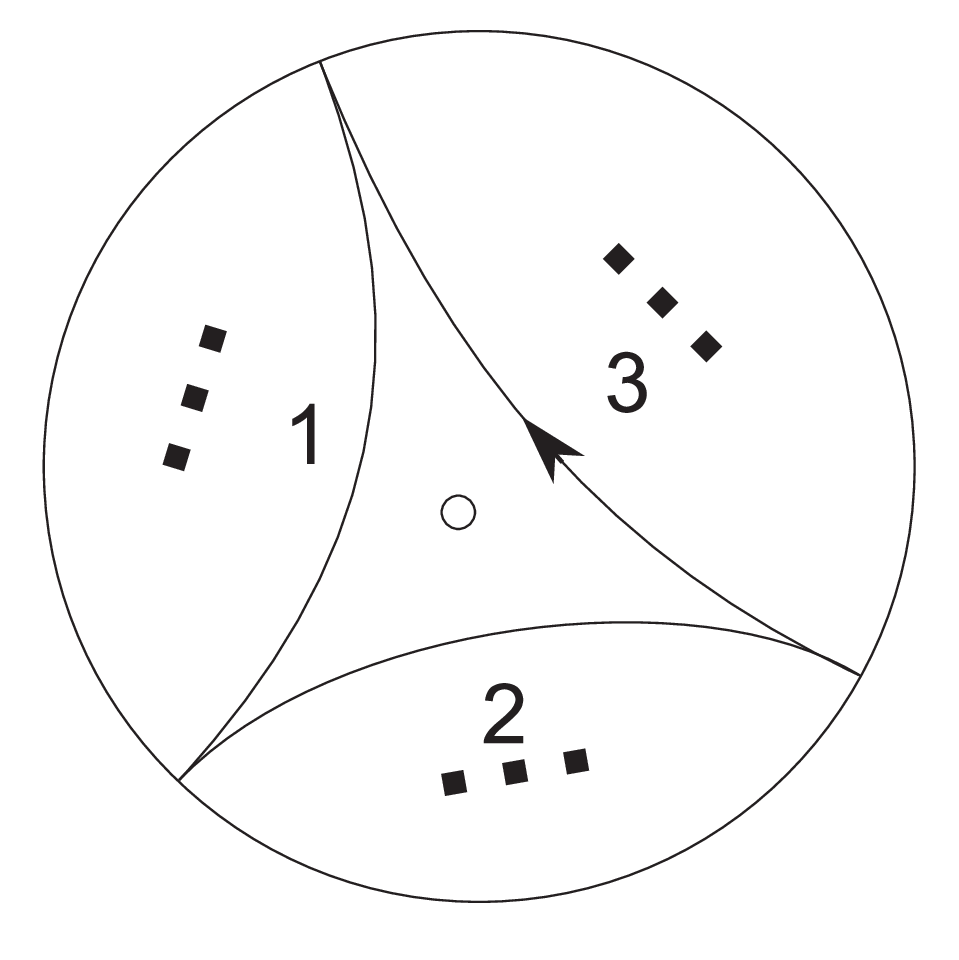}
\end{array}$
\\
\vspace{-4mm}

\begin{subfigure}[b]{0.48\textwidth}
\caption{The action of $\alpha^\sharp$ on $Pt^\sharp$}
\label{fig:action_of_alpha_sharp}
\end{subfigure}
\hfill
\begin{subfigure}[b]{0.5\textwidth}
\caption{The action of $\beta^\sharp$ on $Pt^\sharp$}
\label{fig:action_of_beta_sharp}
\end{subfigure}

\vspace{-4mm}

\caption{Some elementary morphisms of $Pt^\sharp$}
\label{fig:some_elementary_morphisms_of_Pt_sharp}
\end{figure}


\end{definition}

It is necessary to define one more kind of elementary morphisms, namely {\em braids} among the $*$-punctures or among the $\sharp$-punctures, induced by the following homeomorphisms classes.
\begin{definition}[\cite{FuKa2}: braiding]
\label{def:braiding}
Let $\diamond \in \{*,\sharp\}$. A {\em simple arc in $\mathbb{D}^\diamond$} is a homotopy class of non-self-intersecting unoriented paths in $\mathbb{D}^\diamond$. Let $e$ be a simple arc in $\mathbb{D}^\diamond$ connecting two distinct $\diamond$-punctures.  A {\em braiding $\sigma_e$ associated to $e$} is the homotopy class of a homeomorphism $\mathbb{D}^\diamond \to \mathbb{D}^\diamond$ which moves clockwise the two $\diamond$-punctures at the endpoints of $e$ in a thin neighborhood of $e$, interchanging their positions, and which is identity outside this neighborhood. Such a braiding is called {\em positive}, while $\sigma_e^{-1}$ {\em negative}.
\end{definition}

\begin{definition}[braids]
\label{def:braids}
A morphism of $Pt^*$ (resp. $Pt^\sharp$) induced by the braiding associated to a simple arc in $\mathbb{D}^\diamond$ connecting two distinct $*$-punctures (resp. $\sharp$-punctures) is called a {\em braid}.
\end{definition}

\begin{definition}[special braids]
\label{def:special_braids}
A morphism of $Pt^*$ induced by the braiding $\sigma_e$ associated to $e$, the unique simple arc in $\mathbb{D}^*$ connecting the $*$-punctures of the d.o.e. and the edge on the immediate right to the d.o.e. which does not traverse any edge of the $*$-tessellation which the morphism is being applied to, is labeled by $\sigma^*$; see Fig.\ref{fig:action_of_sigma_star}.

A morphism of $Pt^\sharp$ induced by the braiding $\sigma_{e'}$ associated to $e'$, the unique simple arc in $\mathbb{D}^\sharp$ connecting the $\sharp$-punctures of the two ideal triangles having the d.o.e. as one of their sides and traversing the d.o.e. exactly once while not traversing any other edge of the $\sharp$-tessellation which the morphism is being applied to, is labeled by $\sigma^\sharp$; see Fig.\ref{fig:action_of_sigma_sharp}.

\vspace{-4mm}
\begin{figure}[htbp!]
$\begin{array}{ll}
\hspace{-5mm}
\includegraphics[width=35mm]{t9m.eps}
\begin{pspicture}[showgrid=false,linewidth=0.5pt,unit=7.5mm](-0.5,-1.5)(0.1,2.0)
\rput[l](-0.7,0.6){\pcline[linewidth=0.7pt, arrowsize=2pt 4]{->}(0,0)(1.0;0)\Aput{$\sigma^*$}}
\end{pspicture}
\includegraphics[width=35mm]{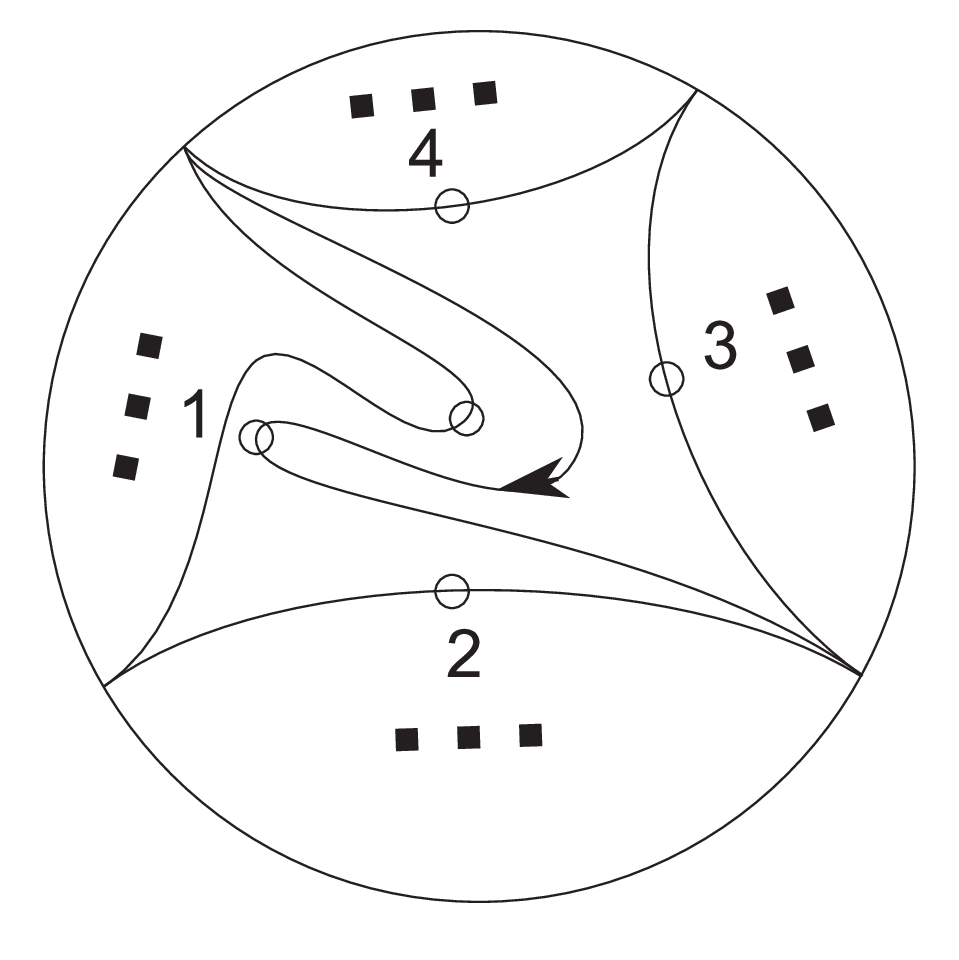}
&
%
\centering
\includegraphics[width=35mm]{t9.eps}
\begin{pspicture}[showgrid=false,linewidth=0.5pt,unit=7.5mm](-0.5,-1.5)(0.1,2.0)
\rput[l](-0.7,0.6){\pcline[linewidth=0.7pt, arrowsize=2pt 4]{->}(0,0)(1.0;0)\Aput{$\sigma^\sharp$}}
\end{pspicture}
\includegraphics[width=35mm]{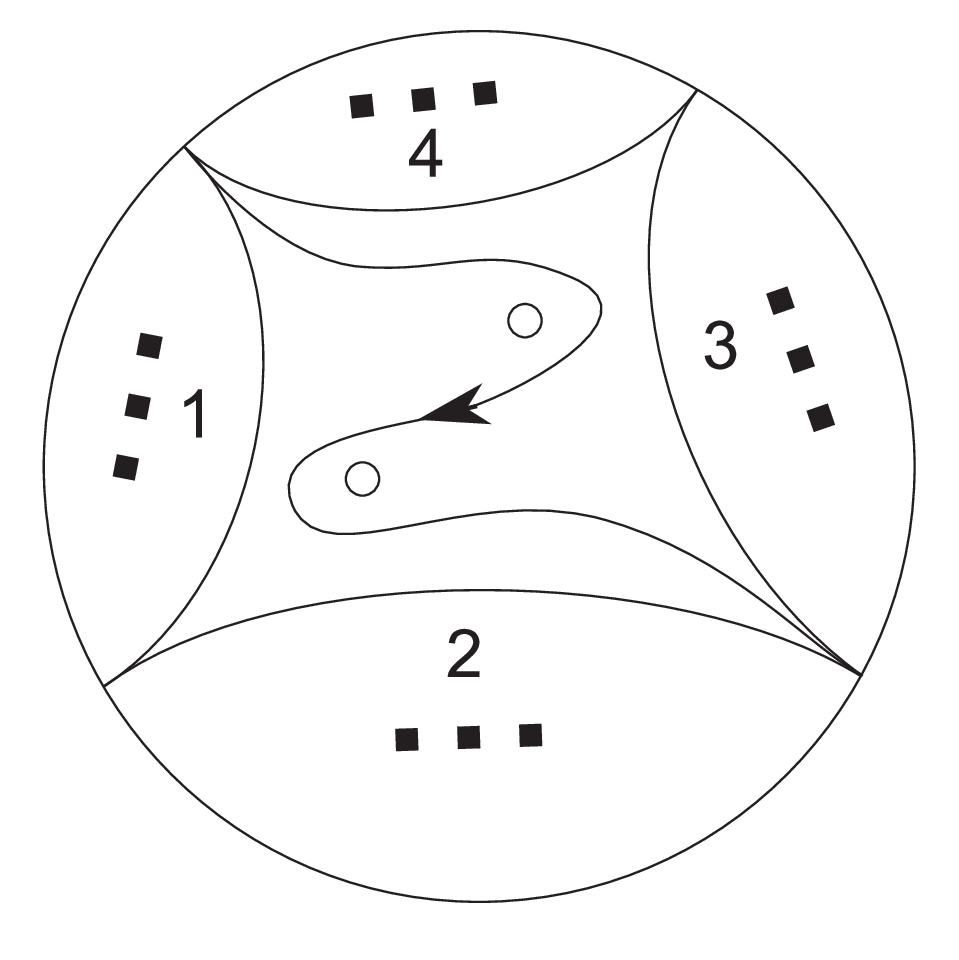}
\end{array}$
\\
\vspace{-4mm}

\begin{subfigure}[b]{0.48\textwidth}
\caption{The action of $\sigma^*$ on $Pt^*$}
\label{fig:action_of_sigma_star}
\end{subfigure}
\hfill
\begin{subfigure}[b]{0.5\textwidth}
\caption{The action of $\sigma^\sharp$ on $Pt^\sharp$}
\label{fig:action_of_sigma_sharp}
\end{subfigure}

\vspace{-4mm}

\caption{Special braids}
\label{fig:special_braids}
\end{figure}

\end{definition} 

\vspace{-8mm}

\begin{definition}
The morphisms of $Pt^*$ and $Pt^\sharp$ appearing in Def.\ref{def:braided_alpha_and_beta} and Def.\ref{def:special_braids} are called {\em elementary} morphisms of $Pt^*$ and $Pt^\sharp$ respectively.
\end{definition}

\begin{theorem}[\cite{FuKa2}]
Any morphism of $Pt^*$ is a composition of finite number of elementary morphisms labeled by $\alpha^*$, $\beta^*$. Any morphism of $Pt^\sharp$ is a composition of finite number of elementary morphisms labeled by $\alpha^\sharp$, $\beta^\sharp$, $\sigma^\sharp$.
\end{theorem}

The elementary morphisms satisfy some algebraic relations, such as $(\beta^*)^3={\rm id}$ and $(\beta^\sharp)^3={\rm id}$. We can then define groups presented by generators and relations, where the generators are associated to elementary morphisms and the relations are those satisfied by elementary morphisms. For $Pt^*$, we include $\sigma^*$ as one of the generators although not necessary, to make the presentation nicer.
\begin{definition}[\cite{FuKa2}]
Let $\diamond \in \{*,\sharp\}$. Let $T^\diamond$ be the group presented with generators $\alpha^\diamond$, $\beta^\diamond$, $\sigma^\diamond$ and the relations coming from the ones satisfied by the elementary morphisms of $Pt^\diamond$. These groups $T^*$ and $T^\sharp$ are called {\em braided Ptolemy-Thompson groups}.
\end{definition}
Obtaining and proving complete presentations of $T^*$ and $T^\sharp$ is quite a difficult job, and it is done in \cite{FuKa2}. In the present paper, we will need to use only some relations in $T^\sharp$, which is not difficult to show.
\begin{proposition}
\label{prop:some_relations_of_Pt_sharp}
The elementary morphisms of $Pt^\sharp$ satisfy
\begin{align*}
& (\alpha^\sharp)^4 = (\sigma^\sharp)^2, \qquad
(\beta^\sharp)^3 = {\rm id}, \qquad
(\beta^\sharp \alpha^\sharp)^5 = \sigma^\sharp \beta^\sharp \sigma^\sharp (\beta^\sharp)^{-1} \sigma^\sharp,  \\
& [\beta^\sharp \alpha^\sharp \beta^\sharp, \, (\alpha^\sharp)^2 (\sigma^\sharp)^{-1} \beta^\sharp \alpha^\sharp \beta^\sharp (\alpha^\sharp)^2 (\sigma^\sharp)^{-1}]=1, \\
& [\beta^\sharp \alpha^\sharp \beta^\sharp, \, (\alpha^\sharp)^2 (\sigma^\sharp)^{-1} (\beta^\sharp)^2 (\alpha^\sharp)^2 (\sigma^\sharp)^{-1} \beta^\sharp \alpha^\sharp \beta^\sharp (\alpha^\sharp)^2 (\sigma^\sharp)^{-1} \beta^\sharp (\alpha^\sharp)^2 (\sigma^\sharp)^{-1}]=1.
\end{align*}
\end{proposition}
The above relations are taken from \cite{FuKa2}, and can immediately be checked by hand using pictures, as these relations act only on a finite number of ideal triangles in an interesting way. In fact, for our purposes, we do not even have to take from \cite{FuKa2} the precise relations above, and what we actually need can be easily obtained from scratch. We will come back to this point at the end of the next subsection \S\ref{subsec:relative_abelianizations}.

\subsection{The relative abelianizations $T^*_{\rm ab}$, $T^\sharp_{\rm ab}$ of $T^*$, $T^\sharp$}
\label{subsec:relative_abelianizations}

By forgetting the $*$-punctures and $\sharp$-punctures, we obtain natural surjective group homomorphisms
$$
T^\diamond \to T \quad : \quad \alpha^\diamond \mapsto \alpha, \quad \beta^\diamond \mapsto \beta, \quad \sigma^\diamond \mapsto 1, \qquad \mbox{for} \quad \diamond \in \{*,\sharp\}.
$$
It is not hard to see that the kernel of this map is the following subgroup:
\begin{definition}
Let $\diamond \in \{*,\sharp\}$. Let $B(\mathbb{D}^\diamond)$ be the subgroup of $T^\diamond$ generated by the elements representing braids of $Pt^\diamond$ (Def.\ref{def:braids}).
\end{definition}
Thus we get a short exact sequence
\begin{align*}
\qquad \qquad \xymatrix{
1 \ar[r] & B(\mathbb{D}^\diamond) \ar[r] & T^\diamond \ar[r] & T \ar[r] & 1,
} \qquad \mbox{for}\quad \diamond\in\{*,\sharp\}.
\end{align*}
As observed in \cite{FuKa2}, the group $B(\mathbb{D}^\diamond)$ can be viewed as the inductive limit of the group $B(\mathbb{D}^\diamond_n)$ generated by elements of $T^\diamond$ induced by braids in a finite subsurface $\mathbb{D}^\diamond_n$ consisting of $n$ $\diamond$-punctured ideal triangles of any fixed $\diamond$-punctured tessellation of $\mathbb{D}^\diamond$, where $\mathbb{D}^\diamond_n$ are chosen such that $\mathbb{D}^\diamond_n \subset \mathbb{D}^\diamond_{n+1}$ and $\mathbb{D}^\diamond = \bigcup_n \mathbb{D}^\diamond_n$. It is pointed out to the author by Louis Funar that $B(\mathbb{D}^\diamond_n) \cong B_n$, where $B_n$ is the usual braid group of Artin on $n$ strands; see \cite{S}, in which the braid groups for graphs were considered for the first time. 
\begin{definition}
\label{def:B_infty}
The inductive limit of the Artin braid group $B_n$ on $n$ strands, with respect to the inclusion $\sigma_i \in B_n \mapsto \sigma_i \in B_{n+1}$ where $\sigma_i$ is the braid generator for $i$-th and $i+1$-th strands, is called the {\em stable braid group}, or the {\em infinite braid group}, and is denoted by $B_\infty$.
\end{definition}
Then we have
$$
B(\mathbb{D}^\diamond) \cong B_\infty,
$$
and therefore we obtain the short exact sequence in \eqref{eq:ses1}:
\begin{align*}
\qquad \qquad \xymatrix{
1 \ar[r] & B_\infty \ar[r] & T^\diamond \ar[r] & T \ar[r] & 1,
} \qquad \mbox{for}\quad \diamond\in\{*,\sharp\}.
\end{align*}
It is known that the abelianization of the kernel $B_\infty$ of this short exact sequence is isomorphic to $\mathbb{Z}$, which we can expect from the fact that every generator of $B_\infty$ is conjugate to each other, which we can see from the standard braid relations.
\begin{proposition}[see e.g. \cite{FuS}]
The abelianization of the group $B_\infty$ is $H_1(B_\infty) = \mathbb{Z}$.
\end{proposition}
Quotienting $T^\diamond$ by the commutator subgroup $[B_\infty,B_\infty]$ of the kernel of \eqref{eq:ses1} is called the {\em relative abelianization} of the short exact sequence \eqref{eq:ses1}, yielding another short exact sequence:
\begin{align}
\nonumber
\left\{ \begin{array}{l}
\xymatrix{
1 \ar[r] & B_\infty \ar[r] \ar[d] & T^\diamond \ar[r] \ar[d] & T \ar[r] \ar[d]^{\rm id} & 1 \\
1 \ar[r] & B_\infty/[B_\infty,B_\infty] \ar[r] & T^\diamond/[B_\infty,B_\infty] \ar[r] & T \ar[r] & 1.
} \end{array} \right.
\end{align}

\begin{definition}
\label{def:relative_abelianization_of_T_diamond}
For $\diamond \in \{*,\sharp\}$, we denote
\begin{align}
\label{eq:T_diamond_ab}
T^\diamond_{\rm ab} := T^\diamond/[B_\infty,B_\infty],
\end{align}
and it is called the {\em relative abelianization of $T^\diamond$}. We denote by
\begin{align}
\nonumber
\til{\alpha}^\diamond, \til{\beta}^\diamond \in T^\diamond_{\rm ab}
\end{align}
the images of $\alpha^\diamond$, $\beta^\diamond \in T^\diamond$ under the projection $T^\diamond \to T^\diamond/[B_\infty,B_\infty] =T^\diamond_{\rm ab}$.
\end{definition}
So we obtained the short exact sequence \eqref{eq:ses2}
\begin{align*}
\xymatrix{
1 \ar[r] & \mathbb{Z} \ar[r] & T^\diamond_{\rm ab} \ar[r] & T \ar[r] & 1,
}
\end{align*}
that is, $T^\diamond_{\rm ab}$ is an extension of $T$ by $\mathbb{Z}$. We can easily prove that this is in fact a central extension.

\begin{proposition}
\label{prop:T_diamond_ab_is_central_extension}
For $\diamond \in \{*,\sharp\}$, $T^\diamond_{\rm ab}$ is a central extension of $T$ by $\mathbb{Z}$.
\end{proposition}

\begin{proof}
Recall that $\sigma^\diamond$ is induced by a positive braiding associated to some arc, so we can say that it is a positive braid. We observe that the element of $B_\infty$ induced by the positive braiding $\sigma_e$ for any arc $e$ is mapped to the same element by the map $B_\infty \to B_\infty/[B_\infty,B_\infty] \cong \mathbb{Z}$, say, always to $1$, or always to $-1$. Now, we can see that each of $\alpha^\diamond \sigma^\diamond (\alpha^\diamond)^{-1}$ and $\beta^\diamond \sigma^\diamond (\beta^\diamond)^{-1}$ is induced by a positive braiding associated to some arc, therefore maps to the same element $z\in B_\infty/[B_\infty,B_\infty]$ by the map $B_\infty \to B_\infty/[B_\infty,B_\infty]$ as $\sigma^\diamond$ does. Hence, by applying the map $T^\diamond \to T^\diamond/[B_\infty,B_\infty]$ we get $\til{\alpha}^\diamond z (\til{\alpha}^\diamond)^{-1} =z$ and $\til{\beta}^\diamond z (\til{\beta}^\diamond)^{-1} =z$, proving the desired statement, as $z$ is a generator of the kernel $\mathbb{Z} \cong B_\infty/[B_\infty,B_\infty]$ of the short exact sequence \eqref{eq:ses2}.
\end{proof}

From the full presentations of $T^*$ and $T^\sharp$ obtained in \cite{FuKa2}, we can give presentations of their relative abelianizations $T^*_{\rm ab}$ and $T^\sharp_{\rm ab}$ in the style of Thm.\ref{thm:FS_classification}:
\begin{proposition}[\cite{FuKa2}: presentation of $T^*_{\rm ab}$]
\label{prop:FuKa_T_star_ab_presentation}
The group $T^*_{\rm ab}$ \eqref{eq:T_diamond_ab} has a presentation with generators $\til{\alpha}^*, \til{\beta}^*$, $z$ and the relations
\begin{align*}
{\renewcommand{\arraystretch}{1.2}
\begin{array}{l}
(\til{\beta}^*\til{\alpha}^*)^5 = z, \qquad (\til{\alpha}^*)^4 = 1, \qquad (\til{\beta}^*)^3= 1, \qquad \left[\til{\alpha}^*,z\right]=\left[\til{\beta}^*,z\right]=1, \\
\left[\til{\beta}^*\til{\alpha}^*\til{\beta}^*, \, (\til{\alpha}^*)^2 \til{\beta}^*\til{\alpha}^*\til{\beta}^*(\til{\alpha}^*)^2\right]= \left[\til{\beta}^*\til{\alpha}^*\til{\beta}^*, \, (\til{\alpha}^*)^2\til{\beta}^*(\til{\alpha}^*)^2\til{\beta}^*\til{\alpha}^*\til{\beta}^*(\til{\alpha}^*)^2(\til{\beta}^*)^2(\til{\alpha}^*)^2\right]= 1,
\end{array} }
\end{align*}
Hence 
$T^*_{\rm ab} \cong T_{1,0,0,0}$, 
where $T_{n,p,q,r}$ is as in Thm.\ref{thm:FS_classification}.
\end{proposition}

\begin{proposition}[\cite{FuKa2}: presentation of $T^\sharp_{\rm ab}$]
\label{prop:FuKa_T_sharp_ab_presentation}
The group $T^\sharp_{\rm ab}$ \eqref{eq:T_diamond_ab} has a presentation with generators $\til{\alpha}^\sharp, \til{\beta}^\sharp$, $z$ and the relations
\begin{align*}
{\renewcommand{\arraystretch}{1.2}
\begin{array}{l}
(\til{\beta}^\sharp\til{\alpha}^\sharp)^5 = z^3, \qquad (\til{\alpha}^\sharp)^4 = z^2, \qquad (\til{\beta}^\sharp)^3= 1, \qquad \left[\til{\alpha}^\sharp,z\right]=\left[\til{\beta}^\sharp,z\right]=1, \\
\left[\til{\beta}^\sharp\til{\alpha}^\sharp\til{\beta}^\sharp, \, (\til{\alpha}^\sharp)^2 \til{\beta}^\sharp\til{\alpha}^\sharp\til{\beta}^\sharp(\til{\alpha}^\sharp)^2\right]= \left[\til{\beta}^\sharp\til{\alpha}^\sharp\til{\beta}^\sharp, \, (\til{\alpha}^\sharp)^2\til{\beta}^\sharp(\til{\alpha}^\sharp)^2\til{\beta}^\sharp\til{\alpha}^\sharp\til{\beta}^\sharp(\til{\alpha}^\sharp)^2(\til{\beta}^\sharp)^2(\til{\alpha}^\sharp)^2\right]= 1,
\end{array} }
\end{align*}
Hence 
$T^\sharp_{\rm ab} \cong T_{3,2,0,0}$, 
where $T_{n,p,q,r}$ is as in Thm.\ref{thm:FS_classification}.
\end{proposition}

Therefore, from the presentation of $T^*_{\rm ab}$ (Prop.\ref{prop:FuKa_T_star_ab_presentation}) and that of $\wh{T}^{\rm CF}$ (Thm.\ref{thm:FS}), we can deduce that $T^*_{\rm ab} \cong \wh{T}^{\rm CF}$, namely Prop.\ref{prop:FS_T_star_ab}. From the presentation of $T^\sharp_{\rm ab}$ (Prop.\ref{prop:FuKa_T_sharp_ab_presentation}) and that of $\wh{T}^{\rm Kash}$ (Thm.\ref{thm:main}), we can deduce that $T^\sharp_{\rm ab} \cong \wh{T}^{\rm Kash}$, namely Prop.\ref{prop:T_sharp_ab}. What is being done in the present section is to give an alternative proof of Thm.\ref{thm:main}, by proving $T^\sharp_{\rm ab} \cong \wh{T}^{\rm Kash}$ (Prop.\ref{prop:T_sharp_ab}) directly without knowing presentations of $T^\sharp_{\rm ab}$ or $\wh{T}^{\rm Kash}$. Then, what would remain to do is to get a presentation of $T^\sharp_{\rm ab}$, i.e. Prop.\ref{prop:FuKa_T_sharp_ab_presentation}. This way of proving Thm.\ref{thm:main} makes a crucial use of the central extension $T^\sharp_{\rm ab}$ of $T$ obtained by topological arguments, and therefore we call this a `topological' proof of Thm.\ref{thm:main}, our main theorem. We consider that $\wh{T}^{\rm Kash} \cong T^\sharp_{\rm ab}$ (Prop.\ref{prop:T_sharp_ab}) gives more insight on the nature of the central extension $\wh{T}^{\rm Kash}$ coming from the Kashaev quantization, than just a presentation of $\wh{T}^{\rm Kash}$ (Thm.\ref{thm:main}) does.

\vs

One way of proving Prop.\ref{prop:FuKa_T_sharp_ab_presentation} is to use Prop.\ref{prop:some_relations_of_Pt_sharp}, but we promised another way which we can come up with from scratch. Namely, take each $\alpha, \beta$-relation of $T$ in \eqref{eq:T_presentation_intro}, and replace $\alpha,\beta$ by $\alpha^\sharp, \beta^\sharp$, respectively. Apply such expression to any marked $\sharp$-punctured tessellation of $\mathbb{D}^\sharp$. 
Then we get a new marked $\sharp$-punctured tessellation differing from the original only by braids; the underlying marked tessellations are same, but the way how the ideal arcs go around $\sharp$-punctures are different. Apply braids to the obtained marked $\sharp$-punctured tessellation which `unravel' the picture eventually to the original one. The number of negative braids applied minus the number of positive braids applied tells you the power of the central element $z$ which we should put in the RHS of the lifted version of this $\alpha,\beta$-relation. We do not have to know the precise expressions of these braids in terms of elements of $T^\sharp$, but only need to know the number of positive ones and negative ones. An example of this procedure is shown in Fig.\ref{fig:topological_check}; we see that $(\beta^\sharp \alpha^\sharp)^5$ is resolved to identity by three negative braids, so the corresponding relation for $T^\sharp_{\rm ab}$ is $(\til{\beta}^\sharp \til{\alpha}^\sharp)^5 = z^3$.

\input{fig-topological_check.tex}

\vspace{-5mm}

\subsection{The $\sharp$-punctured Kashaev group $K^\sharp$}
\label{subsec:K_sharp}

Our strategy to prove $\wh{T}^{\rm Kash} \cong T^\sharp_{\rm ab}$ (Prop.\ref{prop:T_sharp_ab}) is to prove a corresponding statement in terms of the Kashaev group, namely using dotted $\sharp$-punctured tessellations of $\mathbb{D}^\sharp$. So we first have to define a braided version of the Kashaev group from the groupoid $Pt^\sharp_{\rm dot}$. As in \S\ref{subsec:T_and_K}, we view each morphism of $Pt^\sharp_{\rm dot}$ as a transformation of a dotted $\sharp$-punctured tessellation of $\mathbb{D}^\sharp$ into another, and define elementary morphisms of $Pt^\sharp_{\rm dot}$:
\begin{definition}[elementary morphisms of $Pt^\sharp_{\rm dot}$]
\label{def:elementary_morphisms_of_Pt_sharp_dot}
We label a morphism of $Pt^\sharp_{\rm dot}$ by $A^\sharp_{[j]}$ ($j\in \mathbb{Q}^\times$) if it transforms a dotted $\sharp$-punctured tessellation of $\mathbb{D}^\sharp$ as in Fig.\ref{fig:action_of_A_j_sharp}, i.e. moves the dot $\bullet$ of the triangle $j$, leaving all other parts indicated by triple dots in Fig.\ref{fig:action_of_A_j_sharp} intact.

\vs

We label a morphism of $Pt^\sharp_{\rm dot}$ by $T^\sharp_{[j][k]}$ ($j,k\in \mathbb{Q}^\times$, $j\neq k$) if it transforms a dotted $\sharp$-punctured tessellation of $\mathbb{D}^\sharp$ as in Fig.\ref{fig:action_of_T_jk_sharp}, leaving all other parts indicated by triple dots in Fig.\ref{fig:action_of_T_jk_sharp} intact. That is, it is described in the same way as $T_{[j][k]}$ of Def.\ref{def:elementary_moves_of_Pt_dot}, where we require that the $\sharp$-punctures of triangles $j,k$ must not be touched while rotating the d.o.e. to the other diagonal.

\vs

For a $\mathbb{Q}^\times$-permutation $\gamma$, a morphism of $Pt^\sharp_{\rm dot}$ labeled by $P^\sharp_\gamma$ is defined analogously to $P_\gamma$ in Def.\ref{def:elementary_moves_of_Pt_dot}, leaving the underlying $\sharp$-tessellation intact.

\vs

We label a morphism of $Pt^\sharp_{\rm dot}$ by $\sigma^\sharp_{[j][k]}$ ($j,k\in \mathbb{Q}^\times$, $j\neq k$) if it transforms a dotted $\sharp$-punctured tessellation of $\mathbb{D}^\sharp$ as in Fig.\ref{fig:action_of_sigma_jk_sharp}, leaving all other parts indicated by triple dots in Fig.\ref{fig:action_of_sigma_jk_sharp} intact. That is, it can be applied only to dotted $\sharp$-punctured tessellations on which $T^\sharp_{[j][k]}$ can be applied to, and the action is induced by the braiding (Def.\ref{def:braiding}) associated to the unique simple arc in $\mathbb{D}^\sharp$ connecting the $\sharp$-punctures of the triangles $j$ and $k$ which intersects with only one ideal arc of the initial dotted $\sharp$-punctured tessellation and only once.

\vs

These morphisms are called {\em elementary}.
\end{definition}

\begin{figure}[htbp!]
$\begin{array}{ll}
\hspace{-5mm}
\includegraphics[width=35mm]{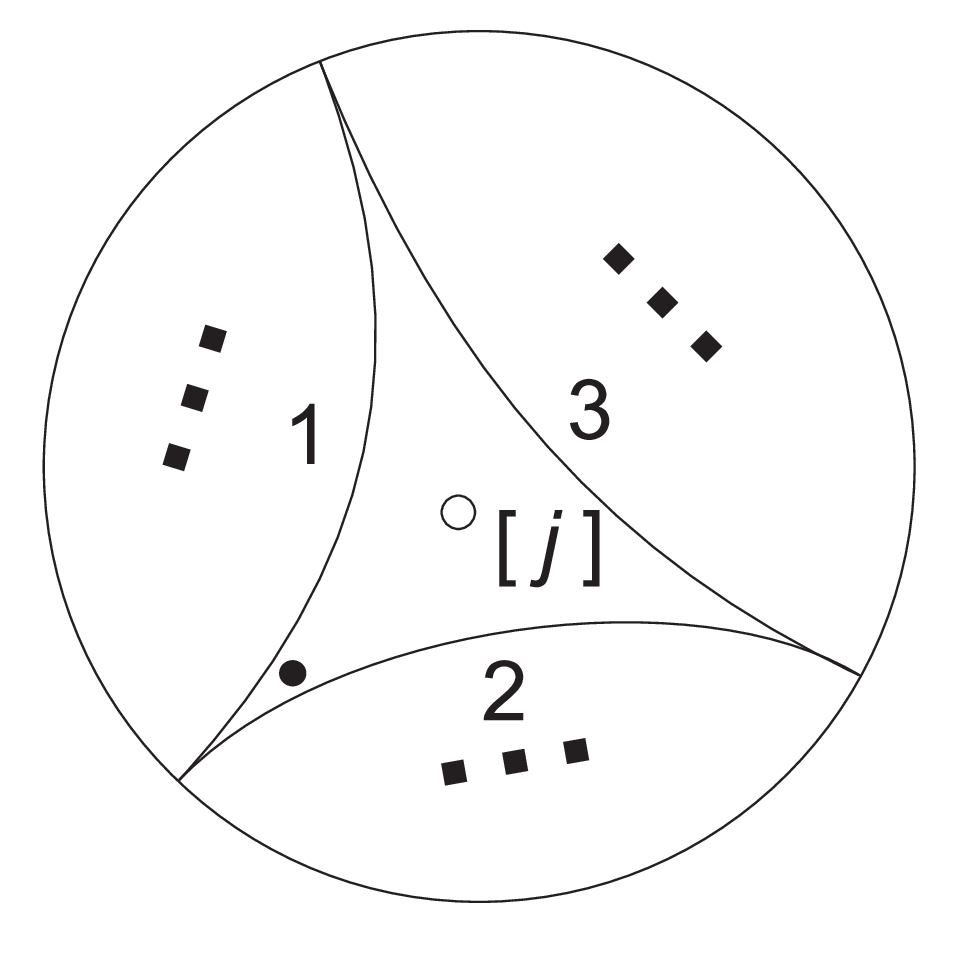}
\begin{pspicture}[showgrid=false,linewidth=0.5pt,unit=7.5mm](-0.5,-1.5)(0.1,2.0)
\rput[l](-0.7,0.6){\pcline[linewidth=0.7pt, arrowsize=2pt 4]{->}(0,0)(1.0;0)\Aput{$A_{[j]}^\sharp$}}
\end{pspicture}
\includegraphics[width=35mm]{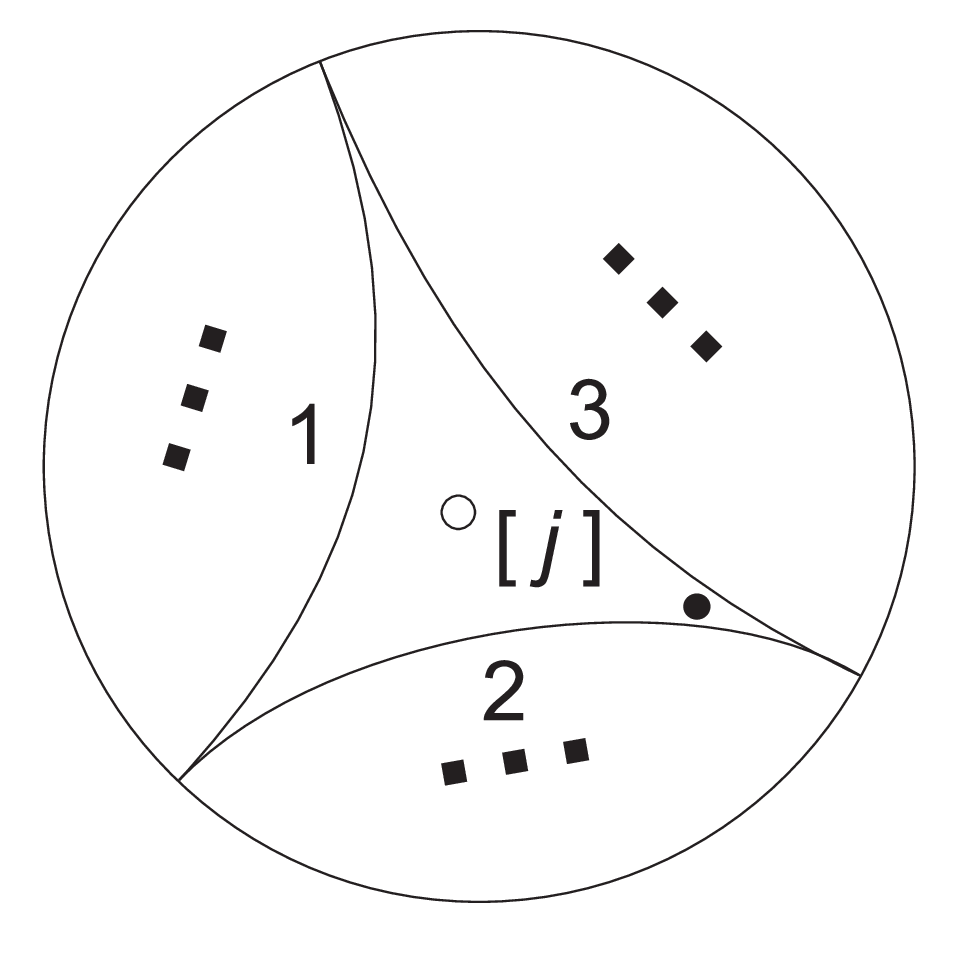}
&
\centering
\includegraphics[width=35mm]{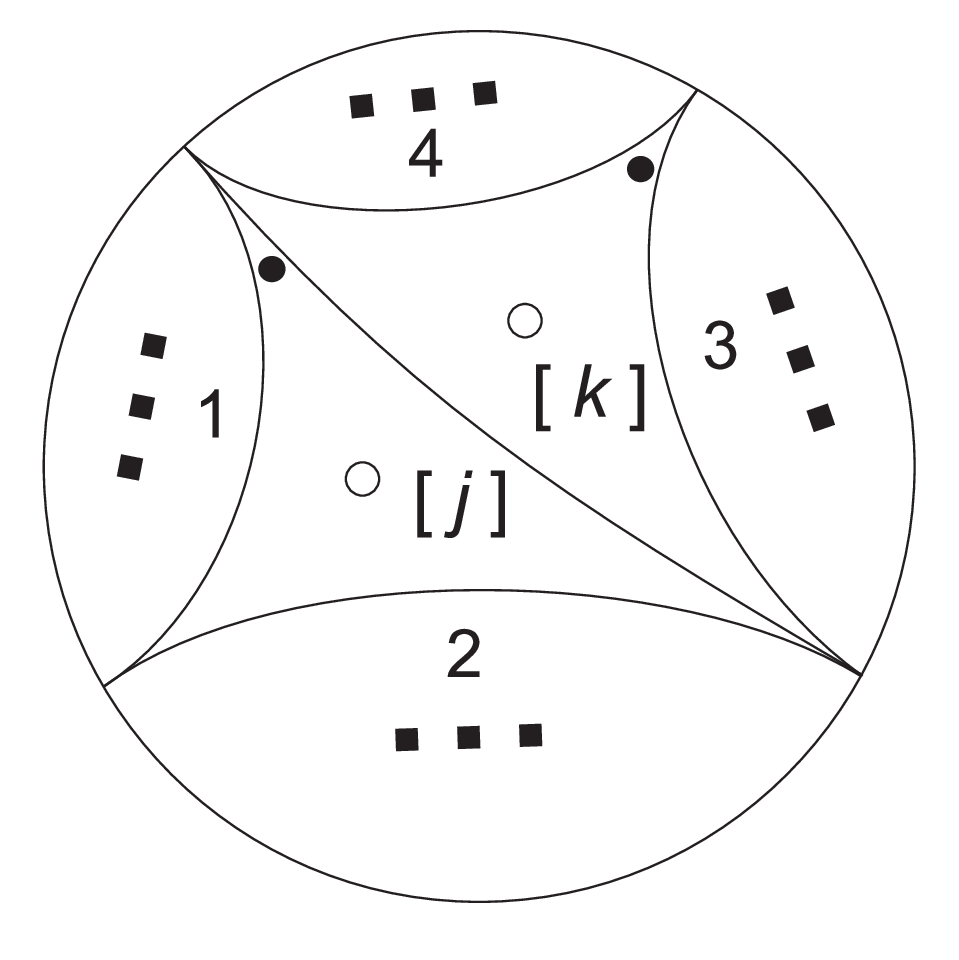}
\begin{pspicture}[showgrid=false,linewidth=0.5pt,unit=7.5mm](-0.5,-1.5)(0.1,2.0)
\rput[l](-0.7,0.6){\pcline[linewidth=0.7pt, arrowsize=2pt 4]{->}(0,0)(1.0;0)\Aput{$T_{[j][k]}^\sharp$}}
\end{pspicture}
\includegraphics[width=35mm]{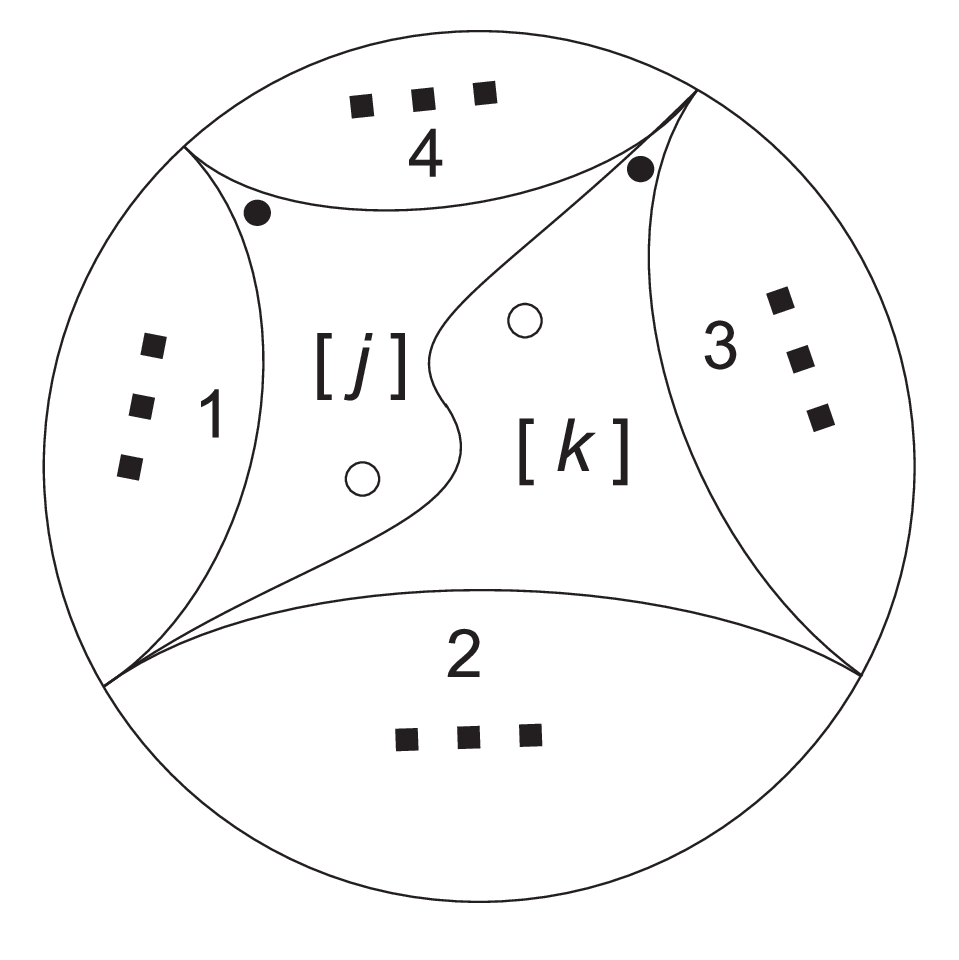}
\end{array}$
\\
\vspace{-2mm}

\begin{subfigure}[b]{0.48\textwidth}
\caption[The action of A-sharp on a sharp-punctured tessellation]{The action of $A_{[j]}^\sharp$ on $Pt^\sharp_{dot}$}
\label{fig:action_of_A_j_sharp}
\end{subfigure}
\hfill
\begin{subfigure}[b]{0.5\textwidth}
\caption{The action of $T^\sharp_{[j][k]}$ on $Pt^\sharp_{dot}$}
\label{fig:action_of_T_jk_sharp}
\end{subfigure}

\includegraphics[width=35mm]{t5.eps}
\begin{pspicture}[showgrid=false,linewidth=0.5pt,unit=7.5mm](-0.5,-1.5)(0.1,2.0)
\rput[l](-0.7,0.6){\pcline[linewidth=0.7pt, arrowsize=2pt 4]{->}(0,0)(1.0;0)\Aput{$\sigma_{[j][k]}^\sharp$}}
\end{pspicture}
\includegraphics[width=35mm]{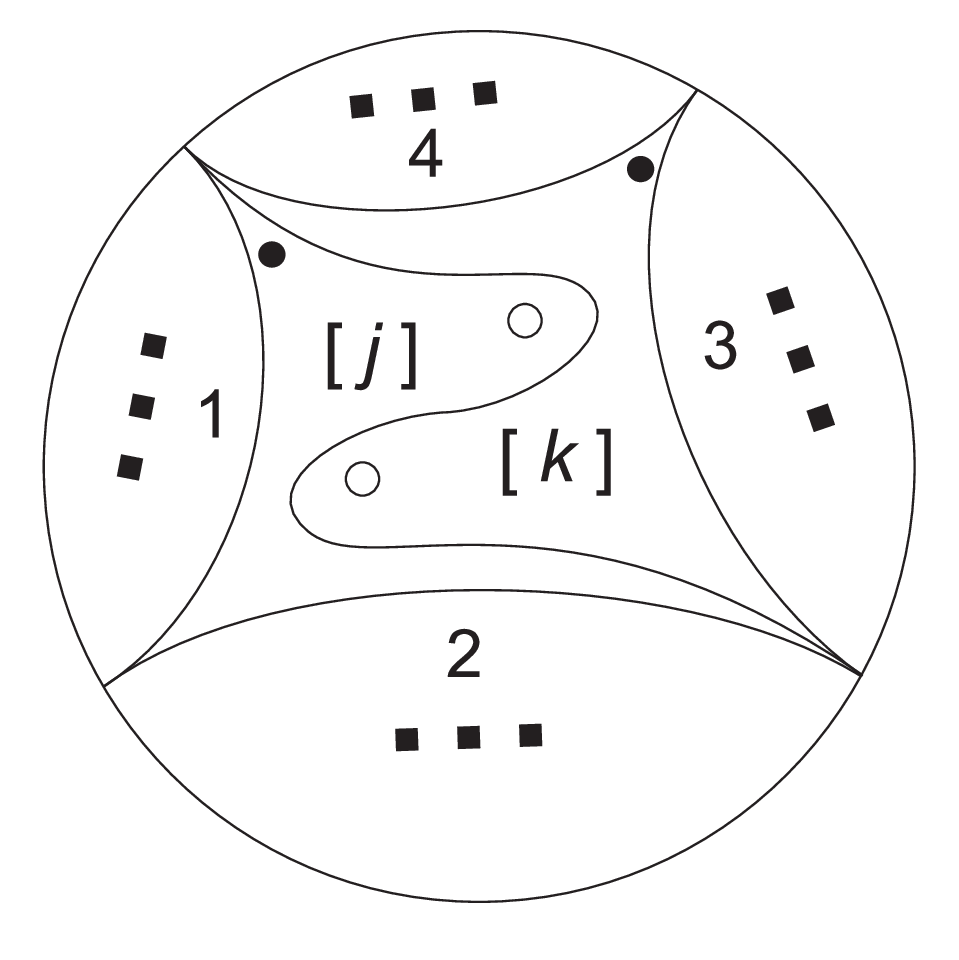}

\vspace{-2mm}

\begin{subfigure}[b]{0.48\textwidth}
\caption{The action of $\sigma_{[j][k]}^\sharp$ on $Pt^\sharp_{dot}$}
\label{fig:action_of_sigma_jk_sharp}
\end{subfigure}

\vspace{-4mm}
\caption{Some elementary morphsims of $Pt^\sharp_{dot}$}
\label{fig:some_elementary_morphisms_of_Pt_sharp_dot}
\end{figure}

\vspace{-3mm}

\begin{remark}
The definition of $\sigma^\sharp_{[j][k]}$ given in \cite{Ki} (denoted by $\sigma_{[j][k]}$ there) is not precise enough, and should be replaced by the above definition. Accordingly, some algebraic relations involving $\sigma^\sharp_{[j][k]}$ should be corrected, although this change doesn't affect the end result of \cite{Ki}.
\end{remark}
These elementary morphisms satisfy certain algebraic relations, some of which are:
\begin{proposition}
\label{prop:K_sharp_relations}
The above elementary moves satisfy the following relations:
\begin{align}
\label{eq:K_sharp_relations_major}
\left\{ {\renewcommand{\arraystretch}{1.4} \begin{array}{ll}
(A_{[j]}^\sharp)^3 = {\rm id},  &
T_{[k][\ell]}^\sharp T_{[j][k]}^\sharp = T_{[j][k]}^\sharp T_{[j][\ell]}^\sharp T_{[k][\ell]}^\sharp, \\
A_{[j]}^\sharp T_{[j][k]}^\sharp A_{[k]}^\sharp = A_{[k]}^\sharp T_{[k][j]}^\sharp A_{[j]}^\sharp, &
T_{[j][k]}^\sharp A_{[j]}^\sharp T_{[k][j]}^\sharp = ( A^{\sharp \, -1}_{[k]}\sigma^{\sharp \, -1}_{[k][j]} A^\sharp_{[k]} ) A_{[j]}^\sharp A_{[k]}^\sharp P_{(jk)}^\sharp,
\end{array} } \right.
\end{align}
where $j,k,\ell \in \mathbb{Q}^\times$ are mutually distinct. The `trivial relations' as in Thm.\ref{thm:algebraic_relations_of_elementary_moves} with $A_{[j]}, T_{[j][k]}, P_\gamma$ replaced by $A^\sharp_{[j]}, T^\sharp_{[j][k]}, P^\sharp_\gamma$ are also satisfied. The usual braid relation
$$
\sigma^\sharp_{[j][k]} \sigma^\sharp_{[k][\ell]} \sigma^\sharp_{[j][k]} 
= \sigma^\sharp_{[k][\ell]} \sigma^\sharp_{[j][k]} \sigma^\sharp_{[k][\ell]}
$$
holds for mutually distinct $j,k,\ell\in \mathbb{Q}^\times$.
\end{proposition}

\begin{proof}
The proof of $(A_{[j]}^\sharp)^3 = {\rm id}$ and the trivial relations can be easily seen. The proof of the other three relations in \eqref{eq:K_sharp_relations_major} is manifest from Figures \ref{fig:pentagon_relation_for_sharp}, \ref{fig:first_relation_for_sharp} and \ref{fig:second_relation_for_sharp}. The braid relation can be checked by pictures in a similar way, but we omit it here.
\end{proof}


\vspace{-5mm}

\input{fig-pentagon_relation_for_sharp.tex}

\input{fig-first_relation_for_sharp.tex}

\input{fig-second_relation_for_sharp.tex}

\begin{definition}
\label{def:K_sharp}
Let $K^\sharp$ be the group presented by the generators $A_{[j]}^\sharp$, $T_{[j][k]}^\sharp$, $P_\gamma^\sharp$, $\sigma^\sharp_{[j][k]}$, for $j,k \in \mathbb{Q}^\times$ ($j \neq k$) and $\mathbb{Q}^\times$-permutations $\gamma$, and relations satisfied by the corresponding elementary moves of $Pt^\sharp_{\rm dot}$.
\end{definition}

The relations shown in Prop.\ref{prop:K_sharp_relations} are part of the presentation of $K^\sharp$, but may not give its full presentation. For our purposes, it is enough to have the ones obtained in Prop.\ref{prop:K_sharp_relations}. It is easy to see from the relations \eqref{eq:K_sharp_relations_major} that $A_{[j]}^\sharp$, $T_{[j][k]}^\sharp$, $P_\gamma^\sharp$ generate $K^\sharp$.
By forgetting the punctures we get the group homomorphism 
\begin{align}
\nonumber
K^\sharp \to K : A_{[j]}^\sharp \mapsto A_{[j]}, \quad T_{[j][k]}^\sharp \to T_{[j][k]}, \quad P_\gamma^\sharp \to P_\gamma,
\end{align}
whose kernel is generated by braids, and we can show that this kernel is isomorphic to $B_\infty$ (Def.\ref{def:B_infty}). Thus we get a short exact sequence
\begin{align*}
\xymatrix{
1 \ar[r] & B_\infty \ar[r] & K^\sharp \ar[r] & K \ar[r] & 1.
}
\end{align*}
Quotienting by $[B_\infty,B_\infty]$ induces the commutative diagram for the relative abelianization
\begin{align}
\nonumber
\left\{ \begin{array}{l} \xymatrix{
1 \ar[r] & B_\infty \ar[r] \ar[d] & K^\sharp \ar[r] \ar[d] & K \ar[r] \ar[d]^{\rm id} & 1, \\
1 \ar[r] & \mathbb{Z} \ar[r] & K^\sharp_{\rm ab} \ar[r] & K \ar[r] & 1,
}
\end{array} \right.
\end{align}
where
$$
K^\sharp_{\rm ab} \cong K^\sharp/[B_\infty, B_\infty]
$$
is the relative abelianization of $K^\sharp$. A similar proof as in Prop.\ref{prop:T_diamond_ab_is_central_extension} tells us that $K^\sharp_{\rm ab}$ is a central extension of $K$ by $\mathbb{Z}$. We denote the images of the projection $K^\sharp \to K^\sharp/[B_\infty,B_\infty] = K^\sharp_{\rm ab}$ by
\begin{align}
\label{eq:K_sharp_to_K_sharp_ab}
K^\sharp \to K^\sharp_{\rm ab} \quad : \quad A_{[j]}^\sharp \mapsto \til{A}_{[j]}^\sharp, \quad T_{[j][k]}^\sharp \mapsto \til{T}_{[j][k]}^\sharp, \quad P^\sharp_\gamma \mapsto \til{P}_\gamma^\sharp, \quad \sigma^\sharp_{[j][k]} \mapsto z,
\end{align}
where $z$ is the generator of the central kernel of $K^\sharp_{\rm ab}$ isomorphic to $\mathbb{Z}$. It is now easy to obtain the following presentation of $K^\sharp_{\rm ab}$, from Prop.\ref{prop:K_sharp_relations}, Def.\ref{def:K_sharp}, eq.\eqref{eq:K_sharp_to_K_sharp_ab}, and the relations defining $K$ (Thm.\ref{thm:algebraic_relations_of_elementary_moves} and Def.\ref{def:Kashaev_group}):

\begin{proposition}
\label{prop:K_sharp_ab_presentation}
The group $K^\sharp_{\rm ab}$ can be presented with the generators $\til{A}_{[j]}^\sharp$, $\til{T}_{[j][k]}^\sharp$, $\til{P}_\gamma^\sharp$, $z$, for $j,k \in \mathbb{Q}^\times$ with $j\neq k$ and $\mathbb{Q}^\times$-permutations $\gamma$, with the following relations:  
\begin{align*}
\left\{ {\renewcommand{\arraystretch}{1.4} \begin{array}{ll}
(\til{A}_{[j]}^\sharp)^3 = {\rm id},
& \til{T}_{[k][\ell]}^\sharp \til{T}_{[j][k]}^\sharp = \til{T}_{[j][k]}^\sharp \til{T}_{[j][\ell]}^\sharp \til{T}_{[k][\ell]}^\sharp, \\
\til{A}_{[j]}^\sharp \til{T}_{[j][k]}^\sharp \til{A}_{[k]}^\sharp = \til{A}_{[k]}^\sharp \til{T}_{[k][j]}^\sharp \til{A}_{[j]}^\sharp, 
& \til{T}_{[j][k]}^\sharp \til{A}_{[j]}^\sharp \til{T}_{[k][j]}^\sharp = z^{-1} \til{A}_{[j]}^\sharp \til{A}_{[k]}^\sharp \til{P}_{(jk)}^\sharp,
\end{array} } \right.
\end{align*}
where $j,k,\ell \in \mathbb{Q}^\times$ are mutually distinct, the `trivial relations' as in Thm.\ref{thm:algebraic_relations_of_elementary_moves} with $A_{[j]}, T_{[j][k]}, P_\gamma$ replaced by $\til{A}^\sharp_{[j]}, \til{T}^\sharp_{[j][k]}, \til{P}^\sharp_\gamma$, and the commuting relations
\begin{align*}
[z, \til{A}_{[j]}^\sharp] = [z, \til{T}_{[j][k]}^\sharp] = [z,\til{P}_\gamma^\sharp]=1. \qed
\end{align*}
\end{proposition}

\subsection{The natural maps ${\bf F}^\sharp : T^\sharp \to K^\sharp$ and  ${\bf F}^\sharp_{\rm ab} : T^\sharp_{\rm ab} \to K^\sharp_{\rm ab}$}
\label{subsec:bolf_F_sharp}

We will mimic the construction in \S\ref{subsec:mcal_F} and \S\ref{subsec:bold_F} of the natural map ${\bf F} : T \to K$ 
in this $\sharp$-punctured setting. Recall that what made the construction  of ${\bf F}$ natural and unique is the group $M$, the asymptotically rigid mapping class group of $\mathbb{D}$ (Def.\ref{def:M}), which is the group of all mapping classes of $\mathbb{D}$ preserving the set of objects of $Pt$. Analogously, we start by defining a certain `asymptotically rigid' mapping class group $M^\sharp$ of the $\sharp$-punctured disc $\mathbb{D}^\sharp$, which is the group of all mapping classes of $\mathbb{D}^\sharp$ preserving the set of objects of $Pt^\sharp$. Recall that homotopies of $\mathbb{D}^\sharp$ are assumed to pointwise fix every point of $S^1 = \partial \mathbb{D}$ {\em and} every $\sharp$-puncture at all times.

\begin{definition}
\label{def:asymptotically_sharp-rigid_homeomorphisms}
Regard $\mathbb{D}^\sharp$ as the open unit disc minus the $\sharp$-punctures (Def.\ref{def:punctured_discs}). A homotopy class of homeomorphisms $\varphi : \mathbb{D}^\sharp \to \mathbb{D}^\sharp$ is called {\em asymptotically $\sharp$-rigid} if $\varphi$ can be continuously extended to the boundary $S^1$, restricts to $S^1$ as an element of ${\rm PPSL}(2,\mathbb{Z})$ defined in Rem.\ref{rem:PPSL2Z}, and takes all but finitely many $\sharp$-Farey ideal arcs (Def.\ref{def:diamond_Farey}) to $\sharp$-Farey ideal arcs.

The {\em asymptotically $\sharp$-rigid mapping class group of $\mathbb{D}^\sharp$} is the group of all asymptotically $\sharp$-rigid homeomorphism classes of $\mathbb{D}^\sharp$, and is denoted by $M^\sharp$.
\end{definition}
Like in the case of $Pt$ and $M$, the group $M^\sharp$ naturally acts on objects of $Pt^\sharp$ and $Pt^\sharp_{\rm dot}$, and we get an analog of Prop.\ref{prop:M-actions}:
\begin{proposition}
\label{prop:M_sharp-actions}
The natural $M^\sharp$-action on the objects of $Pt^\sharp$ is free and transitive, and the natural $M^\sharp$-action on the objects of $Pt^\sharp_{\rm dot}$ is free. \qed
\end{proposition}
We also have an analog of Cor.\ref{cor:T_acts_freely}, namely, $T^\sharp$ acts freely transitively on the objects of $Pt^\sharp$, which together with Prop.\ref{prop:M_sharp-actions} yields a set bijection $M^\sharp \to T^\sharp$; moreover, the proof of Prop.\ref{prop:anti_isomorphism} easily applies here, telling us that this map is an anti-isomorphism of groups. So $T^\sharp$ can be viewed as a combinatorial guise of the asymptotically $\sharp$-rigid mapping class group of $\mathbb{D}^\sharp$. As done in \S\ref{subsec:mcal_F}, the functor
\begin{align*}
\mcal{F}^\sharp : Pt^\sharp \to Pt^\sharp_{\rm dot}
\end{align*}
is naturally and essentially uniquely determined, if we require that it is $M^\sharp$-equivariant on the sets of objects. The `standard objects' of $Pt^\sharp$ and $Pt^\sharp_{\rm dot}$ can be defined in a similar way as those of $Pt$ and $Pt_{\rm dot}$ using the $\sharp$-Farey ideal arcs (Def.\ref{def:diamond_Farey}). Imposing the initial condition for $\mcal{F}^\sharp$ using these standard objects yields a concrete description of $\mcal{F}^\sharp$, precisely as in Def.\ref{prop:concrete_description_of_F}, while we require that $\mcal{F}^\sharp$ preserves the underlying $\sharp$-tessellation. From this functor $\mcal{F}^\sharp : Pt^\sharp \to Pt^\sharp_{\rm dot}$ we can construct an injective group homomorphism
$$
{\bf F}^\sharp : T^\sharp \to K^\sharp
$$
in a similar way as in Prop.\ref{prop:bold_F}. Equivalently, we can describe the construction of ${\bf F}$ by
\begin{align}
\label{eq:bold_F_sharp_condition}
\mcal{F}^\sharp( g. \tau_{\rm mark}^\sharp ) = ({\bf F}^\sharp (g)). (\mcal{F}^\sharp(\tau_{\rm mark}^\sharp)), \quad \mbox{for any object $\tau_{\rm mark}^\sharp$ of $Pt^\sharp$ and any $g\in T^\sharp$,}
\end{align}
as in Lem.\ref{lem:alternative_construction_of_bold_F}.
\begin{proposition}
A unique map ${\bf F}^\sharp: T^\sharp \to K^\sharp$ satisfying \eqref{eq:bold_F_sharp_condition} is given by
\begin{align}
\label{eq:bold_F_sharp_formula}
{\bf F}^\sharp ~ : \quad
\alpha^\sharp \mapsto A_{[-1]}^\sharp T_{[-1][1]}^{\sharp \, -1} A_{[1]}^\sharp P_{\gamma_\alpha}^\sharp, \quad
\beta^\sharp \mapsto A_{[-1]}^\sharp P_{\gamma_\beta}^\sharp,
\quad
\sigma^\sharp \mapsto A^\sharp_{[-1]} \sigma^\sharp_{[-1][1]} A^{\sharp \, -1}_{[-1]},
\end{align}
where $\gamma_\alpha$, $\gamma_\beta$ are as in Def.\ref{def:P_gamma_alpha_beta}, and is an injective group homomorphism. \qed
\end{proposition}
We omit the proofs and details, as similar arguments as in 
the non-punctured case work.
One can now descend this map ${\bf F}^\sharp : T^\sharp \to K^\sharp$ to the relative abelianizations of $T^\sharp$ and $K^\sharp$:
\begin{proposition}
\label{prop:bol_F_sharp_ab}
There exists a unique group homomorphism ${\bf F}^\sharp_{\rm ab} : T^\sharp_{\rm ab} \to K^\sharp_{\rm ab}$ making the following diagram to commute:
\begin{align}
\label{eq:bold_F_sharp_diagram}
\begin{array}{l}
\xymatrix@C+2pc{
T^\sharp \ar[r]^{ {\bf F}^\sharp } \ar[d] 
& K^\sharp \ar[d] \\
T^\sharp_{\rm ab} \ar@{.>}[r]^{{\bf F}^\sharp_{\rm ab}} & K^\sharp_{\rm ab}
}
\end{array}
\end{align}
where the vertical arrows are the relative abelianization homomorphisms, that is, projections $T^\sharp \to T^\sharp/[B_\infty,B_\infty] = T^\sharp_{\rm ab}$ and $K^\sharp \to K^\sharp/[B_\infty,B_\infty] = K^\sharp_{\rm ab}$. It is given by
\begin{align}
\label{eq:bold_F_sharp_ab}
{\bf F}^\sharp_{\rm ab} : T^\sharp_{\rm ab} \to K^\sharp_{\rm ab}: \quad \til{\alpha}^\sharp \mapsto \til{A}_{[-1]}^\sharp \til{T}_{[-1][1]}^{\sharp \, -1} \til{A}_{[1]}^\sharp \til{P}_{\gamma_\alpha}^\sharp, \qquad \til{\beta}^\sharp \mapsto \til{A}_{[-1]}^\sharp \til{P}_{\gamma_\beta}^\sharp,
\end{align}
where $\til{\alpha}^\sharp$, $\til{\beta}^\sharp$, $\til{A}^\sharp_{[j]}$, $\til{T}^\sharp_{[j][k]}$, $\til{P}^\sharp_\gamma$ are as defined in Def.\ref{def:relative_abelianization_of_T_diamond} and \eqref{eq:K_sharp_to_K_sharp_ab}, and $\gamma_\alpha$, $\gamma_\beta$ as in Def.\ref{def:P_gamma_alpha_beta}. Furthermore, ${\bf F}^\sharp_{\rm ab}$ is injective.
\end{proposition}

\begin{proof}
We first observe that the embeddings $B_\infty \to T^\sharp$ and $B^\infty \to K^\sharp$ are induced by the $M^\sharp$-action of $B_\infty$ which can naturally be viewed as a subgroup of $M^\sharp$. We denote the images of both of these embeddings as $B_\infty$ here, by abuse of notation; then we see that the restriction of ${\bf F}^\sharp : T^\sharp \to K^\sharp$ to these subgroups $B_\infty$ is the identity. The composition of ${\bf F}^\sharp$ and the relative abelianization map $K^\sharp \to K^\sharp_{\rm ab} = K^\sharp / [B_\infty, B_\infty]$ yields a map ${\bf F}^\sharp_0 : T^\sharp \to K^\sharp_{\rm ab}$.
Since $[B_\infty, B_\infty]$ is in the kernel of ${\bf F}^\sharp_0$, the map ${\bf F}^\sharp_0$ factors through the relative abelianization map $T^\sharp \to T^\sharp_{\rm ab} = T^\sharp / [B_\infty, B_\infty]$
\begin{align*}
\xymatrix@C+2pc{
T^\sharp \ar[r]^{ {\bf F}^\sharp } \ar[d] \ar[rd]^{{\bf F}^\sharp_0} & K^\sharp \ar[d] \\
T^\sharp_{\rm ab} \ar@{.>}[r]_{{\bf F}^\sharp_{\rm ab}} & K^\sharp_{\rm ab}
}
\end{align*}
hence yielding a unique group homomorphism ${\bf F}^\sharp_{\rm ab} : T^\sharp_{\rm ab} \to K^\sharp_{\rm ab}$ making \eqref{eq:bold_F_sharp_diagram} to commute. The formula \eqref{eq:bold_F_sharp_ab} of ${\bf F}^\sharp_{\rm ab}$ comes from the formulas of ${\bf F}^\sharp$ \eqref{eq:bold_F_sharp_formula} and the two relative abelianization maps as in Def.\ref{def:relative_abelianization_of_T_diamond} and \eqref{eq:K_sharp_to_K_sharp_ab}. One can easily see from the presentation of $T^\sharp_{\rm ab}$ in Prop.\ref{prop:FuKa_T_sharp_ab_presentation} that $T^\sharp_{\rm ab}$ is generated by $\til{\alpha}^\sharp$ and $\til{\beta}^\sharp$, so \eqref{eq:bold_F_sharp_ab} is enough to describe ${\bf F}^\sharp_{\rm ab}$.

\vs

Suppose $x \in \ker {\bf F}^\sharp_{\rm ab} \subset T^\sharp_{\rm ab}$. Choose any of its lift $X$ in $T^\sharp$. Then ${\bf F}^\sharp(X) \in K^\sharp$ projects to $1\in K^\sharp_{\rm ab}$ by the commutativity of the diagram \eqref{eq:bold_F_sharp_diagram}, hence ${\bf F}^\sharp(X) \in [B_\infty, B_\infty] \subset K^\sharp$. By an earlier observation about $B_\infty$, we have $X \in [B_\infty, B_\infty] \subset T^\sharp$, and therefore its projection $x$ in $T^\sharp_{\rm ab}$ is the identity element. Hence ${\bf F}^\sharp_{\rm ab}$ is injective.
\end{proof}

\subsection{Identification of $\wh{T}^{\rm Kash}$ with $T^\sharp_{\rm ab}$}

Our goal is to prove Prop.\ref{prop:T_sharp_ab}, that is, to construct an isomorphism between the two central extensions $\wh{T}^{\rm Kash}$ and $T^\sharp_{\rm ab}$ of the Ptolemy-Thompson group $T$ by $\mathbb{Z}$. 
Recall from Def.\ref{def:Ptolemy-Thompson_group} that $T$ is presented with generators and relations, that is, of the form $T \cong F_{\rm mark}/R_{\rm mark}$, where $F_{\rm mark}$ is the free group of generators $\alpha$, $\beta$ and $R_{\rm mark}$ is the normal subgroup of $F_{\rm mark}$ of relations of $T$. By applying the procedure in \S\ref{subsec:minimal_central_extensions} to Kashaev's almost $T$-homomorphism (Def.\ref{def:almost_G-homomorphism}) $\rho^{\rm Kash} : F_{\rm mark} \to {\rm GL}(\mathscr{M})$ \eqref{eq:rho_Kash}, we obtained in \S\ref{sec:dilogarithmic_central_extensions_of_T} the central extension $\wh{T}^{\rm Kash}$ of $T$ by $\mathbb{Z}$. On the other hand, $T^\sharp_{\rm ab}$ is obtained as the relative abelianization of the extension $T^\sharp$ of $T$ by the infinite braid group $B_\infty$.

\vs

For the central extension $T^\sharp_{\rm ab}$ of $T$, we use 
the following tautological almost $T$-homomorphism, in the sense of Def.\ref{def:tautological_almost_group_homomorphisms}:
\begin{align}
\label{eq:first_tautology_of_F_mark}
F_{\rm mark} \to T^\sharp_{\rm ab} ~ : \quad \alpha \mapsto \til{\alpha}^\sharp, \quad \beta \mapsto \til{\beta}^\sharp,
\end{align}
where $\til{\alpha}^\sharp$ and $\til{\beta}^\sharp$ are as in Def.\ref{def:relative_abelianization_of_T_diamond}. 
By Prop.\ref{prop:tautological_leads_to_itself}, this tautological almost $T$-homomorphism \eqref{eq:first_tautology_of_F_mark} yields the central extension $T^\sharp_{\rm ab}$ by the procedure in \S\ref{subsec:minimal_central_extensions}. Since equivalent almost $T$-homomorphisms (in the sense of Def.\ref{def:equivalent_almost_group_homomorphisms}) yield isomorphic central extensions of $T$ (Prop.\ref{prop:equivalent_almost_group_homomorphisms}), it suffices to prove that the two almost $T$-homomorphisms $\rho^{\rm Kash} : F_{\rm mark} \to {\rm GL}(\mathscr{M})$ and $F_{\rm mark} \to T^\sharp_{\rm ab}$ \eqref{eq:first_tautology_of_F_mark} are equivalent.

\vs

This will be done in two steps, and what plays the role of a bridge between the two central extensions of $T$ is the central extension $K^\sharp_{\rm ab}$ of the Kashaev group $K$ studied in \S\ref{subsec:K_sharp}, which is the relative abelianzation of the extension $K^\sharp$ of $K$ obtained by introducing the $\sharp$-punctures. Analogously, we use the following tautological $K$-homomorphism
\begin{align}
\label{eq:F_dot_to_K_sharp_ab}
F_{\rm dot} \to K^\sharp_{\rm ab} ~ : \quad A_{[j]} \mapsto \til{A}_{[j]}^\sharp, \quad T_{[j]} \mapsto \til{T}_{[j][k]}^\sharp, \quad P_\gamma \mapsto \til{P}_\gamma^\sharp,
\end{align}
where $K = F_{\rm dot}/R_{\rm dot}$ and $F_{\rm dot}$ is the free group generated by $A_{[j]}$, $T_{[j][k]}$, $P_\gamma$ for $j,k\in \mathbb{Q}^\times$ ($j\neq k$) and $\mathbb{Q}^\times$-permutations $\gamma$ and $R_{\rm dot}$ is the normal subgroup of $F_{\rm dot}$ of relations of $K$. This tautological $K$-homomorphism \eqref{eq:F_dot_to_K_sharp_ab} yields the central extension $K^\sharp_{\rm ab}$ of $K$  by the procedure in \S\ref{subsec:minimal_central_extensions} (Prop.\ref{prop:tautological_leads_to_itself}).

\vs

The two-step strategy can be roughly sketched as
\begin{align}
\label{eq:rough_strategy}
(\eqref{eq:first_tautology_of_F_mark}: F_{\rm mark} \to T^\sharp_{\rm ab})  ~\sim ~(\eqref{eq:F_dot_to_K_sharp_ab}: F_{\rm dot} \to K^\sharp_{\rm ab})  ~\sim  ~(\eqref{eq:rho_F_dot_to_GL_mathscr_M} \, \rho_{\rm dot}: F_{\rm dot} \to GL(\mathscr{M})),
\end{align}
where $\sim$ should be understood {\em only heuristically}. The first $\sim$ in \eqref{eq:rough_strategy} will `hold' because of the identification of the subgroups $B_\infty$ of $T^\sharp$ and $K^\sharp$, and the second $\sim$ is by inspection of the presentation of $K^\sharp_{\rm ab}$ and the lifted relations satisfied by the images of the generators of $F_{\rm dot}$ under $\rho_{\rm dot}: F_{\rm dot} \to {\rm GL}(\mathscr{M})$. To be more precise, the latter two almost $K$-homomorphisms in \eqref{eq:rough_strategy} should be pre-composed with $F_{\rm mark} \to F_{\rm dot}$ \eqref{eq:bold_F_on_free_group}, and so the two equivalences of almost $T$-homomorphisms that we shall actually prove are:
\begin{align}
\label{eq:actual_strategy}
(\eqref{eq:first_tautology_of_F_mark} : F_{\rm mark} \to T^\sharp_{\rm ab}) \, \stackrel{\mbox{\ding{192}}}{\simeq} \, (F_{\rm mark} \to K^\sharp_{\rm ab}) \, \stackrel{\mbox{\ding{193}}}{\simeq} \, (\eqref{eq:rho_Kash} \, \rho^{\rm Kash}: F_{\rm mark} \to {\rm GL}(\mathscr{M})).
\end{align}
The map $F_{\rm mark} \to F_{\rm dot}$ \eqref{eq:bold_F_on_free_group} that we pre-composed above is the injective group homomorphism
\begin{align}
\label{eq:F_mark_to_F_dot}
F_{\rm mark} \to F_{\rm dot} ~ : \quad \alpha \mapsto A_{[-1]} T_{[-1][1]}^{-1} A_{[1]} P_{\gamma_\alpha}, \quad \beta \mapsto A_{[-1]} P_{\gamma_\beta}
\end{align}
coming from the formula \eqref{eq:bold_F_alpha_beta} of the map ${\bf F}: T \to K$ obtained in Prop.\ref{prop:bold_F} (Prop.\ref{prop:our_bold_F}). Therefore, the map $F_{\rm mark} \to K^\sharp_{\rm ab}$ appearing in the middle of \eqref{eq:actual_strategy} is given by
\begin{align}
\label{eq:second_tautology_of_F_mark}
F_{\rm mark} \to K^\sharp_{\rm ab} ~ : \quad \alpha \mapsto \til{A}_{[-1]}^\sharp \til{T}_{[-1][1]}^{\sharp \, -1} \til{A}_{[1]}^\sharp \til{P}_{\gamma_\alpha}^\sharp, \quad \beta \mapsto \til{A}_{[-1]}^\sharp \til{P}_{\gamma_\beta}^\sharp.
\end{align}
For completeness, we recall that the third almost $T$-homomorphism $\rho^{\rm Kash}: F_{\rm mark} \to {\rm GL}(\mathscr{M})$ appearing in \eqref{eq:actual_strategy} is given as in \eqref{eq:rho_Kash_images}:
\begin{align}
\nonumber
\rho^{\rm Kash}: F_{\rm mark} \to {\rm GL}(\mathscr{M}) ~ : \quad \alpha \mapsto  \wh{\alpha}, \quad \beta \mapsto \wh{\beta}.
\end{align}
The latter two maps $F_{\rm mark} \to K^\sharp_{\rm ab}$ and $\rho^{\rm Kash} : F_{\rm mark} \to {\rm GL}(\mathscr{M})$ of \eqref{eq:actual_strategy} are indeed almost $T$-homomorphisms by the first part of Lem.\ref{lem:compositions_of_equivalent_almost_group_homomorphisms}, because they are obtained by pre-composing the latter two almost $K$-homomorphisms in \eqref{eq:rough_strategy} with the map $F_{\rm mark} \to F_{\rm dot}$ \eqref{eq:F_mark_to_F_dot}, which satisfies the condition of Lem.\ref{lem:compositions_of_equivalent_almost_group_homomorphisms} because \eqref{eq:F_mark_to_F_dot} takes $R_{\rm mark}$ to $R_{\rm dot}$ since it induces a well-defined group homomorphism $T \to K$ \eqref{eq:bold_F_alpha_beta}.

\vs

We first prove the equivalence \ding{192} of \eqref{eq:actual_strategy}, using our knowledge about the relationship between $T^\sharp_{\rm ab}$ and $K^\sharp_{\rm ab}$ studied in \S\ref{subsec:bolf_F_sharp}.
\begin{proposition}
\label{prop:ding_192}
One has the following equivalence of the almost $T$-homomorphisms
\begin{align*}
(\eqref{eq:first_tautology_of_F_mark}: F_{\rm mark} \to T^\sharp_{\rm ab}) \simeq (\eqref{eq:second_tautology_of_F_mark} : F_{\rm mark} \to K^\sharp_{\rm ab}),
\end{align*}
via the isomorphism ${\bf F}^\sharp_{\rm ab} : T^\sharp_{\rm ab} \to {\bf F}^\sharp_{\rm ab}(T^\sharp_{\rm ab}) \subset K^\sharp_{\rm ab}$ \eqref{eq:bold_F_sharp_ab}, in the sense of Def.\ref{def:equivalent_almost_group_homomorphisms}.
\end{proposition}

\begin{proof}

By looking at the formulas \eqref{eq:F_mark_to_F_dot}, \eqref{eq:first_tautology_of_F_mark}, \eqref{eq:F_dot_to_K_sharp_ab}, and \eqref{eq:bold_F_sharp_ab}, and since \eqref{eq:second_tautology_of_F_mark} was defined to be the composition of \eqref{eq:F_mark_to_F_dot} and \eqref{eq:F_dot_to_K_sharp_ab}, one can see that the following diagram commutes:
\begin{align*}
\xymatrix@C+2pc{
F_{\rm mark} \ar[r]^{\eqref{eq:F_mark_to_F_dot}} \ar[d]_{\eqref{eq:first_tautology_of_F_mark}} \ar@{.>}[dr]^{\eqref{eq:second_tautology_of_F_mark}} & F_{\rm dot} \ar[d]^{\eqref{eq:F_dot_to_K_sharp_ab}} \\
T^\sharp_{\rm ab} \ar@{^{(}->}[r]_{{\bf F}^\sharp_{\rm ab} } & K^\sharp_{\rm ab} 
}
\end{align*}
As the bottom map ${\bf F}^\sharp_{\rm ab}$ \eqref{eq:bold_F_sharp_ab} is injective (Prop.\ref{prop:bol_F_sharp_ab}), we get the desired result by Lem.\ref{lem:compositions_of_equivalent_almost_group_homomorphisms}.
\end{proof}

As mentioned earlier, the second equivalence \ding{193} of \eqref{eq:actual_strategy} is just by inspection of the relations of Kashaev's operators and the presentation of the $K^\sharp_{\rm ab}$:
\begin{proposition}
\label{prop:ding_193}
One has the following equivalence of the almost $T$-homomorphisms
\begin{align}
\label{eq:equiv3}
(\eqref{eq:second_tautology_of_F_mark} : F_{\rm mark} \to K^\sharp_{\rm ab}) \simeq (\eqref{eq:rho_Kash} \, \rho^{\rm Kash} : F_{\rm mark} \to {\rm GL}(\mathscr{M})),
\end{align}
via the group homomorphism
\begin{align}
\label{eq:K_sharp_ab_to_GL_mathscr_M}
K^\sharp_{\rm ab} \to {\rm GL}(\mathscr{M}) ~: \quad 
\til{A}_{[j]}^\sharp \mapsto {\bf A}_{[j]}, \,\,\,~ ~
\til{T}_{[j][k]}^\sharp \mapsto {\bf T}_{[j][k]}, \,\,\,~ ~
\til{P}_\gamma^\sharp \mapsto {\bf P}_\gamma, \,\,\,~ ~
z \mapsto \zeta^{-1}.
\end{align}
in the sense of Def.\ref{def:equivalent_almost_group_homomorphisms}, where ${\bf A}_{[j]}$, ${\bf T}_{[j][k]}$, ${\bf P}_\gamma$, $\zeta$ are as in \eqref{eq:rho_A_j}, \eqref{eq:rho_T_jk}, \eqref{eq:rho_P_gamma}, \eqref{eq:zeta}.
\end{proposition}

\begin{proof}
By inspection of the equations appearing in Propositions \ref{prop:lifted_Kashaev_relations} and \ref{prop:K_sharp_ab_presentation}, the tautological almost $K$-homomorphism $F_{\rm dot} \to K^\sharp_{\rm ab}$ \eqref{eq:F_dot_to_K_sharp_ab} is equivalent to the almost $K$-homomorphism $\rho : F_{\rm dot} \to {\rm GL}(\mathscr{M})$ \eqref{eq:rho_F_dot_to_GL_mathscr_M}, i.e. we have the equivalence
\begin{align}
\label{eq:equiv2}
(\eqref{eq:F_dot_to_K_sharp_ab}: F_{\rm dot} \to K^\sharp_{\rm ab}) \simeq (\eqref{eq:rho_F_dot_to_GL_mathscr_M} \, \rho : F_{\rm dot} \to {\rm GL}(\mathscr{M}))
\end{align}
via the group homomorphism $K^\sharp_{\rm ab} \to {\rm GL}(\mathscr{M})$ \eqref{eq:K_sharp_ab_to_GL_mathscr_M}. By pre-composing this equivalence \eqref{eq:equiv2} with the group homomorphism $F_{\rm mark} \to F_{\rm dot}$ \eqref{eq:F_mark_to_F_dot}, we get the desired result \eqref{eq:equiv3}, by the first statement of Lem.\ref{lem:compositions_of_equivalent_almost_group_homomorphisms}; we already saw that \eqref{eq:F_mark_to_F_dot} satisfies the condition of the Lem.\ref{lem:compositions_of_equivalent_almost_group_homomorphisms} because it sends $R_{\rm mark}$ to $R_{\rm dot}$.
\end{proof}

Since the equivalence of almost group homomorphisms is an equivalence relation (Prop.\ref{prop:equivalent_almost_group_homomorphisms}), from  Propositions \ref{prop:ding_192} and \ref{prop:ding_193} (i.e. \ding{192} and \ding{193} of \eqref{eq:actual_strategy}) we get:
\begin{corollary}
\label{cor:equivalence_of_two_T-homomorhisms}
One has the following equivalence of the almost $T$-homomorphisms
\begin{align*}
(\eqref{eq:first_tautology_of_F_mark}: F_{\rm mark} \to T^\sharp_{\rm ab}) \simeq
(\eqref{eq:rho_Kash} \, \rho^{\rm Kash} : F_{\rm mark} \to {\rm GL}(\mathscr{M})),
\end{align*}
via the group homomorphism
\begin{align}
\nonumber
T^\sharp_{\rm ab} \to {\rm GL}(\mathscr{M}) :\quad
\til{\alpha}^\sharp \mapsto 
\widehat{\alpha}, \qquad
\til{\beta}^\sharp \mapsto 
\widehat{\beta}
\end{align}
which is obtained as the composition of \eqref{eq:bold_F_sharp_ab} and \eqref{eq:K_sharp_ab_to_GL_mathscr_M}. \qed 
\end{corollary}

Let us wrap up the results. Cor.\ref{cor:equivalence_of_two_T-homomorhisms} gives the equivalence of the two almost $T$-homomorphisms, $F_{\rm mark} \to T^\sharp_{\rm ab}$ \eqref{eq:first_tautology_of_F_mark} and $\rho^{\rm Kash}: F_{\rm mark} \to {\rm GL}(\mathscr{M})$ \eqref{eq:rho_Kash}. 
The first one is the tautological almost $T$-homomorphism, hence yields by the procedure in \S\ref{subsec:minimal_central_extensions} the central extension $T^\sharp_{\rm ab}$ of $T$ (Prop.\ref{prop:tautological_leads_to_itself}). The second one yields the central extension $\wh{T}^{\rm Kash}$ of $T$ by the procedure in \S\ref{subsec:minimal_central_extensions}. Thus, from Prop.\ref{prop:equivalent_almost_group_homomorphisms} we can deduce the following group isomorphism
$$
\wh{T}^{\rm Kash} \stackrel{\sim}{\longrightarrow} T^\sharp_{\rm ab} ~ : \quad \ol{\alpha} \longmapsto \til{\alpha}^\sharp, \quad \ol{\beta} \longmapsto \til{\beta}^\sharp,
$$
where $\ol{\alpha}$ and $\ol{\beta}$ are generators of $\wh{T}^{\rm Kash}$ in the sense of its presentation \eqref{eq:wh_T_Kash_presentation}, and $\til{\alpha}^\sharp$ and $\til{\beta}^\sharp$ are as in Def.\ref{def:relative_abelianization_of_T_diamond}. This proves the desired Prop.\ref{prop:T_sharp_ab}, with an explicit isomorphism.


\begin{thebibliography}{FuKas14}

\bibitem[Ba01]{B} E. W. Barnes, {\it Theory of the double gamma function}, Phil. Trans. R. Soc. Ser. A {\bf 196} (1901) 265--388.

\bibitem[BeFu04]{BeFu} P. Bellingeri and L. Funar, {\it Braids on surfaces and finite type invariants}, C. R. Math. Acad. Sci. Paris {\bf 338} no. 2 (2004) 157--162.

\bibitem[CFo99]{FC} L. Chekhov and V. V. Fock, {\it A quantum Teichm\"{u}ller space}, Theor. Math. Phys. {\bf 120} (1999) 511--528.

\bibitem[Fa95]{F} L. D. Faddeev, {\it Discrete Heisenberg-Weyl group and modular group}, Lett. Math. Phys. {\bf 34} (1995) 249--254.

\bibitem[FaKas94]{FK} L. D. Faddeev and R. M. Kashaev, {\it Quantum dilogarithm}, Modern Phys. Lett. {\bf A9} (1994) 427--434.

\bibitem[Fo97]{Fo} V. V. Fock, {\it Dual Teichm\"{u}ller spaces}, arXiv:dg-ga/9702018.

\bibitem[FoG06]{FG06} V. V. Fock and A. B. Goncharov, {\em Moduli spaces of local systems and higher Teichm\"uller theory}, Publ. Math. Inst. Hautes \'Etudes Sci. 103 (2006) 1--211, [math/0311149v4]


\bibitem[FoG09]{FG} V. V. Fock and A. B. Goncharov, {\it The quantum dilogarithm and representations of the quantum cluster varieties}, Invent. Math. {\bf 175} (2009) 223--286.


\bibitem[FrKi12]{FrKi} I. B. Frenkel and H. Kim, {\it Quantum Teichm\"uller space from the quantum plane}, Duke Math. J. {\bf 161} no. 2 (2012) 305--366.


\bibitem[FuKap08]{FuKa2} L. Funar and C. Kapoudjian, {\it The braided Ptolemy-Thompson group is finitely presented}, Geom. Topol. {\bf 12} (2008) 475--530.

\bibitem[FuKapS]{FuKaS} L. Funar, C. Kapoudjian, and V. Sergiescu, ``Asymptotically rigid mapping class groups and Thompson's groups'' in {\it Handbook of Teichm\"uller theory} Vol. III, IRMA Lect. Math. Theor. Phys., {\bf 17}, Eur. Math. Soc., Z\"urich, 2012, pp. 595--664. Also arXiv: 1105.0559.


\bibitem[FuKas14]{FuKas} L. Funar and R. M. Kashaev, {\it Centrally extended mapping class groups from quantum Teichm\"uller theory}, Adv. Math. {\bf 252} (2014) 260--291. 

\bibitem[FuS10]{FuS} L. Funar and V. Sergiescu, {\it Central extensions of the Ptolemy-Thompson group and quantized Teichm\"uller theory}, J. Topol. {\bf 3} (2010) 29--62.

\bibitem[GhS87]{GS} E. Ghys and V. Sergiescu, {\it Sur un groupe remarquable de diff\'{e}omorphismes du cercle}, Comment. Math. Helv. {\bf 62} (1987) 185--239.

\bibitem[G07]{G} A. B. Goncharov, ``Pentagon relation for the quantum dilogarithm and quantized $\mcal{M}_{0,5}$'' in {\it Geometry and Dynamics of Groups and Spaces} (Special volume dedicated to the memory of Alexander Reznikov). Progr. Math., vol. 265, pp. 316--329. Birkh\"auser, Basel (2007) (arXiv:math.QA/0706405)

\bibitem[GuLi09]{GuLi} R. Guo and X. Liu, {\it Quantum Teichm\"{u}ller space and Kashaev algebra}, Algebr. Geom. Topol. {\bf 9} (2009) 1791--1824.

\bibitem[I97]{I} M. Imbert, ``Sur l'isomorphisme du groupe de Richard Thompson avec le groupe de Ptol\'em\'ee'' in {\it Geometric Galois Actions, 2. The inverse Galois problem, moduli spaces, and mapping class groups}, LMS Lecture Notes {\bf 243}, Cambridge Univ. Press., Cambridge, 1997, pp. 313--324.


\bibitem[Kas98]{Kash98} R. M. Kashaev, {\it Quantization of Teichm\"{u}ller spaces and the quantum dilogarithm}, Lett. Math. Phys. {\bf 43} (1998) 105--115.

\bibitem[Kas00]{Kash00} R. M. Kashaev, ``On the spectrum of Dehn twists in quantum Teichm\"{u}ller theory" in {\it Physics and Combinatorics (Nagoya, 2000)}, World Sci. Publ., River Edge, NJ, 2001, pp. 63--81.

\bibitem[Ki13]{Ki} H. Kim, {\it Quantum Teichm\"uller space and universal modular groupoid}, Ph.D. thesis, Yale University, New Haven, CT, 2013. (see arXiv:1211.4300v1)

\bibitem[Ki14]{Ki14} H. Kim, {\it Ratio coordinates for higher Teichm\"uller spaces}, to appear in Math. Z., arXiv:1407.3074

\bibitem[Ki16]{K16} H. Kim, {\it Phase constants in the Fock-Goncharov quantization of cluster varieties}, long version is arXiv:1602.00361, short version for journal submission is arXiv:1602.00797


\bibitem[LS97]{LSc} P. Lochak and L. Schneps, ``The universal Ptolemy-Teichm\"uller groupoid'' in {\it Geometric Galois Actions}, 2, LMS Lecture Note Ser. {\bf 243}, Cambridge Univ. Press, Cambridge, 1997, pp. 325--347.



\bibitem[P87]{Penner} R. C. Penner, {\it The decorated Teichm\"{u}ller space of punctured surfaces}, Comm. Math. Phys. {\bf 113} (1987) 299--339.

\bibitem[P93]{Penner2} R. C. Penner, {\it Universal constructions in Teichm\"{u}ller theory}, Adv. Math. {\bf 98} (1993), 143--215.

\bibitem[RSi80]{RSi80} M. Reed and B. Simon, {\it Methods of modern mathematical physics. I: Functional analysis}, Revised and Enlarged ed. (1980), Academic Press, New York--London, 1972.

\bibitem[S93]{S} V. Sergiescu, {\it Graphes planaires et pr\'esentations des groups de tresses}, Math. Z. {\bf 214} (1993) 477--490. \quad MR1245207

\bibitem[Te07]{T} J. Teschner, ``An analog of a modular functor from quantized Teichm\"uller theory'' in {\it Handbook of Teichm\"uller theory} Vol. I, IRMA Lect. Math. Theor. Phys., {\bf 11}, Eur. Math. Soc., Z\"urich, 2007, pp. 685--760. Also arXiv:0510174.

\bibitem[Th80]{Th} W. P. Thurston, {\it The geometry and topology of $3$-manifolds}, lecture notes, Princeton University, 1980, avaliable at \url{http://library.msri.org/books/gt3m} 

\bibitem[Xu14]{Xu14} B. Xu, {\it Central extension of mapping class group via Chekhov-Fock quantization}, arXiv:1410.5551


\end{thebibliography}
\end{document}